\documentclass[11 pt, oneside, a4paper]{amsart}

\usepackage{amsfonts, amsmath, amssymb, amsthm, amsopn, latexsym, graphicx, mathrsfs,  verbatim, url, stmaryrd, tikz-cd} 
\usepackage[pagebackref]{hyperref}
\usepackage[shortlabels]{enumitem}

\usepackage[all]{hypcap} 
 \usepackage{subcaption}
\usepackage[paper=a4paper]{geometry}

\newcommand{\R}{\mathbb{R}}
\newcommand{\C}{\mathbb{C}}
\newcommand{\D}{\mathbb{D}}
\newcommand{\N}{\mathbb{N}}
\newcommand{\Z}{\mathbb{Z}}
\newcommand{\Q}{\mathbb{Q}}
\newcommand{\A}{\mathbb{A}}

\renewcommand{\P}{\mathbb{P}}

\newcommand{\nA}{\mathbb{A}}

\newcommand{\nC}{\mathbb{C}}

\newcommand{\nL}{\mathbb{L}}
\newcommand{\nN}{\mathbb{N}}

\newcommand{\nR}{\mathbb{R}}

\newcommand{\nU}{\mathbb{U}}

\newcommand{\cT}{\mathcal{T}}

\newcommand{\cK}{\mathcal{K}}
\newcommand{\cL}{\mathcal{L}}
\newcommand{\cB}{\mathcal{B}}
\newcommand{\cW}{\mathcal{W}}

\newcommand{\CritB}{\mathcal{C}}		
\newcommand{\CritT}{\mathcal{T}}		

\newcommand{\norm}[1]{\left\|#1\right\|}

\newcommand{\mc}{\mathcal}

\newcommand{\eps}{\varepsilon}
\newcommand{\ord}{\mathrm{ord}}

\newcommand{\loc}{\mathrm{loc}}

\newcommand{\diam}{\mathrm{diam}}

\newcommand{\abs}[1]{{\left|{#1}\right|}}
\newcommand{\on}[1]{\operatorname{#1}}

\renewcommand{\mod}{\on{mod}}

\newcommand{\id}{\on{id}}

\newcommand{\set}[1]{{\left\{#1\right\}}}

\providecommand{\fps}[1]{[[{#1}]]}
\providecommand{\laurent}[1]{\left(\!\left({#1}\right)\!\right)}
\providecommand{\puiseux}[1]{\langle\!\langle{#1}\rangle\!\rangle}

\newcommand{\hz}{\hat{z}}
\newcommand{\hw}{\hat{w}}

\newcommand{\Puis}{k\puiseux{z}}	
\newcommand{\xg}{\zeta_g}

\DeclareMathOperator{\an}{an}	
\DeclareMathOperator{\na}{na}

\DeclareMathOperator{\spf}{Spf}	
	
\DeclareMathOperator{\capa}{cap}	
\DeclareMathOperator{\jac}{Jac}	
\DeclareMathOperator{\supp}{Supp}

\newcommand{\lcm}{\on{lcm}}

\newcommand{\fr}{\partial}

\newcommand{\rest}[1]{ \arrowvert_{#1}}

\newcommand{\unsur}[1]{\frac{1}{#1}}

\newcommand{\lrpar}[1]{\left(#1\right)}
\newcommand{\bra}[1]{\left\langle #1\right\rangle}

\newcommand{\inv}{^{-1}}

\newtheorem{introthm}{Theorem}

\newtheorem{thm}{Theorem}[section]
\newtheorem{lem}[thm]{Lemma}
\newtheorem{cor}[thm]{Corollary}
\newtheorem{prop}[thm]{Proposition}
\newtheorem{quest}[thm]{Question}

\theoremstyle{definition}
\newtheorem{rmk}[thm]{Remark}

\theoremstyle{definition}

\title[Polynomial skew products ]{Polynomial skew products with small relative degree}

\author{Romain Dujardin}
\address{Sorbonne Universit\'e, Laboratoire de probabilit\'es, statistique et mod\'elisation, UMR 8001,  4 place Jussieu, 75005 Paris, France}
\email{\href{romain.dujardin@sorbonne-universite.fr}{romain.dujardin@sorbonne-universite.fr}}

\author{Charles Favre}
\address{CNRS - Centre de Math\'ematiques Laurent Schwartz, 
	\'Ecole Polytechnique, 
	91128 Palaiseau Cedex, France}
\email{\href{charles.favre@polytechnique.edu}{charles.favre@polytechnique.edu}}

\author{Matteo Ruggiero}
\address{Université Paris Cité, Sorbonne Université, CNRS, Institut de Mathématiques de Jussieu-Paris Rive Gauche, F-75013 Paris, France}
\email{\href{mailto:matteo.ruggiero@imj-prg.fr}{matteo.ruggiero@imj-prg.fr}}

\date{\today}

\begin{document}

\begin{abstract}
We investigate the local dynamics of a proper superattracting holomorphic germ $f$ in $(\C^2,0)$
possessing  a totally invariant line $L$ such that $f^*L = d L$ with $d\ge 2$, and such that 
$f|_L$ has a superattracting fixed point at $0$ of order $2 \le c < d$. 
We prove that any such map is formally conjugated to a skew product
of the form $(z^d, P(z,w))$, where $P \in \nC\fps{z}[w]$ is polynomial in $w$ of degree $c$, hence it  
 induces a natural  dynamics on the Berkovich affine
line over $\C\laurent{z}$. Such  
 non-Archimedean  skew products were  recently studied by Birkett and Nie-Zhao. 

On the non-Archimedean side, we focus on the restriction  of the dynamics on the Berkovich
open unit ball  (which naturally contains all irreducible analytic germs at the origin). 
We exhibit an invariant     compact set $\cK$ outside of which all points tend
to $L$, and which supports a natural ergodic invariant measure.

By a careful analysis of local intersection numbers, we prove that
the growth of multiplicity of iterated curves   is controlled by the recurrence properties of the critical set. 
In particular,  when no critical branch of $f$ belongs to $\cK$,
any point in $\cK$ corresponds to  a curve of uniformly bounded multiplicity at $0$.
 
We then  return to the complex picture and show the existence of
an invariant  pluripolar positive closed  $(1,1)$-current $T$, outside of which all orbits 
converge  to $0$ at super-exponential speed $c$. 
Under the same assumption on the  critical branches  as above, 
we prove  that $T$ admits a geometric representation as 
  an average of currents of integration over the curves in $\cK$,
with respect to the natural invariant measure. 
In particular,  $T$ is uniformly laminar outside the origin.
\end{abstract} 

 \maketitle
 
\vspace{-1em}
\setcounter{tocdepth}{1}
\tableofcontents

\newpage

\section*{Introduction}

\subsection{A class of superattracting germs}
In this article we investigate  the local dynamics of a class of 
superattracting points in $\C^2$. 
 It has been  known for a long time  that, contrary to the
one-dimensional  case, the local dynamics at superattracting points in higher dimension  
gives rise to subtle and varied  phenomena (see, e.g., \cite{hubbard-papadopol:supfixpnt, favre:rigidgerms, favre-jonsson:eigenval, gignac:conjarnold, gignac-ruggiero:attractionrates,zbMATH06256927,zbMATH07068928,zbMATH06236906}). 

Our  setting is 
the  following: assume that $f$ is a germ of holomorphic map at 
$0\in \C^2$ with $f(0) = 0$, and   that there are integers $2\leq c<d$ satisfying: 
(i) the line $\set{z=0}$  is totally invariant and as divisors (or currents) we have  
$f^*\set{z=0}= d\set{z=0}$; and 
(ii) the restriction of $f$ to $\set{z=0}$ admits a superattracting point 
of order $c$. Then, in  local holomorphic coordinates, $f$ can be expressed as
\begin{equation}\label{eq:f_intro}
f(z,w) = \big(z^d, w^c+ zh(z,w)\big), \text{ with } h(0,0)=0.
\end{equation}
This class of mappings arises from our attempts at dealing with  the  missing case of 
the dynamical Manin-Mumford conjecture for regular polynomial mappings
  in~\cite{dujardin-favre-ruggiero:DMM}. 

When $p = (z,w)$ is close to 0, $f^n(p)$ converges super-exponentially fast to 0, and there are two distinct regimes: the majority of orbits converge tangentially to $\set{z=0}$, eventually landing in a region of the form $\set{\abs{z}< C\abs{w}^c}$.  For such orbits, the 
 dynamics is dominated by that of $f\rest{\set{z=0}}$ and the 
asymptotic contraction rate 
$$
c_\infty(p) 
= \lim_{n\to\infty} \sqrt[n]{\big|\log\abs{f^n(p)}\big|}
$$ 
is equal to $c$. 
There is however a thin subset $\cW$ of points  converging faster to the origin, with 
asymptotic rate $d$. The set $\cW$ may thus be viewed as a kind of ``strong (super-)stable manifold'' of the 
origin. This is a closed pluripolar set with a rich and intriguing structure. More precisely, the 
  following result (which is already partly contained in~\cite{favre-jonsson:eigenval}) 
  is proven in Section~\ref{sec:complex}.

\begin{introthm}\label{thm:complex_intro}
Let $f$ be a germ of holomorphic mapping in $\C^2$ of the form~\eqref{eq:f_intro} with $2\leq c<d$. Then  for every $p = (z,w)$ close to the origin, the asymptotic contraction rate exists and equals $c$ or $d$. 

The sequence  $\lrpar{\unsur{c^n}\log\abs{f^n(z,w)}}_{n\geq 0}$ converges to a plurisubharmonic function $g $ with values in $\R\cup\set{-\infty}$,
 such that $e^{g }$ is continuous,  
$g \circ f  = cg $, and $g $ is pluriharmonic in $\set{g >-\infty}$. 

The closed  pluripolar set $\cW= \set{g  = -\infty} = \supp({\mathrm{d}}{\mathrm{d}}^cg )$ is precisely the locus where $c_\infty = d$. It is not a subvariety unless $f$ is conjugated to a product map.  
\end{introthm}

Our main purpose in this paper is to demonstrate  
  that the structure of the set $\cW$ and of the current $\supp({\mathrm{d}}{\mathrm{d}}^cg )$
 are encoded by a natural  dynamical system whose phase space is of non-Archimedean nature. 
This space is defined as a suitable completion of  the space of germs of irreducible curves 
  in $(\C^2, 0)$, and corresponds to the valuative tree, which was extensively studied in~\cite{favre-jonsson:valtree,favre-jonsson:eigenval}. 
  We found it more convenient to view this space as a ball in the Berkovich analytification of the affine line over the field of formal Laurent series $\C\laurent{z}$, which makes the connection with other relevant works more transparent 
 (these two points of view are compared in \S~\ref{sec:compare} below). 
 
 \subsection{Polynomial skew products over a non-Archimedean field}
To define our non-Archime\-dean model we work with mappings of the form 
 \begin{equation}\label{eq:cd_intro}
f(z,w) = \biggl(z^d, w^c + \sum_{j=0}^{c-1} h_j(z) w^j\biggr)
\end{equation}
where $h_j \in \C\fps{z}$ satisfy $h_j(0)=0$, 
 and 
as in~\eqref{eq:f_intro}, we assume that $2\leq c<d$. Any map $f$ as in~\eqref{eq:cd_intro} falls within the class of \emph{polynomial skew products}. 
It induces a self-map $f_\diamond$ 
of the field of (formal) Puiseux series $\C\puiseux{z}$ and of its completion $\nL$ for the $t$-adic norm, defined by the formula 
\begin{equation}\label{eq:fdiamond_intro}
f_\diamond\colon \phi(z)\longmapsto \lrpar{\phi \big({z^{1/d}}\big)}^c+ \sum_{j=0}^{c-1} h_j\big({z^{1/d}}\big)\, \lrpar{\phi\big({z^{1/d}}\big)}^j.
\end{equation}
Recall that a   Puiseux series (resp. an element of $\nL$) is a series of the form $\sum_{n=0}^\infty {a_n} z^{\beta_n}$ 
where  $(a_n)\in \C^\N$ and  $(\beta_n)$ is an increasing sequence of rational numbers with a fixed 
denominator (resp. an increasing sequence of rational numbers with $\beta_n\to \infty$). 

The map $f_\diamond$ admits a  natural continuous extension to the Berkovich analytification of 
the affine line  $\A^{1, \an}_{\nL}$ over  $\nL$, enabling us to exploit analytic and potential theoretic
tools to analyze its dynamics. It is crucial to   consider at the same time 
the induced dynamics of $f_\diamond$
on $\A^{1, \an}_{\C\laurent{z}}$. Indeed the latter space is more directly related  to geometry, 
as the classical points lying in the open unit ball correspond to irreducible formal germs in $(\C^2,0)$ 
different from $\{z=0\}$ 
(while classical points in $\A^{1, \an}_{\nL}$ should be thought of as 
parameterizations of these curves). There is a canonical surjective map $\pi\colon \A^{1, \an}_{\nL}\to\A^{1, \an}_{\C\laurent{z}}$, and 
we say that a point (in either space) is of type 1 when it corresponds to a point in $\nL$.

To make the discussion clearer, we focus on the dynamics of $f_\diamond$ on 
$\A^{1, \an}_{\C\laurent{z}}$. 
 We first observe (Theorem~\ref{thm:basic_dynamical}) that the interesting dynamics takes place in the (Berkovich) unit ball $B^{\an}(0,1)$: 
 under iteration,  a point $x\in B^{\an}(0, 1)$ either converges  to the Gauss point $\xg$
 or it remains in 
 $B^{\an}(0,\rho)$ for some $\rho<1$. We define
 \begin{equation}
\cK := \{x \in B^{\an}(0,1)\subset \A^{1,\an}_{\C\laurent{z}}, \, f^n_\diamond(x) \text{ does not converge to }\xg \}~.
\end{equation}
The ball $B^{\an}(0,1)$ is a $\R$-tree with one (non-compact) end, and the compact set $\cK$
consists of the set of points that do not reach a neighborhood of this end under iteration.  
It thus plays the role of the filled-in Julia set in our context. It will also 
turn out to be the non-Archimedean counterpart of $\cW$. In Section~\ref{sec:cantor} 
we prove the following:
 \begin{introthm}\label{thm:2B}
 Let $f$ be a map of the form~\eqref{eq:cd_intro} with $2\leq c<d$. Then 
 $\cK$   is either a Cantor set of type 1 points, or it is reduced to a single point corresponding to a smooth formal curve transverse to $\set{z=0}$, in which case $f$ is conjugated to 
 a product map. 
 
 There exists a canonical $f_\diamond$-invariant ergodic probability measure $\mu_{\na}$, with  $\mathrm{Supp}(\mu_{\na}) = \cK$ and  such that 
 for every $x\in B^{\an}(0,1)$, the sequence 
 $\frac{d^n}{c^n}( f^{n}_\diamond )^*\delta_x$ converges to $\mu_{\na}$. 
 \end{introthm}

See below Section~\ref{sec:berkovich} for basics on Berkovich theory. 
 Note that in Sections~\ref{sec:berkovich} to~\ref{sec:examples}, we 
 work in the slightly more general context of the field $k\laurent{z}$, with $k$ an arbitrary  algebraically closed field $k$ of characteristic 0. 
 
 \subsection{History and related results}
 Polynomial and rational dynamics on Berkovich spaces in dimension 1 
 is now a well-developed subject; a valuable overview is provided in the books of Baker-Rumely~\cite{baker-rumely}  and Benedetto~\cite{benedetto:dyn1NA}. Mappings of the form $f_\diamond$ are not directly covered by these references. Their dynamical properties were recently investigated by  Birkett \cite{birkett:skewproductsberkproj} and independently
by  Nie and Zhao \cite{nie-zhao:Berkdyntwistratmaps}.

In \cite{nie-zhao:Berkdyntwistratmaps}, the authors consider  \emph{twisted rational maps}, on $\P^{1,\an}_{K}$, where $K$ is an algebraically closed field of characteristic $0$, that is non-trivially valued. Such maps  are of the form $R_\tau=R \circ \tau$, where $R\in K(w)$ is a rational map, and $\tau$ is a continuous automorphism of $K$ satisfying the property $\abs{\tau(\cdot)} = \abs{\cdot}^{\lambda}$ for some $\lambda \in \nR_{>0}$. The degree of $R$ is called the \emph{relative degree}
 of $R_\tau$ (this terminology is from~\cite{birkett:skewproductsberkproj}).
Under the  assumption that $R_\tau$ has 
\emph{large relative degree} $\lambda \deg (R)  > 1$, they obtain 
equidistribution results (in the spirit of  Favre and Rivera-Letelier \cite{favre-riveraletelier}) for backward orbits. When moreover $R$ is a polynomial, the authors describe the action of $R_\tau$ on the Julia set, as a one-sided shift on $\deg(R)$ symbols.

For mappings of the form~\eqref{eq:fdiamond_intro}, if we let $\tau$ be the automorphism of $\nL$
given by $\tau(z)=z^{1/d}$,   we have $f_\diamond = \hat{f} \circ \tau$, where
$\hat{f}$ is the polynomial
\[
\hat f(w)=w^c + \sum_{j=0}^{c-1} h_j(\tau(z)) w^j~.
\]
It follows that in our situation 
the relative degree equals $\deg \big(\hat f\big) = c$ and $\lambda = 1/d$
 so that we are in the opposite case $\lambda \deg \big(\hat f\big) = c/d < 1$
of \emph{small relative degree}.

In the language of Birkett \cite{birkett:skewproductsberkproj}, the value $\lambda_\tau=1/d$ is  
 the \emph{scale factor}, which places us in the \emph{superattracting} case of that paper. 
Under this  condition, it is shown in  \cite[Corollary 3.14]{birkett:skewproductsberkproj}   that $f_\diamond$ is a contraction on the hyperbolic space, given  
 the Berkovich projective line minus the rigid points,  endowed with the hyperbolic distance. 
 This property was   investigated in the context of valuation spaces in \cite{gignac-ruggiero:locdynnoninvnormsurfsing}.
 
 It is worth mentioning  that our set $\cK$ lies in the Fatou set of $f_\diamond$ both in the sense of~\cite[Definition~5.1]{birkett:skewproductsberkproj} and~\cite[Definition~3.1]{nie-zhao:Berkdyntwistratmaps}. 
 However the set of iterates does not form a normal family near any point of $\cK$ 
 in the sense of~\cite{zbMATH06051033}. This underlines some of the  subtleties that may
 arise in the Fatou-Julia
 theory of general skew products. 

 \subsection{Curves of bounded multiplicity in $\cK$} 
 Recall that a type 1 point $x\in \A^{1,\an}_{\C\laurent{z}}$ is defined by a series $\phi$ in $\nL$. 
 When $\phi$ is a Puiseux series,  the point  $x$ corresponds to an irreducible formal germ $C(x)$, and
  to comply with the terminology of  non-Archimedean geometry, we say that it is \emph{rigid}. 
 We define the  \emph{multiplicity} $m(x)$  to be   the order of tangency of $C(x)$  with 
 $\set{z=0}$ (this terminology is consistent with Berkovich theory but it is also a little unfortunate since $m(x)$ is not the multiplicity of $C(x)$ at 0). 
When $x$ is not defined by a Puiseux series, by definition it is 
 not rigid, and we set $m(x)=\infty$.  
 The geometric interpretation of non-rigid points
is less clear (see~\cite[Proposition~6.9]{favre-jonsson:valanalplanarplurisubharfunct} 
 for an attempt using currents).  

%
%
 
In sections~\ref{sec:upper} and~\ref{sec:infinite} we show   that 
 the multiplicity of points in $\cK$ is controlled by the recurrence properties of the critical set. 
 More precisely, let $\mathcal C$ be the (finite) set consisting of all rigid points associated with the
 irreducible components of the critical locus of $f$ different from $\set{z=0}$. 
 We also introduce the restricted critical set $\mathcal C^+\subset \mathcal C$
corresponding to critical branches $c$ such that $1+J(c)$ 
 does not divide $d$, where $J(c) = \ord_c\jac_w(f)$ is the  multiplicity of $c$ 
 as a critical branch.

 Our second main theorem is  the following. Here $\omega(x)$ is the $\omega$-limit set, i.e.  
 the cluster set of the set of iterates of $x$.
 
 \begin{introthm}\label{thm:multiplicity_intro}
Let $f$ be a map of the form~\eqref{eq:cd_intro} with $2\leq c<d$. Then:
\begin{enumerate}[(i)]
\item any point in $\cK$ such that $\omega(x)\cap \mathcal C^+=\emptyset$ has finite multiplicity;
\item if  $\mathcal C^+$ is non-recurrent in the sense that $\omega(c)\cap \mathcal C^+ = \emptyset$ 
for any $c\in \mathcal C^+$, then the multiplicity is uniformly bounded on $\cK$;
\item if $\cK$ is not reduced to a point and there is a periodic  critical point in $\mathcal C^+$, 
  then $\cK$ contains non-rigid points and rigid points with arbitrary large 
multiplicity. 
\end{enumerate}
 \end{introthm}
 
 In~\cite{trucco:algJuliaBerko}, Trucco studies
 the dynamics of a polynomial mapping $P$ with coefficients in the field of 
 complex Puiseux series. His Theorem~F
 asserts  that the Julia set $J_P$ contains only Puiseux series 
 if and only if no critical point in  $J_P$  is recurrent, and in that case the multiplicity is uniformly bounded.  
Theorem~\ref{thm:multiplicity_intro} is thus 
a  generalization of his results to skew-products of small relative degree. Note that it
does not say anything in the presence of a non-periodic recurrent critical point in $\mathcal C^+$. 
The very  existence of such   examples is a non-trivial fact and remains open. 
We present explicit examples illustrating each case of the theorem in Section~\ref{sec:examples}.

The proof of  Theorem~\ref{thm:multiplicity_intro} relies on the notion of
\emph{generic multiplicity} for type 2 points of $B^{\an}(0,1)$
which we discuss in  \S~\ref{ssec:generic-mult}. The key step is to bound from above
the generic multiplicity at a point $x$ in terms of the generic multiplicity along its orbit. 
This delicate point is handled in Proposition~\ref{prop:123jac}.

 \subsection{Back to germs in $\C^2$} 
 Building on the previous results 
 we come back to the problem of describing the structure of the super-stable set $\cW$ 
 for a germ of the form~\eqref{eq:f_intro}.
 After a formal conjugacy, we can bring $f$ to the form~\eqref{eq:cd_intro} with $2\leq c<d$,  
 so that Theorem~\ref{thm:complex_intro} applies. We work under the assumption  that 
no critical branch corresponds to a point in $\cK$, or equivalently is contained  in $\cW$ as a subset of $\C^2$ (see Proposition~\ref{prop:anal-na}). Then 
by the second conclusion of 
Theorem~\ref{thm:multiplicity_intro}, any point $x\in \cK$ is associated with a germ of (a priori) formal curve $C(x)$, 
and the  multiplicity $m(x)$ is uniformly bounded by some integer $m$.  

Recall from Theorem~\ref{thm:2B}, that $\cK$ supports a canonical ergodic measure $\mu_{\na}$. 
  Our main result is the following:

\begin{introthm}\label{thm:structure_intro}
 Assume that $f$ is a holomorphic map of the form~\eqref{eq:f_intro} with $2\leq c<d$, and that
no critical branch belongs to $\cK$. 

Then there exist an integer $m\in\N^*$ and a constant $r_0>0$, such that for any $x\in \cK$ there is a Puiseux series
 $\phi_x\in\C\puiseux{z}$ which is convergent  in the disk $\D(0,r_0)$, and such that 
 $C(x)$ is parameterized by $t\mapsto (t^{m},\phi_x(t^{m}))$.

Moreover, we have the equality of currents
\begin{equation}\label{eq:trez}
 {\mathrm{d}}{\mathrm{d}}^c g = \int_\cK \frac{[C(x)]}{m(x)} \, {\mathrm{d}}\mu_{\na}(x)~,
\end{equation}
and ${\mathrm{d}}{\mathrm{d}}^c g$ is a uniformly laminar current outside the origin. 
\end{introthm}

The proof of this result is given in Section~\ref{sec:structure}, and requires several steps. 
We first make a suitable base change, that is, we take the pull-back by some branched cover $\beta(z,w)=(z^k,w)$
 in order  to reduce to the case $m=1$. 
Then we construct a suitable geometric model by exhibiting a sequence of point blow-ups over the origin $\kappa\colon X\to (\C^2,0)$, 
and a finite set of    points $p_i\in \kappa^{-1}(0)$
such that a germ of curve $C$ belongs to $\cK$ if and only if  the strict transform of  $f^n(C)$ intersects 
$\kappa^{-1}(0)$ at one of the points $p_i$ for every $n\geq 0$. 

At the non-Archimedean level, the set of points $p_i$ corresponds to 
a covering of $\cK$ by open balls. 
More precisely, for any $i$, 
the set of all germs of irreducible curves $C$ whose strict transform contains $p_i$
is an open ball $B_i$. 
Furthermore, the open cover $\set{B_i}$ can be chosen to form 
 a Markov partition for $f_\diamond$.
Introduce the oriented graph $\Gamma$ whose vertices are 
balls $B_i$ and an (oriented) edge joins   $B_i$ to $B_j$ if and only if 
$f_\diamond(B_i) \supset B_j$. 
We prove that  the subshift of finite type induced by 
$\Gamma$ is canonically conjugated to the dynamics of $f_\diamond$ on $\cK$. 

Back to the Archimedean picture, let $\widehat F \colon X\dashrightarrow X$ be the lift of $f$. 
It is a meromorphic map such that  all points $p_i$ are indeterminate.
In the oriented graph $\Gamma$ an edge joins $B_i$ to $B_j$ if and only if $\widehat F(p_i) \ni p_j$. 
Our analysis then proceeds as follows: 
pulling-back curves by $\widehat F$ gives rise to a natural \emph{graph transform} $\cL_{e}$
for every edge $e= (p_ip_j)$ of $\Gamma$. The superattracting   nature  of $f$ implies that 
$\cL_e$ sends a curve passing through $p_j$ of a definite size $r$ 
to a curve  through $p_i$ of the same size. 
Analyzing  this graph transform in local coordinates shows that it is a contraction 
on a suitable function space. 
Since any point $x\in \cK$ is determined by its itinerary
$(v_j)_{j\geq 0}$ in the cover $\set{B_i}$ and, at the same time, the sequence 
 $\cL_{e_0} \circ \cdots \circ  \cL_{e_n}(C)$ converges, where $e_j = (v_jv_{j+1})$ and $C$ is any initial data, we conclude that $C(x)$ coincides with (the projection under $\kappa$ of)
 its  limiting curve. 
 The control of the size $r$ guarantees the convergence of the Puiseux parameterizations  over a fixed neighborhood of the origin. 
Finally,  the representation~\eqref{eq:trez}   follows from the equidistribution of pull-backs both at the Archimedean and non-Archimedean levels. 

To  complete the proof, we push   this picture down  by the branched cover $\beta$. A few non-trivial arguments are required to justify  the uniform laminarity of the image current  and the formula~\eqref{eq:trez}. 
Note that the appearance of the multiplicities in~\eqref{eq:trez} is a feature of our coordinate dependent approach. 

\subsection{Related works} Theorem~\ref{thm:structure_intro} implies that $\cW$ is a 
bouquet of analytic curves,  and  a lamination with Cantor transversals outside the origin. 
Such a structure first appeared in the work of Yamagishi~\cite{zbMATH01582425} (see 
also~\cite{zbMATH05717809}) who exhibited a simple mechanism for a meromorphic map in $(\C^2,0)$  
to have  an invariant Cantor bouquet of analytic curves at an indeterminacy point. 
Gignac~\cite{gignac:conjarnold} studied in detail 
the map $(z^4, w^2-z^4)$ (which is of the form~\eqref{eq:cd_intro}) 
to build a counter-example to a question of Arnol'd. 
For this example, he proved that  all curves in $\cK$ are smooth and intersect $\set{z=0}$ transversally.  
In all these works though, there was no explicit connection to non-Archimedean dynamics.
Note also that   global variants of these constructions for polynomial mappings of $\C^2$ 
were studied in \cite{dinh-dujardin-sibony, vigny}. 

In a different direction, Kiwi~\cite[Theorem~1.6]{zbMATH05145705}   proved that the 
locus of non-renormalizable complex cubic polynomials with disconnected Julia set
is naturally foliated  by analytic curves,    corresponding to 
non-renormalizable  cubic polynomials defined over $\nL$. It could be interesting to expand on his arguments
and see whether the bifurcation current   defined in~\cite{zbMATH05375512} can be described similarly 
 as in~\eqref{eq:trez},
where the measure $\mu_{\na}$ would be replaced by a suitable bifurcation measure in the non-Archimedean 
cubic parameter space.

Finally, we  note that although the laminarity of currents is a pervasive property in holomorphic dynamics, see e.g.~\cite{bls, bedford-jonsson, dujardin:birat,dinh:laminaire,dujardin:ICM},   a non-Archimedean
interpretation of the set of plaques defining the current is not generally possible. 
The complex structure needs to degenerate to exhibit a limiting
non-Archimedean behaviour, a property which is guaranteed by the fact that we are  working locally at a superattracting fixed point. 
In the global situation which motivated this work, 
where  $f$ is the germ of a regular polynomial mapping at some superattracting point on the line at infinity, the laminar structure of the current $ {\mathrm{d}}{\mathrm{d}}^cg$ is reminiscent from  that of the Green current obtained in~\cite{bedford-jonsson}, with the dynamics of $f_\diamond$ playing the role of that of  $f\rest{L_\infty}$. 

%
%
 
 \subsection{Plan of the paper} In Section~\ref{sec:berkovich}, we introduce the necessary tools from Berkovich theory, in particular over the fields $k\laurent{z}$ and $\nL$, where $k$ is  
  an arbitrary  algebraically closed field $k$ of characteristic 0. The non-Archimedean dynamics of 
  the skew-product $f_\diamond$ acting on the Berkovich unit ball is studied in Sections~\ref{sec:dynamics_psp} to~\ref{sec:examples}. After some generalities on the topological properties of $f_\diamond$, 
  the structure of the set $\cK$ and the measure $\mu_{\na}$ is analyzed in Section~\ref{sec:cantor}, and the conclusions of Theorem~\ref{thm:2B} are contained 
Theorems~\ref{thm:invariantcantor}, \ref{thm:green}, and \ref{thm:ergo-meas}. 
 In Sections~\ref{sec:upper} and~\ref{sec:infinite}, we prove Theorem~\ref{thm:multiplicity_intro}
 as a combination of Theorems~\ref{thm:estim-mult} and~\ref{thm:infinity-mult}. In Section~\ref{sec:examples}, we study the formal conjugacy of maps of the form~\eqref{eq:f_intro} to 
 maps of the form~\eqref{eq:cd_intro} (Theorem~\ref{thm:formal}) and illustrate Theorem~\ref{thm:multiplicity_intro} by a few examples.  In the last part of the paper (Sections~\ref{sec:complex} and~\ref{sec:structure}), we turn to complex dynamics. Theorem~\ref{thm:complex_intro} follows from Theorem~\ref{thm:complex}, and Theorem~\ref{thm:structure_intro} corresponds to   Theorem~\ref{thm:structure}.
  
\section{Dynamics on the Berkovich affine line}\label{sec:berkovich}

In this section we gather the necessary material from Berkovich theory, in particular over
  the field of Laurent series. The  self maps  induced by 
  our skew products will be introduced in the 
  next section.  

\subsection{The Berkovich affine line} (See  \cite{benedetto:dyn1NA,baker-rumely,jonsson:berkovich} for a more detailed presentation.) 
Let  $(K,|\cdot|)$ be a complete non-Archimedean field. 
The Berkovich affine line $\A^{1,\an}_K$ over $K$ is 
by definition the set of all multiplicative seminorms on $K[w]$
whose restriction to $K$ coincides with $|\cdot|$, endowed with the weakest 
topology making all evaluation maps continuous. 

Given a point $x\in \A^{1,\an}_K$ and a polynomial $P\in K[w]$ we write
$|P(x)|\in\R_+$ for the value of $P$ at $x$, that is 
$x$ is a seminorm $\abs{\cdot}_x$ and $\abs{P(x)} = \abs{P}_x$. 
Note that when $P=c$ is a constant, 
then $|c(x)|=|c|$. The space  $\A^{1,\an}_K$ is locally compact. 

\subsubsection{}
Suppose that $K$ is algebraically closed. 
We denote by 
 \[\overline B(a,r)=\set{b\in K, \ \abs{b-a}\leq r}\] (resp. 
$B(a,r)=\set{b\in K, \ \abs{b-a}< r}$) the closed (resp. open) ball of center $a$ and radius $r$ in $K$. 
 Then Berkovich classified points
 $x\in \A^{1,\an}_K$ into four types: 
\begin{enumerate}[label=Type \arabic*:]
\item there exists $a\in K$ such that $|P(x)|= |P(a)|$;
\item there exists $a\in K$, and $r\in |K^*|$ such that $|P(x)|= \sup_{|b-a|\le r}|P(b)|$;
\item there exists $a\in K$, and $r\in \R_{>0}\setminus|K^*|$ such that $|P(x)|= \sup_{|b-a|\le r}|P(b)|$;
\item there exists a decreasing sequence of balls 
$(B_n)$ 
such that $\bigcap_n B_n=\emptyset$ and $|P(x)|= \inf_n \sup_{b\in B_n} |P(b)|$.
\end{enumerate}
We denote by $\zeta(a,r)$ the point (of type $1$, $2$ or $3$) associated to the 
ball of center $a\in K$ and radius $r\in \R_{\geq 0}$, and by $\xg =\zeta(0,1)$ the Gauss point.
The norm $\abs{\,\cdot\,}$ being non-Archimedean, two (closed) balls  
$\overline{B}(a,r)$ and $\overline{B}(b,r)$ of same radius $r \geq 0$ are either disjoint or equal.
It follows that  $\zeta(a,r)$ and $\zeta(b,r)$ are distinct when $r<\abs{b-a}$, and coincide when $r\geq \abs{b-a}$.
This phenomenon hints at the $\nR$-tree structure of $\A^{1,\an}_K$ (see~\cite[Chapter 3]{favre-jonsson:valtree}), with branching points at points of type $2$.
More precisely, all points in $\A^{1,\an}_K$ are associated with nested collections of balls in $\A^1(K)$, i.e., filters for the poset of closed balls in $\A^1(K)$, ordered by inclusion, that are closed under the intersection of decreasing sequences (when it gives a ball).
The inclusion of nested collections of balls defines a partial order on $\A^{1,\an}_K$, 
which gives to the latter the structure of a $\nR$-tree (see \cite[Section 3]{jonsson:berkovich} for further details).
More precisely, we say that $x \le y \in  \A^{1,\an}_{K}$ iff $|P(x)| \le |P(y)|$ for all $P\in K[w]$. 
This order relation endows $ \A^{1,\an}_{K}$ with a tree structure, whose 
minimal points consist of rigid and type 4 points. 

We define the \emph{diameter} of a point $x = \zeta(a,r)$ by $\diam(x) = r$ and observe that it admits a unique extension as a upper semi-continuous and nondecreasing function on $ \A^{1,\an}_{K}$. 
It follows that $\diam(x) = 0$ if and only if $x$ is of type 1, and  if $x_n$ decreases to $x$, then 
$\diam(x_n)\to \diam(x)$. 

\subsubsection{} When $K$ is not algebraically closed,  let 
$\widehat{K}$ be the completion of the algebraic closure of $K$. 
The absolute Galois group $G  = \mathrm{Gal}( K^{\mathrm{alg}}/K)$  
  (where $K^{\mathrm{alg}}$ is an algebraic closure of $K$)
acts on $\widehat K$ by  isometries so it preserves the radii of balls and 
acts naturally on $\A^{1,\an}_{\widehat{K}}$. 
There exists a canonical $G$-equivariant 
restriction map $\pi\colon\A^{1,\an}_{\widehat{K}}\to\A^{1,\an}_K$, which identifies  
$\A^{1,\an}_K$  with the quotient $\A^{1,\an}_{\widehat{K}}/G$. 

The action of $G$ preserves the type of points so that we may speak of 
the type of a point in $\A^{1,\an}_{K}$.  Rigid points in $\A^{1,\an}_{K}$ are type 1 points 
defined over a finite extension of $K$
(in general the extension $\widehat{K}/K$ is infinite).

We define the \emph{capacity} of a point $x\in \A^{1,\an}_{K}$ by  
\begin{equation}\label{eq:capa}
\log \capa(x)= \inf \left\{\frac{\log |P(x)|}{\deg P}, \, P\in K[w] \text{ monic non-constant}\right\}\text{.}
\end{equation}
When $K$ is algebraically closed, the infimum above can be taken over polynomials of degree $1$ and 
in this case we deduce that $\capa(x)=  \diam(x)$ for every $x\in \A^{1,\an}_{K}$. 

For arbitrary $K$, we only have $\capa(\zeta(a,r)) \leq r$. 
In particular the capacity is zero for a rigid point. The  computation of the capacity is 
detailed below in the case of formal Laurent series.

As it is customary, we use the following  notation: 
any $x \in \A^{1,\an}_{K}$ is a semi-norm $\abs{\cdot}_x$
on $K[w]$, and we denote $\abs{P(x)} = \abs{P}_x$. Abusing slighlty, we write 
$\abs{x}$ for  $\abs{w}_x$. 

%
%
%

\subsection{The field of formal Laurent series and  its extensions}\label{subs:formal_laurent}
Let $k$ be any algebraically closed field of characteristic $0$. 
We endow the field $k\laurent{z}$ of formal Laurent series 
with the $z$-adic norm normalized by $|z|= e^{-1}$
so that 
\[
\abs{\phi}=e^{-\ord_z(\phi)}\text{, where }  \ord_z\left(\sum_{n\in \Z} a_n z^n\right) = \min \{n \in \Z , \, a_n \neq0\}\text{.}
\]
It is a complete metrized field with value group $e^\Z$, and ring of integers
$$
k\laurent{z}^\circ = \{ x \in k\laurent{z}, \, |x| \le 1\} = k\fps{z}\text{.}
$$ 

The algebraic closure of $k\laurent{z}$  is the field of Puiseux series
\[
k\puiseux{z}
=
\bigcup_{q\in \N^*} 
k\big(\!\big(z^{1/q}\big)\!\big)
\]
 This field is not complete, and its completion
 \[
 \nL = \left\{ \sum_{n\in\N} a_n z^{\beta_n}, \, a_n \in k, \text{ s.t. }
 (\beta_n)\in \Q^\N, \  \forall n\ 
 \beta_n < \beta_{n+1}, \, \lim_{n\to\infty} \beta_n = \infty
 \right\}
 \]
is algebraically closed. 

The absolute Galois group of $k\laurent{z}$ is isomorphic to the projective limit  $\widehat{\nU}$
of the groups of $m$-th root of unity $\nU_m$ under the morphisms $\nU_{nm} \to \nU_m$ defined by 
$z \mapsto z^n$. An element $u \in \widehat{\nU}$
is a collection of $n$-th root $u_n\in\nU_n$ such that 
$u_{nm}^n=u_m$ for all $n,m$.
The action of  an element $u \in \widehat{\nU}$ on $z^\beta$
is given as follows: write $\beta= p/q$ with $p$ and $q$ coprime, 
and set $u\cdot z^\beta =u_q^p z^\beta$.

\subsection{The Berkovich affine line over $k\laurent{z}$}\label{sec:berko-rep}
Let us describe in more detail the structure of  $\A^{1,\an}_{k\laurent{z}}$,
insisting  on its relation with $\A^{1,\an}_{\nL}$. This material is mostly taken from~\cite[Chapter 4]{favre-jonsson:valtree} (albeit with a different language).

\subsubsection{Representation by Puiseux series}
Pick any element $\phi\in\nL$, and any $r\in [0,+\infty)$. 
Then we define the point $\zeta(\phi,r)\in \A^{1,\an}_{\nL}$
by setting 
\begin{equation}\label{eq:def-zeta}
\abs{P\big(\zeta(\phi,r)\big)}:= \prod_i\max\left\{e^{-\ord_z P_i(\phi(z))},r\right\}\text{,}
\end{equation}
where $P = \prod_i P_i \in \nL[w]$ and $\deg(P_i)=1$ (here, we are implicitly applying the Gauss 
lemma on the multiplicativity of the norm $ \abs{P}_{B(\phi,r)}$). 
Note that $\zeta(\phi,r)$ is of type 1 if $r=0$;
of type 2 if $r\in e^\Q$; and of type 3 if $r\notin e^\Q$. 
Moreover $\zeta(\phi,r)=\zeta(\phi',r')$ iff
$r=r'$ and $|\phi-\phi'|\le r$. 
Conversely any point in $\A^{1,\an}_{\nL}$
which is not of type 4 is of the form $\zeta(\phi,r)$ for some
$\phi\in\nL$ and some $r\ge0$. 
Observe that we have
\[
\capa(\zeta(\phi,r))=
\diam(\zeta(\phi,r))= r\in\R_{\geq 0}\text{.}
\]
A point  $x$ of type $4$ can be obtained through series of the form
$\phi= \sum_{n\in\N} a_n z^{\beta_n}$ where  $a_n\in k^*$,
$\beta_n<\beta_{n+1}\in\Q$, and $\beta= \lim_{n\to\infty}\beta_n <\infty$, 
by the analogue 
of the formula~\eqref{eq:def-zeta}
\[
\abs{P\big(\zeta(\phi,r)\big)}:=  e^{-\ord_z P(\phi(z))} .
\] 
Note that   $(\beta_n)$ has  unbounded denominators in this case, so that $P(\phi)$ is never 0. In fact we have $\diam(x)=\beta$. 

Pick now any point $x\in \A^{1,\an}_{k\laurent{z}}$, and suppose that it is the image of a point of type 1, 2 or 3 under $\pi\colon \A^{1,\an}_{\nL}\to \A^{1,\an}_{k\laurent{z}}$.
Then $x$  is represented by a  point $\zeta(\phi,r) \in \A^{1,\an}_{\nL}$, for some  $\phi \in  {\nL}$ and $r\geq 0$.
The representative is  unique  up to the action of the absolute Galois group $G$. When $\phi \in k\laurent{z^{1/q}}$, the $G$-orbit of $\zeta(\phi,r)$ is the same as its orbit under  the action of $\nU_q$.
If moreover $\ord_z(\phi)> 0$, i.e., $\phi\in k[[z^{1/q}]]$ and $\phi(0)=0$, 
then $\phi$ is the Puiseux parametrization of a formal curve $C$ in $\A^2$ that intersects the line $\{z=0\}$ with multiplicity dividing  by $q$ 
and all other representatives give other Puiseux parametrizations of $C$.
In plain words, a rigid point $x$ is attached to a curve $C$ while its preimages under $\pi$ 
corresponds to its Puiseux parameterizations. We will say more about the correspondence between 
germs of curves and points in  $ \A^{1,\an}_{k\laurent{z}}$ in \S~\ref{subs:model}. 

\subsubsection{Multiplicity and capacity  in the closed unit ball $\A^{1,\an}_{k\laurent{z}}$}\label{sssec:multicapa}

Pick any point $x$ in the Berkovich closed unit ball $\overline{B}^{\an}(0,1)=\{|x|\le1\}\subset\A^{1,\an}_{k\laurent{z}}$. 
We call the cardinality of $\pi^{-1}(x)\in \A^{1,\an}_{\nL}$
the \emph{multiplicity} of $x$ and denote it by $m(x)$. 
It is finite iff $x$ is either of type $2$ or $3$, or rigid. 
Later on, we will define another related quantity, the 
 \emph{generic multiplicity} $b(x)$. 
 
Suppose $x\in \overline{B}^{\an}(0,1)$ is a rigid point, so that it is represented by a Puiseux series
$\phi(z)=\sum_{n \geq 0} a_n z^{n/q}\in k\fps{z^{1/q}}$. 
Then, we have 
\begin{equation}\label{eq:def-mult}
m(x)= \lcm \left\{ \frac{q}{\gcd\{n, q\}},\,  a_n\neq 0\right\}
\end{equation}
It can also be interpreted geometrically as follows.
Denote by $C$ the formal curve parameterized by $t\mapsto (t^q, \phi(t^q))$. 
Then $m(x)= (C\cdot\{z=0\})$ where the latter is the local intersection product
between $C$ and the curve $z=0$ viewed in the formal scheme $\spf k\fps{z,w}$ (beware that it is not 
the multiplicity of $C$ at $(0,0)$).


Pick any $x\in \A^{1,\an}_{k\laurent{z}}$, and choose any $y\in \A^{1,\an}_{\nL}$
such that $\pi(y)=x$. Then we set $\diam(x)= \diam(y)$. This is independent
on the choice of $y$ since the diameter function is Galois invariant. 
 The capacity of the point $x$ in $\A^{1,\an}_{k\laurent{z}}$ is defined as before  by
\[
\log \capa(x)= \inf \left\{ \frac{\log \abs{P(x)}}{\deg P}, \, P\in k\laurent{z}\![w] \text{ non-constant}\right\}\text{.}
\]
Notice that the   condition that  $P$ is monic is 
 superfluous in this setting, since the absolute value is trivial on $k$ (i.e., $\abs{k^*}=1$).
It is customary  to  define
\begin{equation}
\alpha(x)= -\log \capa(x)
\end{equation}
for any point $x\in \A^{1,\an}_{k\laurent{z}}$.
Note that $\alpha(\xg)=0$, and $\alpha$ is increasing on each segment $[\xg,x]$ 
(resp. decreasing for the order $\leq$). 

\begin{prop}\label{rmk:computeskewness}  
Let $x \in \overline{B}^{\an}(0,1)\subset\A^{1,\an}_{k\laurent{z}}$ be a point of type $2$ or $3$. 
Pick any  $\phi \in \Puis$ such that 
$\zeta(\phi,0)< x$. Let $P_{\mathrm{min}} \in k\laurent{z}\![w]$ be the minimal polynomial of $\phi$. Then we have
\[
\alpha(x)=\frac{-\log \abs{P_{\mathrm{min}}(x)}}{\deg_w P_{\mathrm{min}}}\text{.} 
\]
\end{prop}

\begin{proof}
We may suppose $\phi \in k\fps{z^{1/q}}$ with $q= m(\phi)$. 
Then the minimal polynomial is equal to $P_{\min}(z,w) = \prod_{u^q=1} (w-u\cdot \phi)$, 
is monic of degree $q$, and belongs to $k[[z]][w]$.

The statement then follows from~\cite[Proposition~3.25]{favre-jonsson:valtree}. 
The latter reference is written in the context of valuations instead of norm. 
We refer to \S\ref{sec:compare} below for a quick discussion on how to compare these
point of views. 
\end{proof}

\begin{prop}\label{prop:approxseq}
Suppose $x\in\overline{B}^{\an}(0,1)\subset\A^{1,\an}_{k\laurent{z}}$. 
Then there exists a  decreasing sequence of  type 2 points
$x_0= \xg$, $ x_{n+1}<x_n$, $\lim_n x_n=x$
such that: 
\begin{enumerate}[(i)]
\item
the multiplicity is constant equal to $m_n$ on the  interval  $I_n=[x_{n+1},x_{n})$;
\item
$- \log \frac{\diam(y)}{\diam(y')}= m_n (\alpha(y) - \alpha(y'))$
for all $y,y'\in \overline{I_n}$.
\end{enumerate}
Moreover $m_1=1$, and $m_n$ divides $m_{n+1}$ for all $n$. 

This sequence is finite if and only if $x$ has finite multiplicity. 
\end{prop}

\begin{proof}
We only give a sketch of proof, referring to~\cite[Chapter 5]{favre-jonsson:valtree} for more details. 

Suppose $x=\pi(\zeta(\phi,r))$ with 
$\displaystyle \phi =\sum_{n\in\N} a_n z^{\beta_n}\in\nL^\circ$ with $a_n\in k^*$, $0\le \beta_n<\beta_{n+1}\in\Q$
and $r\in[0,1]$. 

Write $\beta_n= p_n/q_n$ with $p_n,q_n$ coprime integers. 
We define by induction a sequence of integers $n_j$
such that $n_0=0$, and $n_j$ is the smallest integer such that $q_j$ is not
a multiple of $\lcm\{q_l, l\le n_j-1\}$. Then 
we set $m_j= \lcm\{q_l, l\le n_j\}$ so that $m_j$ is a multiple of $m_{j-1}$. 
One can check that if $\beta_{n_{j-1}}\le -\log r < \beta_{n_{j}}$, then
\[
m\big(\pi(\zeta(\phi,r))\big)= 
m\Big(\pi\Big(\zeta\Big(\sum_{n< n_j} a_n z^{\beta_n},r\Big)\Big)\Big)
=
m_{j-1}
\]
which implies (1). 
Moreover the supremum of $\log|P|/\deg(P)$
on that segment is attained for the minimal polynomial of 
$  \sum_{n< n_j} a_n z^{\beta_n}$. This implies (2).
\end{proof}

Since the multiplicity is increasing on $[\xg,x]$, we obtain:

\begin{cor}\label{cor:approxseq}
The function $-\log \diam|_{[x,\xg]}$ is a convex function of $\alpha$, or equivalently, $\alpha$ is a concave function of $-\log \diam|_{[x,\xg]}$. 
In particular
\[ -\log \diam \ge \alpha\]
and if $m=m(x) <\infty$, we have 
$ -\log \diam \le m(x) \alpha$ on $[x,\xg]$.
\end{cor}

The second item of Proposition~\ref{prop:approxseq}  implies that
the  restriction of $\alpha$ on any segment $[y,\xg]$, $y>x$
 is continuous, and its concavity implies that it is continuous at $x$ as well.  From this 
  we get a canonical parametrization $[0,\alpha(x)] \to [\xg,x]$, 
$t\mapsto x(t)$ such that $\alpha(x(t))=t$. 

We shall also use the following 
\begin{lem}\label{lem:crucial-skew}
For any $x \le x' \in \overline{B}^{\an}(0,1) \subset \A^{1,\an}_{k\laurent{z}}$, 
for any monic $P\in k\fps{z}[w]$ we have
\[
\frac{\alpha(x')}{\alpha(x)}
\le1\le 
\frac{-\log|P(x)|}{-\log|P(x')|}
\le
\frac{\alpha(x)}{\alpha(x')}
\]
\end{lem}
%

\begin{proof} 
The  two  inequalities on the left follow  from the fact that both $\alpha$ and  $-\log|P|$
are non-increasing (with respect to   the  order relation). Let us fix any $\phi\in \Puis$. 
For any $t\in[1,+\infty)$, denote by $x(t)$ the unique point in $[\xg,\phi]$ such that
$\alpha(x(t))= t$. Suppose that $x(t_0)= x'$ and $x(t_1)=x$ for some $t_0< t_1$. 
It follows from~\cite[Prop 3.25]{favre-jonsson:valtree} and a decomposition into irreducible components that  $t\mapsto -\log|P(x(t))|$ is concave. Since $P$ is monic
and its coefficients are formal power series in  $z$, it follows that 
$-\log|P(\xg)|=0$. 
It follows that 
$-\log|P(x(t))| 
\le -\frac{t}{t_0}\,\log |P(x(t_0))|$
which implies the claim.
\end{proof}

\subsubsection{Relation with valuation spaces}\label{sec:compare}
In the literature~\cite{favre-jonsson:eigenval,gignac-ruggiero:attractionrates,gignac-ruggiero:locdynnoninvnormsurfsing}, the local analysis of super-attracting analytic dynamical 
maps in dimension $\ge2$ was mostly formulated in the language of valuations. 
For the convenience of the reader, we indicate how to translate notions introduced above in the context of 
the Berkovich affine line to valuation spaces, referring to \cite[Chapters 3 \& 4]{favre-jonsson:valtree} for further details.

Any seminorm $\abs{\,\cdot\,}_x \in \A^{1,\an}_K$ defines a semivaluation $\nu_x(\cdot) = -\log \abs{\,\cdot\,}_x$ over the ring $K[w]$.
When $K=\nL$, we get a space of semivaluations over $\nL[w]$ which extend the $z$-adic valuation, defined by 
\[
\ord_z\left(\sum_{n \in \nN} a_n z^{\beta_n}\right) = \min\{\beta_n, a_n \neq 0\}\text{.}
\]
on $\nL^*$ and by $\ord_z(0)=+\infty$.
Any such valuation $\nu_x$ extends uniquely to a semivaluation, again denoted by $\nu_x$, over $\nL\fps{w}$ (and all semivaluations over $\nL\fps{w}$ come from the ones in $\nL[w]$).
Valuations associated to series $\phi \in  {\nL} $ such that $\ord_z(\phi)>0$ 
 are called \emph{centered}, and the corresponding subset 
 is denoted by $\widehat{\mc{V}}_z^*$. They correspond to the Berkovich open unit ball $B^{\an}(0,1)=\{|x|<1\}$ in $\A^{1,\an}_\nL$. By adding the non-centered valuation $\ord_z$, corresponding to the Gauss point $\xg$, we get the valuation space $\widehat{\mc{V}}_z$.

The absolute Galois group $G$ of   $k\laurent{z}$ acts on $\widehat{\mc{V}}_z^*$, and its quotient $\mc{V}_z^*=\widehat{\mc{V}}_z^*/G$ can be identified as the set of centered semivaluations over $k\laurent{z}[w]$ extending the $z$-adic valuation on $k\laurent{z}[w]$, or equivalently with the set of centered semivaluations over $k\fps{z,w}$ that extend the trivial vauation on $k$, and normalized by the condition $\nu(z)=1$.
These valuations correspond to the Berkovich open unit ball $B^{\an}(0,1)$ in $\A^{1,\an}_{k\laurent{z}}$. Again, the whole (normalized) valuation space is given by $\mc{V}_z = \mc{V}_z^* \sqcup\{\ord_z\}$. This space is also known as the \emph{non-Archimedean link} of $(\nA^2_k,0)$, see \cite{fantini:normalizedBerkovich}.

The value $1-\log \capa(x)=1-\log \diam(x)$ of $x=\zeta(\phi,r) \in B(0,1) \subset \A^{1,\an}_{\nL}$ is often denoted by $A_z(\nu_x)$, and known as the \emph{relative thinness}, or \emph{relative log-discrepancy} of the valuation $\nu_x \in \mc{V}_z$.

For the capacity $\capa(x)$ of $x \in B^{\an}(0,1) \subset \A^{1,\an}_{k\laurent{z}}$, we have
$$
\alpha_z(\nu_x):= - \log \capa(x) = \sup\left\{\frac{\nu_x(P)}{\deg P} \right\}\text{.}
$$
Here $\deg P$ is the degree with respect to the $w$ coordinate, which coincides with the intersection multiplicity of $(P=0)$ with the line $(z=0)$; this quantity is known as the \emph{relative skewness}, or as the \emph{relative Izumi constant} of $\nu_x \in \mc{V}_z$.
This quantity is often used to parametrize valuations, in the sense that for any $\phi \in \widehat{\nL}$, and for any $t > 0$, there exists a unique semivaluation $\nu_x \in \mc{V}_z$ with $x=\pi\big(\zeta(\phi,r)\big)$ and such that $\alpha_z(\nu_x) = t$. This semivaluation is often denoted by $\nu_{\phi,t}$.
For convenience in this paper  we use the same notation $A(\cdot)$ and $\alpha(\cdot)$ for the corresponding quantities. 

When $x \in B^{\an}(0,1) \subset \A^{1,\an}_{k\laurent{z}}$, the sequence $(x_n)_n$ given by Proposition~\ref{prop:approxseq} corresponds to 
 the approximating sequence in the valuation setting, see \cite[Section 3.5]{favre-jonsson:valtree}.

\section{Dynamics of   polynomial skew products with small relative degree}\label{sec:dynamics_psp}

\subsection{Polynomial skew products and their actions on the Berkovich line}

In this section, we fix once for all an algebraically closed field $k$ of characteristic $0$
and consider a     map of the form
\begin{equation*}\tag{\ref{eq:cd_intro}}\label{eq:cd}
f(z,w) = \biggl(z^d, w^c + \sum_{j=0}^{c-1} h_j(z) w^j\biggr)
\end{equation*}
where $d,c \geq 2$ and $h_j \in k\fps{z}$ satisfy $h_j(0)=0$.
We will often work in the \emph{small relative degree} regime, that is: $c < d$. 

Such a map naturally induces a self-map  $f_\diamond\colon \A^{1,\an}_{k\laurent{z}} \to \A^{1,\an}_{k\laurent{z}}$ as follows. 
Observe first that thanks to the non-Archimedean property, 
for any point $x\in  \A^{1,\an}_{k\laurent{z}}$, the assignment 
$P\mapsto |(P \circ f)(x)|^{1/d}$ defines a multiplicative semi-norm on 
$k\laurent{z}[w]$ (where $P$ is viewed as a function of $(z,w)$). In addition if $P$ has degree 0, 
$|(P \circ f)(x)| = |z|^d$, so we conclude that $P\mapsto |(P \circ f)(x)|^{1/d}$ belongs to $\A^{1,\an}_{k\laurent{z}}$. Thus we may define a map 
$f_\diamond: \A^{1,\an}_{k\laurent{z}}\to \A^{1,\an}_{k\laurent{z}}$ by setting 
\begin{equation}\label{eq:defi}
|P(f_\diamond(x))| := |(P \circ f)(x)|^{1/d} \text{, for any } P\in k\laurent{z}[w].
\end{equation}

It is also convenient to   work over $\nL$ since the latter  
 is algebraically closed and complete. For this, it is enough to observe that~\eqref{eq:defi} 
also defines a 
map $f_\diamond\colon \A^{1,\an}_{\nL} \to \A^{1,\an}_{\nL}$
such that $f_\diamond \circ \pi = \pi \circ f_\diamond$ where $\pi\colon \A^{1,\an}_{\nL} \to \A^{1,\an}_{k\laurent{z}}$
denotes the canonical projection.

\begin{lem}\label{lem:diamond}
For any $\phi\in\nL$, we have 
\begin{equation} \label{eq:comp}
f_\diamond(\phi) 
= 
\lrpar{\phi \big({z^{1/d}}\big)}^c+ \sum_{j=0}^{c-1} h_j\big({z^{1/d}}\big)\, \lrpar{\phi\big({z^{1/d}}\big)}^j.
\end{equation}
\end{lem}
\begin{proof}
Indeed, since $\nL$ is algebraically closed $f_\diamond(\phi)$ is determined by its value on
linear polynomials $w- \psi(z)$ with $\psi\in \nL$, and for such a polynomial $P$ we have 
\begin{align*}
|P(f_\diamond(\phi))| = |(P \circ f)(\phi)|^{1/d}
&= 
\bigg|
\phi^c(z)+ \sum_{j=0}^{c-1} h_j(z) \phi^j(z)-\psi(z^d)\bigg|^{1/d}_{\nL}
\\
&= 
\bigg|
\phi^c\big({z^{1/d}}\big)+ \sum_{j=0}^{c-1} h_j(z) \phi^j\big({z^{1/d}}\big)-\psi(z)\bigg|_{\nL},
\end{align*}
thereby proving the claim. 
\end{proof}

\subsection{Structure of the map $f_\diamond$}\label{subs:structure}

\begin{thm}\label{thm:basic-finite}
Let $f$ be a polynomial map of the form~\eqref{eq:cd}. 
The map $f_\diamond\colon \A^{1,\an}_{k\laurent{z}} \to \A^{1,\an}_{k\laurent{z}}$  (resp. $\A^{1,\an}_{\nL}\to \A^{1,\an}_{\nL}$)
is open, continuous
and proper. It preserves the type of points, the set of rigid points, and the preimage of any point is finite
of cardinality at most $cd$ (resp. at most $c$ in $\A^{1,\an}_{\nL}$). 
\end{thm}

 To prove the theorem, we introduce the map 
$\tau = \tau_d\colon  \A^{1,\an}_{\nL} \to  \A^{1,\an}_{\nL}$ 
defined by  
\begin{equation}\label{eq:tau}
\tau\colon x\longmapsto \lrpar{P\mapsto \abs{P(\tau(x))} := \abs{P(z^d,w)(x)}^{1/d} = \abs{P(z^d,w)}_x^{1/d}}~.
\end{equation}
for any $x\in \A^{1,\an}_{\nL}$ and any $P\in \nL[w]$, 
(where we 
 interpret $P$ as a series in both variables $z$ and $w$).
 In other words, with notation as above, if we put $h(z,w) = (z^d,w)$, then $\tau  = h_\diamond$. 
 
%
%

\begin{lem}\label{lem:tau}
The map $\tau$ defines  a homeomorphism on $\A^{1,\an}_{\nL}$ 
 which preserves the types of points. 
\end{lem}

\begin{proof}
It is immediate that $\tau\inv $ is given by 
\[ \abs{P(\tau^{-1}(x))} = \abs{P\lrpar{z^{1/d},w}(x)}^{d}\text{ for any } P\in \nL[w],\]
and the continuity of $\tau$ and $\tau\inv$ follows immediately from the definitions.  
Lemma~\ref{lem:diamond} applied to $h(z,w)  = (z^d,w)$ gives
$\tau(\phi) = \phi({z^{1/d}})$ for any $\phi\in\nL$. It follows that type 1 points, and Puiseux series
are preserved. Since $|\tau(\phi)- \tau(\psi)| = |\phi-\psi|^{1/d}$, 
and $\tau$ is a homeomorphism, we also have $\tau(B(\phi,r))= 
B(\tau(\phi), r^{1/d})$ which implies that $\tau$   preserves
 type 2 and 3 points, hence also type 4 points. 
\end{proof}

\begin{rmk} ~ \label{rmk:tau}
\begin{enumerate}
\item 
Since $\tau$ restricts to $\nL$ as the field automorphism 
$ \phi(z)\mapsto   \phi(z^{1/d})$, the notation  is consistent with that of~\cite{nie-zhao:Berkdyntwistratmaps}. 
\item 
The definition of $\tau$ in \eqref{eq:tau} makes sense on $\A^{1,\an}_{k\laurent{z}}$ but it does not define  a homeomorphism there.  For instance the Puiseux series $\pm z^{1/2}$ parameterize the same curve $w^2 = z$ but their preimages under $\tau_2$ correspond to the two distinct curves $w= \pm z$. 
\end{enumerate}
\end{rmk}

\begin{lem}\label{lem:f_rond_tau}
In  $\A^{1,\an}_{\nL}$ we have the decomposition $f_\diamond = \hat{f} \circ \tau$
where $\hat{f}$ is the polynomial map 
$$\hat{f}(w) = w^c +\sum_{j=0}^{c-1}h_j\big({z^{1/d}}\big) w^j.$$ 
\end{lem}

%
\begin{proof}
Indeed, a direct computation gives 
\begin{align*}
\left|P\left(\hat{f}(\tau(x))\right)\right| 
&= 
\bigg|
P\biggl(z, w^c +\sum_{j=0}^{c-1}h_j\big({z^{1/d}}\big) w^j\biggr)(\tau(x))
\bigg|
\\
&=
\bigg|
P\biggl({z^d, w^c +\sum_{j=0}^{c-1}h_j ({z } )  w^j}\biggr)(x)
\bigg|^{1/d}
=
\abs{P(f_\diamond(x))}~,
\end{align*}
and we are done. 
\end{proof}

\begin{proof}[Proof of Theorem~\ref{thm:basic-finite}]
From the basic properties of the action of a polynomial map on the Berkovich affine line
(see, e.g.,~\cite{benedetto:dyn1NA}), we infer that 
$\hat{f}$ is a open, continuous, proper, $c$-to-one map on $\A^{1,\an}_{\nL}$, and that it preserves the type
of points.
By composition, we conclude that $f_\diamond$ is a open, proper, continuous, $c$-to-one map
on  $\A^{1,\an}_{\nL}$ that preserves the type of points and rigid points.

\medskip

The continuity of $f_\diamond$ on $\A^{1,\an}_{k\laurent{z}}$ follows 
directly  from formula~\eqref{eq:defi} and the definitions. 
Since $\pi\colon \A^{1,\an}_{\nL}\to \A^{1,\an}_{k\laurent{z}}$ is proper (resp. open), the properness (resp. openness) 
of $f_\diamond$ on $\A^{1,\an}_{k\laurent{z}}$ follows from
the one of $f_\diamond$ on $\A^{1,\an}_{\nL}$.

To see that $f_\diamond$ is (at most) $cd$-to-one on $\A^{1,\an}_{k\laurent{z}}$, it is sufficient to prove
that any rigid point has at most $cd$ preimages, since $f_\diamond$ is open and the set of rigid
points is dense. 

So let $x\in \A^{1,\an}_{k\laurent{z}}$ be a rigid point. We   only treat the case  where $|x|\le 1$.
When $|x|>1$, one needs to work in the coordinates $w^{-1}$. 
Geometrically, $x$ is given by a formal germ of irreducible curve
we denote by $C$. It is determined by some Puiseux series $\phi(z)= \sum_{i\ge 0} a_i z^{i/m}$
where $m= \gcd\{ i, a_i\neq0\}$ is the intersection number $C\cdot (z=0)$. A parameterization
of $C$ is given by $t\mapsto (t^m,\phi(t^m))$. 
An equation for $C$ is given by 
\[P(z,w)=\prod_{i=1}^m (w- \zeta^i \cdot \phi)\]
where $\zeta$ is a primitive $m$-th root of unity and $\zeta\cdot\phi$ is defined as in \S~\ref{subs:formal_laurent}.

Any preimage $y_i$ of $x$ is also a rigid point hence it is determined by a Puiseux series
$\psi_i$ of multiplicity $m_i$, and an equation for the associated branch $D_i$
is denoted by $P_i$. 
Since $f_\diamond(y_i)=x$, $P_i$ divides $P\circ f$. 
We let $r_i = \ord_{P_i}(P\circ f)\in\N^*$.
Also, $f(t^{m_i},\psi(t^{m_i}))= (t^{dm_i},  \phi(t^{dm_i}))$, and
$dm_i= e_i m$ for some $e_i\in\N^*$.   

Let $\{y_i\}_{i\in I}$ be  the set of preimages of $x$ in $\A^{1,\an}_{k\laurent{z}}$. 
We   compute the local intersection product 
\[
(f^*C \cdot f^*\set{z=0})_0
=
 \sum r_i D_i \cdot  d \{z=0\}
=\sum dr_i m_i
\]
As the left hand side is equal to $cd (C\cdot [z=0])_0= cdm$, we obtain
\begin{equation}\label{eq:mult}
\sum_I r_i e_i = cd~, 
\end{equation}
therefore $\#I\leq cd$ and the proof is complete. 
\end{proof}

\subsection{Basic dynamical properties of $f_\diamond$}
A closed (resp. open) ball in the Berkovich affine line over $\nL$
is defined as $\overline{B}^{\an}(\phi,r):= \{x,\, |(w-\phi(z))(x)|\le r\}$ (resp. $B^{\an}(\phi,r):=\{x,\, |(w-\phi(z))(x)|<r\}$) for some 
$\phi\in \nL$ and $r\ge 0$. Observe that $\overline{B}^{\an}(\phi,r)\cap \nL = \overline{B}(\phi,r)$, and 
$B^{\an}(\phi,r)\cap \nL = B(\phi,r)$,

A ball in $\A^{1,\an}_{k\laurent{z}}$ is the image under the canonical projection 
$\pi\colon\A^{1,\an}_{\nL}\to \A^{1,\an}_{k\laurent{z}}$ of a ball in $\A^{1,\an}_{\nL}$.
We abuse notation and identify $\overline{B}^{\an}(\phi,r)$ and  ${B}^{\an}(\phi,r)$ with 
their images in $\A^{1,\an}_{k\laurent{z}}$. Observe that the preimage of a ball under $\pi$ is a union of balls.

\begin{thm} \label{thm:basic_dynamical}
Let $f$ be a polynomial map of the form~\eqref{eq:cd}, with $2 \leq c < d$. 
The map $f_\diamond\colon \A^{1,\an}_{\nL}\to \A^{1,\an}_{\nL}$
  (resp. $\A^{1,\an}_{k\laurent{z}}\to  \A^{1,\an}_{k\laurent{z}}$) has the 
following properties.
\begin{enumerate} [(i)]
\item The Gauss point $\xg$ is totally invariant, and 
 $f_\diamond \colon B^{\an}(0,1) \to B^{\an}(0,1)$ is proper and order preserving.
\item
The ball $B^{\an}_\infty =\{|x|>1\}$ is totally invariant, and $f^n(x) \to \xg$ for any $x\in B^{\an}_\infty$.
\item
The open unit ball is totally invariant. 
\item
For any $a\in k^*$, the map $f_\diamond$ maps $B^{\an}(a,1)$ onto $B^{\an}(a^c,1)$
and
$f_\diamond\colon B^{\an}(a,1) \to B^{\an}(a^c,1)$
is proper. In addition, over $\nL$, it is a homeomorphism. 
\begin{itemize}
\item 
When $a \in k^*$ is not a root of unity, then $f^n_\diamond(x)\to \xg$
for any point $x\in B^{\an}(a,1)$.
\item 
When $a \in k^*$ is a root of unity, 
then $B^{\an}(a,1)$ is eventually mapped to a 
periodic ball $B^{\an}$ which contains a unique periodic cycle lying in $k\fps{z}$.
Any point in $B^{\an}(a,1)$ is either eventually mapped to this cycle, or 
converges under iteration to $\xg$.
\end{itemize}
\end{enumerate}
\end{thm}

\begin{proof} 
From~\eqref{eq:comp}, we infer
\begin{equation}\label{eq:norm}
|f_\diamond(\phi)|=
\bigg|
\phi^c\big({z^{1/d}}\big)+ \sum_{j=0}^{c-1} h_j\big({z^{1/d}}\big)\, \phi^j\big({z^{1/d}}\big)
\bigg|
= 
|\phi|^{c/d}
\end{equation}
for all $\phi\in\nL$ such that $|\phi|>1$ hence for all $|\phi| \ge1$ by continuity. 
Conversely, when $\abs{\phi}<1$  we have 
$|f_\diamond(\phi)|\le \max\{|\phi|^{1/d}, |h_0|\}<1$,  and likewise 
if $\abs{\phi}\leq 1$ then $\abs{f_\diamond(\phi)}\leq 1$. This
 implies that $B^{\an}(0,1)$ is invariant and  $B^{\an}_\infty = \{|x|>1\}$
is totally invariant. Since by assumption   $c/d<1$, it also follows   that  
$|f_\diamond^n(\phi)| = |\phi|^{(c/d)^n} \to 1$ for all 
$\phi\in B^{\an}_\infty$, which proves (ii) in $\A^{1,\an}_{\nL}$. 
Since $\xg$ is the unique boundary point of $B_\infty$ and $f_\diamond$ is open, 
we infer that $\xg$ is totally invariant too.
As $\pi\inv(\xg) = \xg$, the Gauss point is also totally invariant and (ii) holds 
 in $\A^{1,\an}_{k\laurent{z}}$.

The properness of   $f_\diamond \colon B^{\an}(0,1) \to B^{\an}(0,1)$   follows from the fact that 
it extends continuously to the closure $B^{\an}(0,1)\cup\{\xg\}$ of the open unit ball 
and that $f_\diamond(\xg)=\xg$. 
Suppose that $x\le x' \in B^{\an}(0,1)$ and pick any $P\in k\laurent{z}[w]$ (or in $\nL[w]$). 
Since $z$ belongs to the base field $k\laurent{z}$, 
we have $|z(x)|=|z(x')|$,  so we may multiply $P$ by a suitable power of $z$
and assume that $P\in k\fps{z} [w]$. 
In that case
\[
|P(f_\diamond (x))|= |(P\circ f) (x)|^{1/d} \le |(P\circ f) (x')|^{1/d} = |P(f_\diamond (x))|
\]
hence 
$f_\diamond (x)\le f_\diamond (x')$, and assertion~(i) is established (both over $\nL$ and $ k\laurent{z}$).

Suppose now that $\phi \in B(a,1)\subset \nL$ for some $a\in k$.
Then, writing $a(z) = \sum a_n z^{\beta_n}$, we get 
 $\beta_0=0$, $a_0=a$, and since $h_j(0)=0$ we get
\[
f_\diamond(\phi)=
\phi^c\big({z^{1/d}}\big)+ \sum_{j=0}^{c-1} h_j\big({z^{1/d}}\big)\, \phi^j\big({z^{1/d}}\big)
=
a^c + \sum_{n\ge1} b_j z^{\beta'_j}
\]
where $b_j\in k$ and $\beta'_j\in \Q^*_+$. 
This proves that $f_\diamond(B^{\an}(a,1)) \subset B^{\an}(a^c,1)$ for all $a\in k$, which, together with the total invariance of $\set{\abs{x}> 1}$,  implies that $B^{\an}(0,1)$ is totally invariant. 

We suppose now that $a\in k^*$. 
Since $\xg$ is the unique boundary point of $B^{\an}(a,1)$, $f_\diamond(\xg)=\xg$,
 and $f_\diamond$ is continuous
we get that  $f_\diamond\colon B^{\an}(a,1) \to B^{\an}(a^c,1)$ is proper.  
 We claim that $f_\diamond\colon B^{\an}(a,1) \to B^{\an}(a^c,1)$ is a homeomorphism when working 
 in $\A^{1, \an}_{\nL}$.
The easiest way to see this is to pre- and post-compose $f$ by translations
$T_1(w)= w+a$, and $T_2(w)= w- a^c$ so that we are reduced to analyzing $g_\diamond$
on the open unit ball, where 
$g$ is defined by 
\[
g(z,w) = T_2 \circ f \circ T_1 (z,w)  = f(z, w+a) - (0, a^c) = \biggl(z^d , cw + \psi(z)+ \sum_{j=2}^{c}g_{j}(z) w^j\biggr)
 \]
for some $\psi, g_{j} \in k\fps{z}$ with $\psi(0)=0$. 
Observe that the translation by $\psi$
is a homeomorphism so that we can assume that $\psi =0$. 
One can then write $g$ as a composition
$g = \gamma_1 \circ \gamma_2$ where $\gamma_1(z,w)=(z^d,w)$ and
$\gamma_2(z,w)=(z, cw( 1+ \sum_{j=1}^{c-1} c^{-1} g_{j}(z) w^j))$.
Observe that $\gamma_2$ induces  an analytic isomorphism on the Berkovich unit ball (see, e.g.,~\cite[Proposition~3.22]{benedetto:dyn1NA} or \S~\ref{subs:update} below),
and $(\gamma_1)_\diamond=\tau$ 
is a homeomorphism on $\A^{1,\an}_{\nL}$ preserving the unit ball. 
We conclude that
$g_\diamond= (\gamma_1)_\diamond \circ  \gamma_2$ is a homeomorphism of the open unit ball and 
the first assertion of (iv) follows over $\nL$.   
Over  $k\laurent{z}$, this still shows that 
$f_\diamond\colon B^{\an}(a,1) \to B^{\an}(a^c,1)$ is proper and  surjective.

If $a\in k^*$ is not a root of unity, the open balls $B^{\an}(f_\diamond ^n (a),1)$ are disjoint and the 
properties of the Berkovich topology imply that for any $x\in B^{\an}(a,1)$, $f_\diamond^n(x)\to \xg$.  

Finally, suppose that $a$ is a root of unity, and let  $n_0$ be such that $a^{c^{n_0}}$ is periodic under $w\mapsto w^c$, of period $\ell$. Let $B^{\an}= B^{\an}(a^{c^{n_0}},1)$. Then $f^{n_0}(B^{\an}(a,1))=B^{\an}$, and $B^{\an}$ is periodic of period $\ell$. 
Replacing $f$ by a suitable iterate, we may assume that $B^{\an}= B^{\an}(b,1)$ and $\ell=1$. 
Observe that all iterates of $f$ bear the same form~\eqref{eq:cd} with $c,d$ replaced by $c^n$ and $d^n$. 
We then conjugate by a translation $T(w):= w+b$ so that  $g:= T^{-1} \circ f \circ T$ has the form
\[
g(z,w)=(z^d, c w+ O(z,w^2))
\]
In this situation,~\cite[Theorem~2.1]{dujardin-favre-ruggiero:DMM} applies, and we infer the existence of a formal power series
$\phi$ vanishing at zero fixed and by $g$. 
Translating again by this formal power series, we may assume that $\phi=0$, which means  that 
$g(z,w)=(z^d, c w( 1+ O(z,w)))$. 
A direct computation shows that 
\[
g_\diamond (a z^\alpha + o(z^\alpha))
= 
ac z^{\alpha/d}  +  o(z^{\alpha/d})
\] 
for any $a\in k^*$ and $\alpha>0$, and since $g_\diamond$ is order preserving, this  
   implies $g_\diamond^n(x) \to \xg$ for all $x\in B^{\an}\setminus \{0\}$. 
Thus, the proof of (iv) is complete, and we are done. 
\end{proof}


\section{The  invariant  Cantor set $\cK$ and    measure $\mu_{\na}$}\label{sec:cantor}
We work under the same assumptions as in the previous section with a map
of the form
\begin{equation*}\tag{\ref{eq:cd}}
f(z,w)= \biggl(z^d, w^c + \sum_{j=0}^{c-1} h_j(z) w^j\biggr)~,
\end{equation*}
with $d>c\ge2$, and $h_j(0)=0$, together with its associated map $f_\diamond$.  
We   focus our attention to the restriction of $f_\diamond$ to
the open unit ball $B^{\an}(0,1) \subset \A^{1,\an}_{k\laurent{z}}$.  

\subsection{The Jacobian formula (I)}
The Jacobian of $f$ is equal to 
\[
\jac(f) = d z^{d-1} \jac_w(f)\]
where
\[
\jac_w(f) = cw^{c-1} + \sum_{j=1}^{c-1} j h_j(z) w^{j-1}=c \prod_{j=1}^{c-1} \big(w-c_j(z)\big)~,
\]
where the Puiseux series $c_j \in \Puis$ are parametrizations of the critical branches of $f$ (other than the branch $(z=0)$). We let
$\CritB=\{c_1, \ldots, c_{c-1}\}$ be the set of critical branches.

The following Jacobian formula is crucial for our investigation. 
\begin{lem}\label{lem:jac}
For any $x\in B^{\an}(0,1)$ we have
\begin{equation}\label{eq:jac}
- \log \diam(f_\diamond(x)) = \frac1{d} \big(
- \log \diam(x) - \log |\jac_w(f)(x)|\big)~.
\end{equation}
\end{lem}
\begin{proof}
Since the diameter function is Galois equivariant, we can work in $\A^{1,\an}_\nL$. 
We use the decomposition $f_\diamond = \hat{f} \circ \tau$ as in the proof of Theorem~\ref{thm:basic-finite}.  
Since $\tau(B(\phi,r))= B(\tau(\phi), r^{1/d})$, we have $\diam(\tau(x))^d= \diam(x)$. 
The map $\hat{f}= w^c+\sum_j h_j(z^{1/d}) w^j$ is a polynomial function on the Berkovich space, 
and $\nL$ has residual characteristic $0$ so that 
\[
\diam(\hat{f}(x))= \diam(x) \times |\hat{f}'(x)|
\]
by~\cite[Proposition~3.4]{favre-riveralelier:expansion}, where $\hat f'$ is the derivative with respect to $w$. 
It follows that
\[
\diam(f_\diamond(x))
=
\diam(\hat{f}(\tau(x)))
=
\diam(\tau(x)) \times |\hat{f}'(\tau(x))|
=
\diam(x)^{1/d} \times |\jac_w(f)(x)|^{1/d},
\]
thereby proving~\eqref{eq:jac}.
\end{proof}


\subsection{The invariant Cantor set} \label{ssec:invariantCantor}

The next theorem describes the topological dynamics of $f_\diamond$ in $B^{\an}(0,1)$ both in  $\A^{1,\an}_{k\laurent{z}}$
 and  $\A^{1,\an}_\nL$. Let us introduce the sets 
\[
\cK := \{x \in B^{\an}(0,1)\subset \A^{1,\an}_{k\laurent{z}}, \, f^n_\diamond(x) \text{ does not converge to }\xg \} 
\]
and 
\[
\cK_\nL := \{x \in B^{\an}(0,1)\subset \A^{1,\an}_\nL, \, f^n_\diamond(x) \text{ does not converge to }\xg \}. 
\]

\begin{thm}\label{thm:invariantcantor}
Let $f$ be a polynomial map of the form~\eqref{eq:cd}, with $2 \leq c < d$. Then 
$\cK_\nL = \pi^{-1}(\cK)$, and
both sets are compact totally invariant subsets included in $\overline{B}^{\an}(0,\rho_0)$
for some $\rho_0<1$. The set $\cK_\nL$ (resp. $\cK$) 
consists of type 1 points (resp. type 1 points such that $\alpha=\infty$).

 We have the following dichotomy: 
either $\cK$ and $\cK_\nL$ are both  Cantor sets, or they are both 
 reduced to a single point,  that corresponds to 
 a totally invariant smooth (formal)  curve
transverse to $(z=0)$. In the latter case, $f$ is of the form $f(z,w)=(z^d, (w+\phi(z))^c-\phi(z^d))$ for some 
$\phi\in k\fps{z}$.
\end{thm}

\begin{rmk}
Recall that a type $1$ point $x\in\A^{1,\an}_\nL$ is characterized by the condition \[-\log \diam (x)= \infty.\] 
If $\phi$ is a Puiseux series, then $\alpha(\phi)=\infty$ as well (see, e.g., Corollary~\ref{cor:approxseq}). 
When $\phi\in\nL\setminus \Puis$, letting  $(x_n)$ be the sequence of points given by Proposition~\ref{prop:approxseq}, the  condition $\alpha=\infty$ is equivalent to 
\[
\sum_{n\ge0} \frac1{m_n} \log \frac{\diam(x_n)} {\diam(x_{n+1})}=\infty~.
\]
\end{rmk}

\begin{proof}
We first work in $ \A^{1,\an}_\nL$.
Set $\rho_0 :=\max_{0\le j \le c-1} |h_j|^{1/(c-j)}<1$. 
Then, if  $\phi\in \nL$ is 
such that $\rho_0 <|\phi|<1$,  we get   
$|\phi|^c > \max_{0\le j \le c-1} \{|h_j| |\phi|^j\}$,  so that 
\begin{equation}\label{eq:r0}
|f_\diamond(\phi)|=
\bigg|
\phi^c\big({z^{1/d}}\big)+ \sum_{j=0}^{c-1} h_j\big({z^{1/d}}\big)\, \phi^j\big({z^{1/d}}\big)
\bigg|
= 
|\phi|^{c/d}.
\end{equation}
It follows that $U:= \{\rho_0< |x|<1\} \subset B^{\an}(0,1)$
is a $f_\diamond$-invariant open set and
$f_\diamond^n(x) \to \xg$ for any $x\in U$.
We   thus define 
\[
\cK_\nL := \bigcap_{n\ge 0} \Big( B^{\an}(0,1)\setminus f_\diamond^{-n}(U)\Big).\]
It is a compact totally invariant set, and 
$x\in\cK_\nL$ if and only if $f^n_\diamond$ does not converge to $\xg$. 
Since $\pi^{-1} \{\rho_0< |x|<1\} =  \{\rho_0< |x|<1\}$, we also get  that 
$x\in   \pi(\cK_\nL)$ iff $|f_\diamond^n(x)|\le \rho_0$ for all $n$, and
$\cK_\nL = \pi^{-1}(\cK)$.

\smallskip

We now proceed to show that $\diam(x) = 0$ for any  $x\in \cK_\nL$. 
In particular    $x$  must be  of type 1. 
\begin{lem}
For any   $x\le x'\in B^{\an}(0,1)\subset \A^{1,\an}_{\nL}$   of 
 positive diameter, we have   
\begin{equation}\label{eq:dist}
\left|
\log \frac{|\jac_w(f)(x')|}{|\jac_w(f)(x)|} 
\right|
\le 
(c-1)\times 
\log \frac{ \diam(x') }{ \diam(x)} 
~.\end{equation}
\end{lem}
\begin{proof}
%
%
By continuity, we may suppose that $x= \zeta(\phi,r)$ and $x'=\zeta(\phi,r')$
with $\phi\in\nL$ and $r\le r'<1$. 
Write $\jac_w(f) = c \prod_{j=1}^{c-1} (w-c_j(z))$. Then 
\[
\log \frac{|\jac_w(f)(x')|}{|\jac_w(f)(x')|}
=
\log \frac{\prod_j \max\{|\phi-c_j|,r'\}}{\prod_j\max\{|\phi-c_j|,r\}}
\le (c-1) \log \frac{r'}{r}.
\]
Indeed to  justify the inequality,  we split the products according to the position of  $\abs {\phi-c_j}$ with respect to $r'$: if $\abs{\phi- c_j}\leq r'$ then $\frac{  \max\{|\phi-c_j|,r'\}}{ \max\{|\phi-c_j|,r\}}\leq \frac{r'}{r}$ and otherwise this fraction equals 1. The desired inequality~\eqref{eq:dist} follows.
\end{proof}
%

Now, let $x\in B^{\an}(0,1)$ be of finite diameter, and 
observe  that $[x ,\xg)\cap U$ is non-empty. 
Pick any $x'\in I\cap U$, so that $\diam(f^n_\diamond(x'))\to 1$ as $n\to\infty$.
Then for every $n\geq 1$,  Lemma~\ref{lem:jac} (applied to $f^n$) yields:
\begin{align*}
\left|
\log \frac{\diam(f^n_\diamond(x'))}{\diam(f^n_\diamond(x))}
\right|
&\le 
\frac1{d^n}
\left|
\log \frac{\diam(x')}{\diam(x)}
\right|
+ 
\frac1{d^n}
\left|
\log\frac{|\jac_w(f^n)(x')|}{|\jac_w(f^n)(x)|}
\right|
\\
&\le 
\frac{1+(c^n-1)}{d^n}
\left|
\log \frac{\diam(x')}{\diam(x)}
\right| \longrightarrow 0
\end{align*}
which implies $x' \notin \cK_\nL$, as announced. 
 
Note that $\diam (\pi(x)) = \diam(x)$, hence $\cK \subset \{\diam=0\}$ as well.
We now work in $\A^{1,\an}_{k\laurent{z}}$ and 
prove $\alpha(x) = \infty$ for every $x\in \cK$. We rely on the following lemma. 
\begin{lem}
For any $x\le x'\in B^{\an}(0,1)\subset \A^{1,\an}_{k\laurent{z}}$, such that $\alpha(x)<\infty$ and for all $n\in \N$, we have
\begin{equation}\label{eq:uniform-skew}
\frac{\alpha(x')}{\alpha(x)} 
\le
\frac{\alpha(f^n_\diamond(x'))}{\alpha(f^n_\diamond(x))} 
\le 
\frac{\alpha(x)}{\alpha(x')} 
~.\end{equation}
\end{lem}
\begin{proof}
By continuity, we may suppose that $x$ and $x'$ are of type 2. 
Since $f^n_\diamond$ preserves the order, we have 
$f^n_\diamond(x')\le f^n_\diamond(x)$ hence there exists
$P\in k\fps{z}[w]$ monic of degree $\ell$ such that 
\[\alpha(f^n_\diamond(x'))= -\frac1\ell \log |P(f^n_\diamond(x'))| = -\frac1{\ell d^n} \log |(P\circ f^n)(x'))| 
\]
and similarly
$\alpha(f^n_\diamond(x))= -\frac1{\ell d^n} \log |(P\circ f^n)(x))| $.
We conclude using Lemma~\ref{lem:crucial-skew}.
\end{proof}
Suppose now that $x\le \xg$, and $\alpha(x)<\infty$, and recall that $\alpha(\xg) = 0$. Viewing 
the relation~\eqref{eq:uniform-skew} as an equicontinuity property, we infer that 
the set of points in $[x,\xg]$ lying in the basin 
of attraction of $\xg$ is both open and closed. This entails that 
  $f^n_\diamond(x) \to \xg$ as required.

\smallskip

It remains to prove the dichotomy.  
Note that, even if for convenience  we drop the `an' superscript, all the balls we consider in the proof are 
balls in the affine Berkovich line. 
Let us first work over $\A^{1,\an}_{k\laurent{z}}$.
To do this, consider the closed ball $\overline{B}_{0,0}:=\overline{B}(0,\rho_0)$. 
It follows from the previous arguments that $\cK=\bigcap_n f_\diamond^{-n}(\overline{B}_{0,0})$.
Since $f_\diamond$   is finite, type preserving  and order-preserving, 
$f_\diamond^{-n}(\overline{B}_{0,0})$ is a finite union of disjoint closed balls 
$\overline \cB_n = \{\overline{B}_{n,i}\}$ such that  $f^n_\diamond (\overline{B}_{n,i}) = \overline{B}_{0,0}$
for all $n$ and $i$. 

For any $\phi\in\cK$, consider the sequence of balls $\overline{B}_n(\phi) \in \overline \cB_n$
containing $\phi$. It is decreasing hence its intersection $\bigcap_n \overline{B}_n(\phi)$
is included in $\cK$ and
contains $\phi$. It is equal to $\{\phi\}$ because $\cK$ contains only type $1$ points. 

Suppose now that there exists an integer $n_0$ such that $\# \overline \cB_{n_0}\ge2$. 
We shall prove that $\cK$ is a Cantor set. To do so we need to show that 
$\cK$ is a compact, totally disconnected and perfect set. 
The first property is clear. Pick any point $\phi\in\cK$. 
Then for any $n$ the finite collection of balls $\overline \cB_n=\{\overline{B}_{n,i}\}$ covers $\cK$. The function $\chi$ taking value $1$ on  $\overline{B}_n(\phi)\cap \cK$
and $0$ on all other pieces $\overline{B}_{n,i}\cap \cK$ is continuous since the closed balls are disjoint, 
hence the connected component of $\phi$ in $\cK$ is included  in $\overline{B}_n(\phi)$, 
so that $\cK$ is totally disconnected. 
To see that $\cK$ is perfect, choose any open neighborhood $V$ of some $\phi\in\cK$.
For $n$ sufficiently large  $\overline{B}_n(\phi)\subset V$, and $f^n_\diamond \colon\overline{B}_n(\phi)\to \overline{B}_{0,0}$ is surjective. It follows that $\overline{B}_n(\phi)$ contains a preimage by $f^n_\diamond$ of a point in each set $\overline{B}_{m,i} \cap \cK$, and in particular at least $2$ points. 
Thus $\cK$ is perfect. 

Suppose now on the contrary that for every $n$,  $\overline \cB_{n}=\{\overline{B}_{n,0}\}$ consists of a single ball. Then $\cK$ is reduced to a singleton $  \{x\}= \bigcap_n \overline{B}_{n,0}$. To conclude we need to argue that $x$ is represented by a branch of the critical set, and that this branch is smooth and transversal to $(z=0)$.

\medskip
Observe first that we can construct a sequence $\overline \cB_{n,\nL}$ as above over  $\A^{1,\an}_{\nL}$, and the same analysis
proves that  either $\cK_\nL$ is a Cantor set or a single point $\{\phi\}$ with $\phi\in\nL$.

Suppose that $\cK_\nL$ is a singleton. Since $\cK_\nL$ is Galois invariant,  $x$ must be a point of type $1$ of (relative) multiplicity equal to $1$. We deduce that $x$ is represented by a smooth branch $\phi$ transversal to $(z=0)$.
We claim that $w-\phi(z)$ is a branch of the critical set.
In fact, from the fact that $w-\phi(z)$ is totally invariant we deduce that
$$
f^*(w-\phi(z)) = w^c+\sum_{j=0}^{c-1} h_j(z) w^j - \phi(z^d)
$$
must be of the form $(w-\phi(z))^e u(z,w)$ for a suitable $e \in \nN^*$ 
and a unit $u=\sum_{k} u_k(w) z^k$.
By computing at $z=0$ we get $w^e u_0=w^c$, from which we deduce that $e=c$ and $u_0\equiv 1$.
Assume that there exists $k\geq 1$ so that $u_k \not\equiv 0$. Take $k$ to be minimal, and let $h$ be the order of $u_k$ at $0$. Then the expression $(w-\phi(z))^c u(z,w)$ has a non-trivial monomial $z^k w^{h+c}$, while $f^*(w-\phi(z))$ does not, and we get a contradiction.
Thus we have showed that $u \equiv 1$ and that $w \circ f(z,w)=(w-\phi(z))^c + \phi(z^d)$.
Computing  the derivative, we get
$$
\jac_w(f) = c(w-\phi(z))^{c-1}\text{,}
$$
from which we deduce that $w-\phi(z)$ is a critical branch, as desired.  

To conclude the proof of Theorem~\ref{thm:invariantcantor}, it remains to exclude the possibility that 
 $\cK=\{x\}$ is a singleton and $\cK_\nL$ is a Cantor set.  For convenience, we postpone the  
  proof of this fact  until  Theorem~\ref{thm:estim-mult} is established.
%
  \end{proof}
%
%
%
 
%

\subsection{The non-Archimedean Green function}
In this section, we can work indifferently over $\A^{1,\an}_{k\laurent{z}}$
or  $\A^{1,\an}_{\nL}$. 

Recall that a subharmonic function $g\colon U\to [-\infty,\infty)$ 
on an open set $U\subset \A^{1,\an}_{k\laurent{z}}$ (resp.  in $\A^{1,\an}_{\nL}$) is 
upper-semicontinuous, not identically $-\infty$, and satisfies
the submean value inequalities, see, e.g.,~\cite[Definition~8.16]{baker-rumely}. 
Suffice it to say that the set of subharmonic functions 
is stable under taking maxima, sums, multiplication by positive constants
and decreasing limits. It contains all functions of the form
$q\log|P|$ for any $P\in k\laurent{z}[w]$ (resp. in $\nL[w]$),  and $q\in \Q^*_+$, 
and the lattice generated by these functions is dense in the space
of all subharmonic functions, see, e.g.,~\cite{stevenson} (or~\cite[Proposition~8.45]{baker-rumely}). 
Also, if $g$ is subharmonic in $U\subset \A^{1,\an}_{k\laurent{z}}$, then 
$g\circ \pi$ is subharmonic in $\pi\inv (U)\subset \A^{1,\an}_{\nL}$.

\begin{thm}\label{thm:green}
Let $f$ be a polynomial map of the form~\eqref{eq:cd}, with $2 \leq c < d$. 

Then, for every $n\in\N^*$, 
the function $g_n:x\mapsto \frac{d^n}{c^n} \log \abs{f_\diamond ^n(x)}$ is subharmonic on $B^{\an}(0,1)\subset  \A^{1,\an}_{k\laurent{z}}$.
The sequence $(g_n)_{n\geq 1}$  converges 
pointwise   to a subharmonic function 
\[g_{\na}\colon B^{\an}(0,1) \to \R_-
\text{ such that }
g_{\na} \circ f_\diamond = \frac{c}{d} g_{\na}.
\]Furthermore
the convergence is  locally uniform  on compact subsets of $B^{\an}(0,1)\setminus \cK$, in 
particular $g_{\na}$ is continuous there. 

Furthermore    
$g_{\na}(x) = \log|x|$ on an annulus  of the form $\{\rho_0< |x|<1\}$, and 
\[
\cK = \{g_{\na}=-\infty\}
~.\]

The same results hold in $\A^{1,\an}_{\nL}$ and for the corresponding function 
$g_{\na, \nL}$ we have the relation  $g_{\na, \nL} = g_{\na}\circ \pi$.  
\end{thm}

The first assertion is a consequence of the following basic lemma. 
\begin{lem}\label{lem:sh}
If $g$ is a subharmonic  function on $B^{\an}(0,1)$ then $g\circ f_\diamond$ is 
subharmonic. 
\end{lem}

\begin{proof}
For any polynomial $k\laurent{z}[w]$, 
denote by $\log_P$ the function $x\mapsto \log\abs{P(x)}$ on 
$\A^{1, \an}_{k\laurent{z}}$. 
It is enough to check that $\log_P\circ  f_\diamond$ is subharmonic for every $P$. Writing $f(z,w) = (z^d, f_2(z,w))$, for every 
$\psi\in \nL$ we have   
\begin{equation}\label{eq:sh}
(\log_P\circ  f_\diamond )(\psi) = \unsur{d} \log\abs{P\big(z^d,f_2(z, \psi(z) )\big)},
\end{equation}
which shows that $\log_P\circ  f_\diamond = \frac1d \log_{P\circ f}$ is subharmonic. 
\end{proof}

\begin{proof}[Proof of Theorem~\ref{thm:green}]
From~\eqref{eq:r0}, the density of type 1 points, and the continuity of 
$\abs{\cdot}$,  there  exists $0<\rho_0<1$ such that 
 denoting  $U = \set{\rho_0< \abs{x}<1}$, $\abs{f_\diamond(x)} = \abs{x}^{c/d}$ for any $x\in U$. 
 Hence $U$ is $f_\diamond$-invariant and we infer that for any $x\in U$ and any $n\in \N^*$, 
 $g_n(x) = \log\abs{x}$. 
 
 Pick any compact set $L\subset B^{\an}(0,1)\setminus \cK$. There exists an integer
$N$ such that $f_\diamond^N(L) \subset U$. Then for $x\in L$ and $n\geq N$ we have 
\begin{equation}
g_n(x) = \lrpar{\frac{d}{c}}^{n-N+N} \log\abs{f_\diamond^{n-N} \lrpar{f_\diamond^{N}(x)}} = 
\lrpar{\frac{d}{c}}^{N}\log\abs{ f_\diamond^{N}(x)}, 
\end{equation}
thus the sequence is locally stationary, and the local uniform convergence statement follows.  In addition, for $x\in \cK$ we have that 
$\log\abs{f_\diamond^n(x)}\leq \log \rho_0$ for every $n$, from which we infer that 
$g_n(x)\to -\infty$, hence the function $g_{\na}$ is well defined, continuous outside 
$\cK$ and $\cK = \set{g_{\na}  =-\infty}$. 

It remains to check that $g_{\na}$ is subharmonic. Outside $\cK$ this follows from the local uniform convergence.   Next, fix a large integer  $A\gg1$.  Then $\max\{g_n,-A\}$ forms a sequence of bounded subharmonic functions, hence by Hartog's theorem~\cite[Proposition~2.18]{favre-riveraletelier},
there exists a subsequence $\max\{g_{n_k},-A\}$ converging pointwise 
on the complement of type 1 points to a 
subharmonic function $g_A$. When $A$ increases, then $g_A$ decreases toward $g_{\na}$ which is hence subharmonic too.

Finally, the same proof works in $\A^{1, \an}_\nL$, and the relation $g_{\na, \nL} = g_{\na}\circ \pi$ follows from the fact that $g_n(\pi(x)) = g_{n, \nL}(x)$, as follows from the definitions. 
\end{proof}

\subsection{The invariant measure}\label{subs:invariant_measure}
We now  turn to the invariant measure $\mu_{\na}$. By default we work in the unit ball in 
$\A^{1, \an}_{k\laurent{z}}$, but  
for reasons that will appear necessary in Section~\ref{sec:structure}, we have to deal 
at the same time with $\A^{1, \an}_{\nL}$. 

Working in $B^{\an}(0,1)\subset \A^{1, \an}_{\nL}$ first, 
recall that  to any subharmonic function $g$ on $B^{\an}(0,1)$ one can associate
a positive measure $\Delta g$ such that  the identity 
$$
\Delta \log_P = \sum_{P(\phi)=0}  \delta_\phi
$$
holds for any $P\in \nL[w]$, where as before we denote $\log_P: x\mapsto \log\abs{P}_x$, and the roots are counted with multiplicity. Conversely, any probability measure $\mu$ with compact support in $B^{\an}(0,1)$ is of the form $\Delta g$, where $g$ is subharmonic and $g(x) = \log\abs{x}$ in some neighborhood of 
$\xg$: indeed  we may take $g_\mu(x) = \int \log \langle x, y \rangle \,{\mathrm{d}}\mu(y)$ such that $g_\mu(\xg)=0$. 
Here $\langle x, y \rangle = |x\wedge y|$ where $x \wedge y$ is the unique supremum of the segment $[x,y]$, see
the discussion in~\cite[\S 13.3]{benedetto:dyn1NA}, or~\cite[Prop.~8.66]{baker-rumely}. Observe that if $\mu_n\to \mu$ weakly, then 
$g_{\mu_n} \to g_\mu$ pointwise outside the set of Type 1 points.


Now, descending to $\A^{1, \an}_{k\laurent{z}}$ we can also associate 
a positive measure $\Delta g$ to any subharmonic function $g$ on $B^{\an}(0,1)$, which 
satisfies the  fundamental   identity 
\begin{equation}\label{eq:fundid}
\pi_* \Delta(g\circ \pi)  = \Delta g.
\end{equation}
For instance, if  $P\in k\fps{z}[w]$, we now have 
\begin{equation}\label{eq:fundid2}
\Delta \log_P = \sum_{P(\phi)=0} \pi_*\delta_\phi,
\end{equation}
where the roots $\phi\in\Puis$  of $P$ are counted with multiplicity.

In both situations, for  any such positive measure $\mu= \Delta g$, 
we put $f_\diamond^*\mu = \Delta (g\circ f_\diamond)$, which is a positive measure 
by Lemma~\ref{lem:sh}. The operator $f^*_\diamond$ is continuous for the weak convergence of measures, and 
Formula~\eqref{eq:sh} entails that if $\mu$ is a probability measure having compact support in $B^{\an}(0,1)$, 
then $f_\diamond^*\mu$ is a positive measure with compact support in the open unit ball and has mass $c/d$. 

By duality, we infer the existence of a push-forward operator on the space of 
continuous function with compact support on $B^{\an}(0,1)$
so that 
\begin{equation}\label{eq:biz}
\sup \abs{(f_{\diamond})_*h}\le \frac{c}{d} \sup|h| \text{ , and }
\int h  \, {\mathrm{d}} (f^*_{\diamond}\mu) = \int \lrpar{(f_{\diamond})_*h}  {\mathrm{d}}\mu~,
\end{equation}
for any continuous $h$. Namely, $(f_{\diamond})_*h(x) = \bra{f_\diamond^*\delta_x,h}$.

\begin{rmk}
The push-forward operator can be obtained explicitely as follows.
Over $\nL$ recall that  $f_{\diamond} = \hat{f} \circ \tau$
where $\hat{f}$ is a polynomial map of degree $c$  defined over $\nL$, and 
$\tau(\phi(z)) = \phi(z^{1/d})$.
Then we have
$(f_{\diamond})_*h (x) = \frac1d \sum_{f_\diamond(y)=x} e(y) h(y)$
where $e(y)\in\{1, \cdots, c\}$ is the local degree of $\hat{f}$ at
$\tau(y)$ as defined in~\cite{favre-riveraletelier,jonsson:berkovich,benedetto:dyn1NA}. 
A similar description can be obtained over $k\laurent{z}$. 
\end{rmk}

 \begin{lem}\label{lem:pushpull}
For any compactly supported measure $\mu$ on $\A^{1,\an}_{k\laurent{z}}$ (resp. 
$\A^{1,\an}_{\nL}$), we have 
$\displaystyle(f_\diamond)_*(f_\diamond^*\mu)  = \frac{c}{d} \mu$.
\end{lem}

\begin{proof}
By density and linearity, we only need to check the formula for Dirac masses at type 1 points, and may work over $\nL$. 
Then $\mu= \Delta \log |w- \phi|$ for some $\phi\in \nL$, and 
$f_\diamond^*\mu = \Delta \frac1d \log |(w-\phi) \circ f|$ whose support equals $f_\diamond^{-1}\{\phi\}$
so that $(f_\diamond)_*(f_\diamond^*\mu)  = C \mu$
for some positive constant $C$. 
Computing the mass on both sides leads to $C= c/d$
which completes the proof. 
\end{proof}

\begin{thm}\label{thm:ergo-meas}
Keeping  the assumptions and notation 
 of Theorem~\ref{thm:green}, define $\mu_{\na} := \Delta g_{\na}$.
This is a $f_\diamond$-invariant  probability measure  
 satisfying  $f_\diamond^* \mu_{\na} = \frac{c}d \mu_{\na}$.
Its support is equal to $\cK$. 

For any point $x\in B^{\an}(0,1)$, we have the weak convergence  of measures
\begin{equation}\label{eq:plbck}
 \frac{d^n}{c^n}( f^{n}_\diamond )^*\delta_x \to \mu_{\na}~.
 \end{equation}
Furthermore,  the measure $\mu_{\na}$ is mixing (hence ergodic).

Likewise, in $\A^{1, \an}_{\nL}$ we define $\mu_{\na, \nL} := \Delta g_{\na, \nL}$, which satisfies 
$\mu_{\na} = \pi_*\mu_{\na, \nL}$, and the same properties hold. In addition,  
the measure $\mu_{\na, \nL}$ is balanced, that is, for  any open ball 
$B$ in which $f_\diamond$ is injective, we have $  \mu(f_\diamond(B)) = c\mu(B)$.  
\end{thm}

\begin{proof}
Since $g_{\na}(x) = \log|x|$ near $\xg$, $\mu_{\na}$ is a probability measure. 
Its support contains $\cK$ because $g_{\na}= -\infty$ on $\cK$. 
Conversely, pick any point $x$ outside $\cK$. Then for large  $n$, 
 $f^n_\diamond(x)$
is close to $\xg$ thus $g$ is harmonic in a neighborhood of  $f^n_\diamond(x)$. 
Since $f_\diamond$ preserves harmonic functions, it follows that $g_{\na}$ is harmonic
near $x$, and we conclude that $\mathrm{Supp}(\mu_{\na}) =\cK$. 

By definition $\mu_{\na}$ satisfies $f_\diamond^* \mu_{\na}  = \frac{c}{d} \mu_{\na}$, therefore Lemma~\ref{lem:pushpull} implies that $\mu_{\na}$ is invariant.

Let $\mu$  be  any compactly supported probability measure  in $B^{\an}(0,1)$
and write $\mu =\Delta h$ with $h$ subharmonic and 
$h(x) = \log|x|$ in a $f_\diamond$-invariant neighborhood $U$ of $\xg$. 
Arguing as in the proof of Theorem~\ref{thm:green}, we deduce 
that $h_n:=\frac{d^n}{c^n}h\circ f_\diamond^n = g_n$ on $f_\diamond^{-n}(U)$, hence
$h_n$ converges uniformly on compact sets to $g_{\na}$ on the complement of $\cK$. 
But  since $h<0$ and $d>c$ is also clear that $h_n\to-\infty$ on $\cK$,  
 hence $h_n\to g_{\na}$ pointwise everywhere. 
We conclude that $\Delta h_n\to \mu_{\na}$, i.e., $\frac{d^n}{c^n}f^{n*}_\diamond \mu\to \mu_{\na}$. This implies~\eqref{eq:plbck}.

 To prove the mixing property, 
pick any compactly supported continuous function $\psi$ on $B^{\an}(0,1)$. 
By the convergence~\eqref{eq:plbck} and  the  duality relation in~\eqref{eq:biz}, 
  $\frac{d^n}{c^n} (f^n_{ \diamond})_* \psi(x)$ converges pointwise to 
  $\int \psi \,{\mathrm{d}}\mu_{\na}$
for any $x\in B^{\an}(0,1)$.  Without loss of generality, assume $\int \psi \,{\mathrm{d}}\mu_{\na}=0$.
For any other continuous function $\varphi$, 
we get 
\begin{align*}
\int \lrpar{\varphi\circ f_{ \diamond}^n }\psi \,{\mathrm{d}}\mu_{\na}
&=
\bra{\frac{d^n}{c^n} (f_{ \diamond})^* (\varphi\mu_{\na}), \psi}
= \bra{\varphi\mu_{\na} , \frac{d^n}{c^n} (f_{\diamond})_* \psi}\\
& = 
 \int \lrpar{\frac{d^n}{c^n} (f_{\diamond})_*\psi(x)}  \, \varphi (x) \,{\mathrm{d}}\mu_{\na}(x)
\underset{n\to\infty}\longrightarrow 0,\end{align*}
 where in the last step we use the fact that 
  $\frac{d^n}{c^n} (f^n_{\diamond})_* \psi$ is uniformly bounded, and dominated convergence. Therefore $\mu_{\na}$ is mixing, as announced. 

The previous arguments  can be repeated mutatis mutandis in $\A^{1, \an}_{\nL}$, 
and we get the same properties. Note that the relation $\mu_{\na} = \pi_*\mu_{\na, \nL}$  follows 
from $g_{\na, \nL} = g_{\na}\circ \pi$ and~\eqref{eq:fundid}. It remains to establish   the balancing property of $\mu_{\na, \nL}$.  For this, we fix a type 1 point $x $, represented by a series 
$\phi(z)$, and observe that 
\begin{equation}\label{eq:balanced}
\mu_n := \frac{d^n}{c^n} (f^n_{\nL\diamond}) ^* \delta_x = \frac{d^n}{c^n} \Delta \lrpar{\log_P \circ f^n_{\nL\diamond}} = 
 \unsur{c^n}\sum_{f^n_{\nL\diamond}(y) = x} \delta_y ,
 \end{equation}  where $P(w) = w- \phi(z)$. If $B$ is as in the statement of the theorem, then, $f_{\nL\diamond}$ being open, 
  $f_{\nL\diamond}(B)$ is an open ball as well, and $\mu_{\na, \nL}(\fr B) = \mu_{\na, \nL}(\fr f_{\nL\diamond} (B))=0$ because $\cK_{\nL}$ consists of type 1 points. Hence $\mu_n(B)$ converges to $\mu_{\na, \nL}(B)$  and likewise for   $f_{\nL\diamond}(B)$. Finally, by~\eqref{eq:balanced}, $\mu_n(f_{\nL\diamond}(B)) = c\mu_{n+1}(B)$ for all $n$, and we conclude 
   by letting $n\to\infty$. 
\end{proof}


\section{Upper bounds for the  multiplicity on $\cK$}\label{sec:upper}

We still denote by  $  f(z,w) = \left(z^d, w^c + \sum_{j=0}^{c-1} h_j(z) w^j\right)$ 
   a polynomial map as in~\eqref{eq:cd}, and recall that the Jacobian determinant is equal to
$\jac(f) = d z^{d-1} \jac_w(f)$ where 
\[\jac_w(f) =  cw^{c-1} + \sum_{j=1}^{c-1} j h_j(z) w^{j-1}\in k\fps{z}[w].\]
We factor $\jac_w(f)$ in the field of Puiseux series:
\[
\jac_w(f) = c\prod_{i=1}^{c-1} (w - c_i(z))
\]
where $c_i(z)\in\Puis$. Since $h_j(0)=0$ for all $j$, note that $c_i(0)=0$ as well.

The critical points $c_i \in \CritB$ are not necessarily distinct, and we introduce 
\begin{equation}\label{eq:critslope}
J(c_i):=\ord_{c_i}\jac_w(f),
\end{equation}
which is a positive integer.

It will be convenient to use the following notation:
if $\alpha=p/q$ is a rational number with $p$ and $q$ coprime, then we denote by  
$q(\alpha):=q$   its denominator.

To any $c_i$ we associate the integer
\[
\nu(c_i):=q\lrpar{\frac{d}{1+J(c_i)}} \in \nN_{\geq 1}\text{.}
\]
Note that if $d$ is even and   $c$ is a simple critical point then $\nu(c) = 1$. 
We denote by $\CritB^+\subset\CritB$  
the set of critical points $c_i$ such that  $\nu(c_i) \geq  2$

As usual,  we denote by $\omega(x)$ the  limit set of a point $x$, that is,  
\[\omega(x) = \bigcap_{n\ge0} \overline{\{f_\diamond^m(x), \, m\ge n\}}~.\]
Our main result can be stated as follows. Here we work in $\A^{1, \an}_{k\laurent{z}}$. 

\begin{thm}\label{thm:estim-mult}
Let $f$ be a polynomial map of the form~\eqref{eq:cd} with $2 \leq c < d$, and recall that 
  $\CritB^+=\{c_i\in \CritB, \nu(c_i)\ge2\}$.
\begin{enumerate}
\item
Any point $x\in\cK$ such that $\omega(x)\cap \CritB^+=\emptyset$
has finite multiplicity, that is, it  is a rigid point).
\item
Suppose that 
$\CritB^+\cap\omega(c) =\emptyset$ for every  $c\in\CritB^+$. Then there exists $m_0\in\N^*$
such that $\cK \subset \{m \leq m_0\}$. In particular, any point in $\cK$ is rigid.
\end{enumerate}
\end{thm}

As an immediate consequence we get 
 the following corollary, which holds in particular when $\CritB^+=\emptyset$.

\begin{cor}\label{cor:mult}
If  $c\notin \cK$ for every $c\in \CritB^+$, then there exists an integer $m_0$ such that
$\cK \subset \{m \le m_0\}$. 
\end{cor}

Using Theorem~\ref{thm:estim-mult}, 
we can now complete the proof of a theorem from the previous section. 

\begin{proof}[End of proof of Theorem~\ref{thm:invariantcantor}.]
Suppose that $\cK_\nL$ is a Cantor set, and $\cK=\{x\}$ is a single point.
Since $\cK_\nL$ is invariant by the Galois action, 
we deduce that $\cK_\nL$ is the Galois orbit of any lift $y$ of $x$ to $\A^{1,\an}_{\nL}$. 
By assumption $\cK_\nL$ is infinite, hence $m(x)=+\infty$, in particular $x$ is not a Puiseux series. 
Yet, the critical branches are Puiseux series, so 
 all critical orbits belong to the basin of attraction of $\xg$, and by Corollary~\ref{cor:mult}, the multiplicity is uniformly bounded in $\cK$, a contradiction.
\end{proof}

\subsection{The Jacobian formula (II)}\label{sec:jacII}

It is convenient to set $A(x):= 1-\log\diam(x)$ for any non rigid point $x\in  \A^{1,\an}_{\nL}$, and work systematically over $\nL$.
With this notation, ~\eqref{eq:jac} is equivalent to the   formula: 
\begin{equation}\label{eq:jacval}
A (f_\diamond(x)) = \frac1{d}\left(A(x) + g_{\jac(f)}(x)\right)~,
\end{equation}
where  $g_{\jac(f)} (x):=  -\log \abs{\jac(f)(x)}$. Indeed,  $\jac(f)= d z^{d-1} \jac_w(f)$ so 
$\log\abs{\jac(f)} = -(d-1) + \log \abs{\jac_w(f)}$, and the formula follows from \eqref{eq:jac}.
It can be checked that the 
 function $g_{\jac(f)}$  is locally constant outside the convex hull of $\CritB \cup\{\xg\}$
which we denote by $\CritT$, and call the critical tree.

Pick any rigid point $\phi$ in the unit ball, and, for any positive real number $t>0$,
consider the point $\zeta(\phi,e^{-t})$. 
Then we have 
\[g_{\jac(f)}(\zeta(\phi,e^{-t}))= \sum_i  -\max(\log |\phi-c_i|, -t)+(d-1)~.\]
Observing that $A(\zeta(\phi,e^{-t})) = 1 +t$, we obtain the following result.

\begin{prop}\label{prop:derivJ}
Let $\phi\in B(0,1)\subset \nL$ be any rigid point. For any $\tau\in [1,+\infty]$ denote by $x_\tau$ the unique point
in the segment $[\phi,\xg]$ such that $A(x_\tau)=\tau$ (so that $\phi=x_\infty$ and $\xg=x_1$). 

The function $\tau\mapsto g_{\jac(f)}(x_\tau)$ is a concave piecewise affine function with integer slope. 
More precisely, if $x= x_{\tau_0}$, the left  derivative 
\[J(x) :=\left. \frac{d g_{\jac(f)}(x_\tau)}{d\tau}\right|_{\tau=\tau_0}\] 
is an integer, equal to the number of critical curves $c_i$ such that $c_i < x_{\tau_0}$. 
\end{prop}  

We call $J(x)$ the \emph{critical slope} at $x$.
Notice that the notation $J(x)$ is consistent with the one used in \eqref{eq:critslope}.
If the critical slope is constant on a segment $(x_0,x_1)$ with $x_0<x_1$, 
then the function 
$g_{\jac(f)}$ is affine on $[x_0,x_1]$ with respect to the parametrization by $A$. 
The following consequence of this discussion, together with~\eqref{eq:jacval}, will be repeatedly used. 

\begin{cor}\label{cor:jacval}
If the critical slope is constant equal to $J$ on an interval  $(x_0,x_1)$ with $x_0<x_1$, then 
\[\frac{A(f_\diamond(x_1)) - A(f_\diamond(x_0))}{A(x_1) - A(x_0)}= \frac{J+1}{d}~.\]
\end{cor}

\subsection{The generic multiplicity}\label{ssec:generic-mult}


We now work in $\A^{1,\an}_{k\laurent{z}}$. 
Recall that we defined the multiplicity of a point in \S\ref{sssec:multicapa}. 
Pick any type $2$ point $x\in B^{\an}(0,1)$. We define the \emph{generic multiplicity}
of $x$ by the formula
\[
b(x)= \lcm \{ m(x), q(A(x))\}~,
\]
where $q(s)= b$ if $s = a/b$ with $a\in \Z, b\in \N^*$ and $a$ and $b$ coprime.
Observe that $m(x)$ divides $b(x)$ by definition. A geometric interpretation of this quantity, which also explains the terminology, will be given in \S~\ref{subs:model}.

To understand the behaviour of the generic multiplicity, we first extend the multiplicity 
function defined on $\A^{1,\an}_{k\laurent{z}}$ in~\S\ref{sssec:multicapa} to Puiseux series, and then to open and closed balls. 

Pick any $\phi \in \nL$. If $\phi$ is not a Puiseux series we set $m(\phi)=\infty$; 
otherwise $m(\phi)$ is the minimal integer $q$ such that $\phi \in k\fps{z^{1/q}}$. 
In the latter case, if $C$ is the formal curve associated with $\{\phi=0\}$, then
$m(\phi) = (C\cdot\{z=0\})_0$.

Let $B$ be any (Berkovich) ball in $B^{\an}(0,1)$, which may be closed or open. 
We define $m(B) = \min \{m(\phi), \, \phi \in \nL, \, \pi(\phi) \in B\}$. 
Recall that a closed ball $\overline{B}$ has a unique boundary point $x$, and when this point 
is of type 2, then $\overline{B} \setminus \{x\}$ is a union of disjoint open balls $B_i$
such that $\partial B_i=\{x\}$ (corresponding to the tangent directions at $x$ pointing towards $B$).

\begin{thm}\label{thm:gen-mult}
Pick any type 2 point $x\in  B^{\an}(0,1)\subset \A^{1,\an}_{k\laurent{z}}$, and let $\overline{B}$
be the unique closed ball whose boundary is equal to $x$. 
We have $m(\overline{B}) = m(x)$. 

Moreover, the following holds. 
\begin{itemize}
\item
If $m(x)=b(x)$, then for any open ball $B$ whose boundary equals $x$, 
we have $m(B) = m(x)$. 
\item
If $m(x) < b(x)$, then there exists a unique open ball $B_0$
whose boundary equals $x$   such that $m(B_0)= m(x)$.
For all other open balls $B$ with $\partial B= \{x\}$, we have
$m(B)= b(x)$. 
\end{itemize}
\end{thm}

\begin{cor}{\cite[Proposition~3.44]{favre-jonsson:valtree}}\label{cor:gemult}
Suppose $x\in\overline{B}^{\an}(0,1)\subset\A^{1,\an}_{k\laurent{z}}$, and 
consider the decreasing sequence of  type 2 points
$x_0= \xg$, $x_{n+1}<x_n$, $\lim_n x_n=x$
  defined in Proposition~\ref{prop:approxseq}. 

Then for each $n\ge1$, we have the equality $b(x_n)=m(x_{n-1})$.
\end{cor}
\begin{proof}
Indeed, the set of points $y\le x_{n-1}$ is a closed ball
$\overline{B}$.
Consider the open ball $B$ containing $x_n$ and of the same diameter as $\overline{B}$. 
Pick any Puiseux series $\phi$
with $\pi(\phi)\in B$ and $m(\phi)= m(B)$.
Then $m(B)$ is equal to the multiplicity of all $y\in (x_{n-1}, x_n]$
sufficiently close to $x_{n-1}$. Since the multiplicity is constant on this interval
equal to $m(x_n)$, we conclude that $m(B) = m(x_n) >m(x_{n-1})$. 
The previous theorem then implies that $b(x_{n-1})= m(B)= m(x_n)$. 
\end{proof}

\begin{proof}[Sketch of proof of Theorem~\ref{thm:gen-mult}]
We may suppose that $x = \pi(\zeta(\phi,r))$
where $\phi\in \Puis$ is a Puiseux series, and $r = e^{-t}$ with $t\in \Q^*_+$. 
Observe that $\overline{B}(\phi,r) = \overline{B}(\psi,r)$ (resp. $B(\phi,r) = B(\psi,r)$) iff $|\phi-\psi| \le r$
(resp. $|\phi-\psi| <r$) iff $\ord_z(\phi  - \psi)\geq t$ (resp. $\ord_z(\phi  - \psi)> t$). 
We may thus suppose that $\phi(z)= \sum_{n=0}^N a_n z^{\beta_n}$
with $0\le \beta_0 < \cdots < \beta_N< t$.
Any open ball  having $x$ as   boundary points  is then of the form $\pi (B(\phi+ cz^t,r))$
for some $c\in \C$. 

Using~\eqref{eq:def-mult}, we get
\[m(\overline{B}(\phi,r)) = \lcm\{q(\beta_n), 0 \le n \le N\}\]
and
\[m(B(\phi+ cz^t,r)))= \lcm \{q(t), q(\beta_n), 0 \le n \le N\}~.\]
The first identity together with the description of the Galois action on Puiseux series implies that 
$m(\overline{B}(\phi,r)) = m(x)$ (recall that  $m(x)$ is by definition 
the cardinality of $\pi\inv(x)$).
The remaining statements 
follow  from the previous analysis, using the  
 relation  $A(\zeta(\phi,e^{-t}))= 1+t$, which implies  $q(t) = q(A(x))$. 
\end{proof}

\subsection{Estimates on the multiplicity}\label{ssec:mult}
In a dynamical context, the key estimate on multiplicities is given by the following result. 

\begin{prop}\label{prop:123jac}
Pick $x_1\le x_0\in B^{\an}(0,1)$, and suppose that  the critical slope is constant equal to $J$ on the 
interval $(x_1,x_0)$. Set $\nu=q\big(\tfrac{d}{J+1}\big)$.
Then for any type $2$ point $x \in [x_1,x_0]$ we have
\[
b(x)~\big|~\lcm\Big\{ b(x_0), \nu b(f_\diamond(x)), \nu b(f_\diamond(x_0))\Big\}
\]
\end{prop}

The proof relies on the following lemma.

\begin{lem}\label{lem:multbackjac}
Let  $x_1\le x_0\in B^{\an}(0,1)$ be such that  the critical slope is constant equal to $J$ on the interval $(x_1,x_0)$, and the multiplicity is constant on  $[f_\diamond(x_1), f_\diamond(x_0))$.

Then the multiplicity is constant on $[x_1,x_0)$.
\end{lem}

%
%

\begin{proof}
For any point $x'\in [x_1,x_0)$, we have:
\begin{align*}
\frac{\alpha(f_\diamond(x)) - \alpha(f_\diamond(x_0))}{\alpha(x) - \alpha(x_0)}
=
\frac{\alpha(f_\diamond(x)) - \alpha(f_\diamond(x_0))}{A(f_\diamond(x)) - A(f_\diamond(x_0))}
\times
\frac{A(f_\diamond(x)) - A(f_\diamond(x_0))}{A(x) - A(x_0)}
\times
\frac{A(x) - A(x_0)}{\alpha(x) - \alpha(x_0)}
\end{align*}
By assumption,   
\[
\frac{\alpha(f_\diamond(x)) - \alpha(f_\diamond(x_0))}{A(f_\diamond(x)) - A(f_\diamond(x_0))}
\]
is constant equal to the inverse of the multiplicity on  $[f_\diamond(x_1), f_\diamond(x_0))$ 
(see Proposition~\ref{prop:approxseq}).
Since $x_0$ (hence the whole segment $[x_1,x_0]$) is outside the critical tree
by   Corollary~\ref{cor:jacval} we have
\begin{equation}\label{eq:aalpha}
\frac{A(f_\diamond(x)) - A(f_\diamond(x_0))}{A(x) - A(x_0)}= \frac{J+1}{d}~.
\end{equation}
By Corollary~\ref{cor:approxseq},  $A$ is convex in $\alpha$ hence the function 
\[x\mapsto \frac{A(x) - A(x_0)}{\alpha(x) - \alpha(x_0)}\]
is non-decreasing. On the other hand the function 
\[x\mapsto 
\frac{\alpha(f_\diamond(x)) - \alpha(f_\diamond(x_0))}{\alpha(x) - \alpha(x_0)}
\]
is non-increasing, because 
we can write 
\[
\frac{\alpha(f_\diamond(x)) - \alpha(f_\diamond(x_0))}{\alpha(x) - \alpha(x_0)}
= \unsur{\deg_w(P)}
\frac{-\log|P\circ f(x)| +\log|P\circ f(x_0)|}{\alpha(x) - \alpha(x_0)},
\]
where $P\in k\fps{z}[w]$ is such that $\alpha(x')=-\log|P(x')|/\deg_w(P)$
 for every $x'\in  [f_\diamond(x_1), f_\diamond(x_0))$
as in Remark~\ref{rmk:computeskewness},  and $x\mapsto -\log|\Phi\circ f(x)|$ is concave in $\alpha$
by~\cite[Prop 3.25]{favre-jonsson:valtree}. It follows that both functions 
$\frac{\alpha(f_\diamond(x)) - \alpha(f_\diamond(x))}{\alpha(x) - \alpha(x_0)}$ and $\frac{A(x) - A(x_0)}{\alpha(x) - \alpha(x_0)}$
are constant. Since the multiplicity on $[x_1,x_0)$ is equal to the latter, we conclude that it is
constant on this interval.
\end{proof}

\begin{proof}[Proof of Proposition~\ref{prop:123jac}]
Since $f_\diamond$ is strictly order preserving (see Equation~\eqref{eq:aalpha}) 
and continuous,  the segment $[x,x_0]$ is mapped bijectively onto
the segment $[f_\diamond(x),f_\diamond(x_0)]$. 
Decompose the latter segment
$x'_0= f_\diamond(x_0) >x'_1 >\cdots > x'_N= f_\diamond(x)$
so that the multiplicity is constant on $(x'_n,x'_{n+1}]$ equal to $m'_{n+1}$
and $m'_n$ strictly divides $m'_{n+1}$.

Denote by $x_n$ the preimage of $x'_n$ on the segment 
$[x,x_0]$, and for consistence denote $x=x_N$. By the preceding lemma,
the multiplicity is constant on $(x_n,x_{n+1}]$ equal to $m_{n+1}$. 
Note that $m_n$ may be equal to $m_{n+1}$.

Let us prove  by induction on $n\leq N$  that 
\begin{equation}\label{eq:inductionb}
b(x_n) \text{ divides } \lcm\{b(x_0), \nu b(f_\diamond(x_0)), \nu  b(f_\diamond(x_n))\},
\end{equation}
which corresponds to the desired statement for $n=N$.
For this,  we use the observation:
\begin{equation}\label{eq:ZEr}
\text{for every } y\leq x, \  q(A(y))~ |~\lcm \{q(A(x)), \nu q (A(f_\diamond(x)) ), \nu q (A(f_\diamond(y)) ) \} \text{,}
\end{equation}
 which follows from the identity 
 $A(y)= A(x)+ \frac{d}{J+1} \big(A(f_\diamond(y))-A(f_\diamond(x))\big)$ (see~\eqref{eq:aalpha}).

Suppose first that $n=1$. 
Then, we have
\begin{align*}
b(x_1)
&
=\lcm \{m(x_1), q(A(x_1))\}
\\
& 
|\lcm\{m(x_1),  q(A(x_0)), \nu q(A(f_\diamond(x_1)),  \nu q(A(f_\diamond(x_0))\}
\\
& 
|\lcm\{b(x_0),  \nu b(f_\diamond(x_1)),  \nu b(f_\diamond(x_0))\}
\end{align*}
where in the second line we use the fact that by definition 
$q(A(\cdot))\vert b(\cdot)$ and in the third line we use the fact that 
 either $m(x_1)=m(x_0)$, or $m(x_1)=b(x_0)$, see Corollary~\ref{cor:gemult}.
 
Suppose now that~\eqref{eq:inductionb} is true for some $1\leq n<N$.
The same argument gives
\begin{align*}
b(x_{n+1})
&
=\lcm \{m(x_{n+1}), q(A(x_{n+1}))\}
\\
& 
|\lcm\{m(x_{{n+1}}),  q(A(x_n)),  \nu q(A(f_\diamond(x_n))),  \nu q(A(f_\diamond(x_{n+1})))\}
\\
& 
|\lcm\set{b(x_n),  \nu b(f_\diamond(x_n)),  \nu b(f_\diamond(x_{n+1}))}
\end{align*}
since either $m(x_{n+1})=m(x_n)$, or $m(x_{n+1})=b(x_n)$. 
We conclude the proof using the induction hypothesis and the fact that   
$b(f_\diamond(x_n))|b(f_\diamond(x_{n+1}))$: indeed for $n\geq 1$ we have 
\[b(f_\diamond(x_n)) = b(x'_n) = m(x'_{n+1}) | b(x'_{n+1}).\]
The proof is complete.  
\end{proof}

%
%
%

\subsection{Proof of Theorem~\ref{thm:estim-mult}}

Recall from~\S\ref{ssec:invariantCantor} that we constructed a nested family $(\overline \cB_n)_n$  each consisting of finitely many closed (Berkovich) balls $\overline \cB_n=\{\overline{B}_{n,i}\}$, such that $\cK = \bigcap_n \bigcup \overline \cB_n$, and such that for any $B \in \overline \cB_n$,   $f_\diamond(B)\in\overline \cB_{n-1}$  and
$f_\diamond^{-1}(B)$ is an union of balls in $\overline \cB_{n+1}$. Here $\bigcup \overline \cB_n := \bigcup_{i} \overline{B}_{n,i}$ stands for the union of all the elements in $\overline \cB_n$.

\smallskip

Let us  prove Assertion~(1).  
So, assume   that the $\omega$-limit set of the point 
$x$ does not contain any critical point in $\CritB^+$.
Replacing $x$ by one of its iterates, we may suppose that the closure $O^+(x)$ of the orbit of $x$ is compact and avoids $\CritB^+$. 
Denote by $\cT^+$ the convex hull of $\CritB^+$ and $\xg$. 
Any point in $O^+(x)$ is contained in some ball in  $\overline \cB_n$ that is disjoint from $\cT^+$, 
therefore we can  find an integer 
 $N$ and a finite set of balls $\overline{B}_1, \cdots, \overline{B}_m\in \overline \cB_N$ 
disjoint from  $\cT^+$, and whose union contains $O^+(x)$. For the sake of the argument, assume 
further that for every $j$, $f_\diamond (\overline{B}_j)$ is disjoint from $\cT^+$ (that is, decrease possibly $N$ by 1),
and assume that $N$ is so large that each ball in $\overline \cB_{N-1}$ contains at most one critical point of $\CritB\setminus \CritB^+$.
 Let $y_1, \cdots, y_m$ be  
the respective boundary points of the $\overline{B}_j$. 
For any $n\ge N$, let $\hat x_n$ be the boundary point of the unique  ball in $\overline \cB_n$ containing $x$. 
Then for any $\ell \le n-N$, the type $2$ point  $f^\ell_\diamond(\hat x_n)$ is the boundary point of a ball in $\overline \cB_{n-\ell}$, and in particular
$f^{n-N}_\diamond(\hat x_n)\in \{y_1, \ldots, y_m\}$. 

Set $Q= \lcm\{b(y_j),b(f_\diamond(y_j)), \ 1\leq j\leq m\}$.
We inductively apply 
 Proposition~\ref{prop:123jac} to the intervals 
 $(f^\ell_\diamond(\hat x_n), y_{j(\ell)})$, where $y_{j(\ell)}$  is the 
  unique point $y_i\ge f^\ell_\diamond(\hat x_n)$, 
  and observe that 
   $\nu =1$ on these intervals    
 because they are disjoint from $\cT^+$ and contain at most one point of $\CritB\setminus \CritB^+$. 
 We thus  obtain that 
\[
b(\hat x_n)|\lcm\big\{Q,b\big(f^\ell_\diamond(\hat x_n)\big)\big\}
\]
for any $\ell \leq n-N$. For $\ell=n-N$ we obtain 
 $b(\hat x_n)|Q$. 
In particular, the sequence $b(\hat x_n)$ is stationary
for large $n$, which  implies that $x$ has  multiplicity at most $ Q$. 

\medskip

%
%
%
%

Let us now  establish (2).
Since $\CritB^+\cap\omega(c) =\emptyset$ for every  $c\in\CritB^+$, 
 one can find a neighborhood $U$ of $\CritB^+$
and an integer $N$ such that for every $c\in \CritB^+$, 
$f^n_{\diamond}(c)\notin U$ for all $n\ge N$.
We may suppose that $U$ is a finite union of closed balls of $\overline \cB_M$ for some $M$. 
It follows that if $\overline{B}\in \overline \cB_n$ contains a critical point and $n\ge M$, then
$f^k_\diamond(\overline{B})\in\overline \cB_{n-k}$ does not contain any critical point for any $N\le k \le M-n$. 
We shall also assume that $M$ is so large  that the intersection of $\CritT$ with any ball $\overline{B}\in \overline \cB_M$ is either empty or connected.

We introduce some notation. 
 For any point $x\in\cK$ and any $n\in\N$, 
  let $b_n(x)$ be the
generic multiplicity of the unique boundary point of the 
 ball $\overline{B}_n(x)\in\overline \cB_n$ containing $x$ 
 (so that with the above notation, $b_n(x) = b(\hat x_n)$). 
Observe that $f^k_\diamond (\overline B_{n}(x))=\overline B_{n-k}(f^k_\diamond(x))$
for all $k\le n$. 
For any $n\ge 1$, we define
\[Q_n= \lcm_{x = \partial \overline{B}}\{b(x), b( f_\diamond(x))\} 
~,\]
where $\overline{B}$ ranges over all balls in $\overline \cB_n$.

 The theorem will follow if we find an upper bound on 
$b_{n+M}(x)$ independent on $x\in\cK$ and $n$. 
Pick any $x\in\cK$ and any integer $n\ge0$.
Let $n_0$ (resp. $n_1$) be the least (resp. greatest) integer $k\le n$ such that $f^{k}_\diamond (\overline B_{n+M}(x))$ contains a critical point. 
We know that $n_0 \le n_1 \le n_0+N-1$. 
 
For any $j\in \{0,\ldots, n\}$, we apply Proposition~\ref{prop:123jac} to \[x_1=\partial \overline B_{n-j+M}(f^j_\diamond(x))=\partial f^{j}_\diamond (\overline B_{n+M}(x))\]
and to $x_0=\partial \overline{B}_M(f^j_\diamond(x))$ the boundary point of the unique ball in $\overline \cB_M$ containing $x_1$. 
  
When $ j< n_0$ or when $j> n_1$, then the ball $f_\diamond^j(\overline B_{n+M}(x)) =\overline B_{n-j+M}(f^j_\diamond(x))$ does not contain any   point of $\CritB^+$  and at most one critical point of 
$\CritB\setminus \CritB^+$, so
  \[
 b_{n-j+M}(f^j_\diamond(x)) |\lcm\{ b_{n-j-1+M}(f^{j+1}_\diamond(x)), Q_M\}~.
 \]
 
When $j \in \{n_0, \ldots, n_1\}$ the critical slope is  
 constant (by our choice of $M$) on the segment $[x_0,x_1]$ and bounded  from above by $c-1$, so  
\[
b_{n-j+M}(f^j_\diamond(x)) |\lcm\{ L b_{n-j-1+M}(f^{j+1}_\diamond(x)), LQ_M\}~,
\]
where $L$ is the least common multiple of $q\big(\tfrac{d}{1+j}\big)$ when $j$ ranges among all integers $0 \leq j \leq (c-1)$. 
Since $|n_1-n_0|\le N$, we conclude that 
\[
b_n(x) | L^N Q_M
\]
as required.

\section{Points of infinite multiplicity}\label{sec:infinite}

\begin{thm}\label{thm:infinity-mult}
Suppose that $\cK$ is not reduced to a singleton, and that there exists a periodic critical point $c_0$ with $\nu(c_0) \geq 2$.
Then $\cK$ contains both rigid and non-rigid points.
Moreover, the supremum of the multiplicity of rigid points $x\in \cK$ is infinite.
\end{thm}

\begin{rmk}
It follows from the equidistribution statement~\eqref{eq:plbck} in Theorem~\ref{thm:ergo-meas}
 that the set of non-rigid points is dense in $\cK$. Since the multiplicity is lower-semicontinuous, and 
 the  preimages of $c_0$ are rigid and dense in $\cK$, we also deduce that for any integer $M$
 the set of rigid points in  $\cK \cap \{ m\ge M\}$ is also dense in $\cK$.

Observe that  $\{m<\infty\}$ is a closed $f_\diamond$-invariant subset, hence $\mu_{\na}(m<\infty)$ is either $0$ or $1$
by the ergodicity of $\mu_{\na}$.
 \end{rmk}

\begin{quest}
Under which dynamical circumstances can we ensure that $\{m<\infty\}$ has zero (resp. full) measure?
\end{quest} 

\begin{quest}
Does Theorem~\ref{thm:infinity-mult} extend to maps admitting a critical point $c\in\CritB^+$
such that $\omega(c) \cap \CritB^+\neq \emptyset$?
\end{quest} 

\begin{proof}
As in the previous section, we consider the nested family of finitely many closed (Berkovich) balls $\overline \cB_n=\{\overline{B}_{n,i}\}$, such that $\cK = \bigcap_n \bigcup \overline \cB_n$. 
Recall that  for any $\overline{B} \in \overline \cB_n$, we have $f_\diamond(\overline{B})\in\overline \cB_{n-1}$, and
$f_\diamond^{-1}(\overline{B})$ is an union of balls in $\overline \cB_{n+1}$.
In order to clarify our arguments below, we insist in choosing $\overline \cB_0=\{\overline{B}_0\}$ so that $\overline{B}_0$ is 
the smallest closed ball containing $\cK$.  
It follows that the convex hull $\cT(\cK)$ of $\cK$ and $\xg$ is a locally finite tree, 
and $ \partial \overline{B}_0$ is a branching point of $\cT(\cK)$ since we assumed $\cK$
is not reduced to a singleton.

We may suppose that $c_0 \in \CritB^+$ is fixed by $f_\diamond$. 
As before, we let $\hat c_n$ 
be the boundary point of the unique closed ball 
$\overline{B}_n(c_0)\in \overline \cB_n$ containing $c_0$. 
The sequence $\hat c_n$  decreases to $c_0$ in the sense that 
$\hat c_n\in (c_0,\hat c_{n-1})$, and furthermore $f_\diamond(\hat c_n)=\hat c_{n-1}$. 

By assumption, the critical slope $J$ (as defined in \S\ref{sec:jacII}) is constant on the segment $[c_0, \hat c_N]$ for some large enough $N$, and $1+J$ does not divide $d$
since $\nu(c_0)\ge 2$. 

By 
Corollary~\ref{cor:jacval}, we have
\[
A(\hat c_{n+1})= A(\hat c_n)+ \frac{d}{J+1} \big(A(\hat c_n)-A(\hat c_{n-1})\big)
\]
for all $n\ge N$. Since $\nu=q\big(\tfrac{d}{J+1}) \geq 2$, we conclude that 
\begin{equation}\label{eq:difference_q}
 q(A(\hat c_{n+1})) - q(A(\hat c_n) )\to \infty, 
 \end{equation} so there is a subsequence 
$(n_j)$ such that  $q(A(\hat c_{n_j}))\to \infty$. In particular,  $b(\hat c_{n_j}) \to \infty$.

Recall that $\hat c_0$ is a branched point of $\cT(\cK)$. Since $\cK$ is backward invariant, and
$f_\diamond$ is open, the relation
$f_\diamond(\hat c_n)=\hat c_{n-1}$ implies by induction that 
 $\hat c_n$ is a branched point of $\cT(\cK)$ for every $n\ge0$.

When $j$ is  large enough, $b(\hat c_{n_j})> m(c_0)$, and since the multiplicity of a type 2 point is the minimal multiplicity of rigid points in the corresponding closed ball, we also have $m(c_0)\geq m(\hat c_{n_j})$. Since $\hat c_{n_j}$ is a branched point of $\cT(\cK)$, there exists 
  an open ball $B'_{n_j}$ whose boundary point is equal to $\hat c_{n_j}$
and which does not contain $c_0$.  By  Theorem~\ref{thm:gen-mult}, there is a unique open ball with $\hat c_{n_j}$ as a boundary point, containing a rigid point with 
multiplicity $<b(\hat c_{n_j})$, which must thus  be the one containing $c_0$.
It follows that 
the multiplicity of any rigid point in $B'_{n_j}$ is no less than  $b(\hat c_{n_j})$. 
Since  $f^{{n_j}+1}_\diamond$ maps $B_{n_j}$ to an open ball containing the unique
closed ball in $\overline \cB_0$, $c_0$
admits a preimage $\phi_{n_j}\in B_{n_j}$. This preimage is rigid of multiplicity 
$\ge b(\hat c_{n_j})$, which
shows  that the    multiplicity of   rigid points in $\cK$ is unbounded.

Let us now argue that $\cK$ also contains non-rigid points. 
Let $B_0$ be the open ball whose boundary point is 
$\hat c_0$ and contains $c_0$. 
We shall construct a decreasing sequence of open balls  $\{B'_\ell\}_{\ell\ge 1}$ whose boundary point $y_\ell$ is a branched point
of $\cT(\cK)$, and such that  
\begin{enumerate}[(i)]
\item 
 $f_\diamond^{k_\ell}(B'_\ell)=B_0$ for some  $k_\ell\ge1$;
\item
$b(B'_\ell)\ge 2 b(B'_{\ell-1})$.
\end{enumerate}
Then the sequence of  the corresponding closed sets    
$B'_\ell\cup\{y_\ell\}$ forms a decreasing sequence of compact sets
intersecting $\cK$ by (i) and its intersection cannot contain any rigid point by the second condition.

Let us explain how to build this sequence of balls. Choose some large $n_0$ 
such that $b(\hat c_{n_0})\ge 2 b(\hat c_0)$. As before, since $\hat c_{n_0}$ is a branched point of $\cT(\cK)$
there exists an open ball $B$ whose boundary is equal to $\hat c_{n_0}$,  which 
intersects $\cK$, and with the property that 
  all rigid points in $B$ have multiplicity at least $b(\hat c_{n_0})$.
This ball contains a preimage of  $\hat c_0$ by the equidistribution~\eqref{eq:plbck}. From this, we infer the existence of 
an open ball $B'_1\subset B$ such that $f_\diamond^{k_1}(B'_1)=B_0$ and 
$b(B'_1)\geq 2 b(B_0)$.  

Now, suppose that $B'_\ell$ has been constructed. Since  $f_\diamond^{k_\ell}(B'_\ell)=B_0$, 
there exists a preimage $c_\ell$ of $c_0$ under  $f_\diamond^{k_\ell}$ in $B'_\ell$.
For $m\geq 1$, let $y_m$ be the unique preimage of  $\hat c_m$ by  $f^{k_\ell}_\diamond$ lying in $[c_\ell,\partial B'_\ell]$. If $r>0$ is sufficiently small, all critical points of 
$f^{k_\ell}$ in $B^{\an}(c_\ell,r)$ are contained in the orbit of $c_\ell$. It then  follows 
from Corollary~\ref{cor:jacval} that 
  for  $m$ large enough, we have
\[
A(\hat c_{m+1})- A(\hat c_m)=  C\, \big(A(y_{m+1})-A(y_{m})\big)
\]
for some fixed constant $C=C(\ell)$, and since  $q(A(\hat c_m)) - q(A(c_m))\to\infty$ by~\eqref{eq:difference_q}, it follows that $q(A(y_{m_j}))\to \infty$ along some subsequence $(m_j)$.
We can thus find an index $m$ such that $b(y_m) \ge 2 b(B'_\ell)$.
 As $y_m$ is a branched point of $\cT(\cK)$, by Theorem~\ref{thm:gen-mult},
 we can choose  an open ball 
$B$,  with $\fr B = y_m$ and 
such that all  rigid points in $B$ have multiplicity at least $b(y_{m})$.
As before $\hat c_0$ admits a preimage $z$ in $B$ under some  $f^{k_{\ell+1}}_\diamond$
with $k_{\ell+1}> k_\ell$, which must satisfy $b(z)\geq b(y_{m})$, 
and we choose  $B'_{\ell+1}$ to be 
the corresponding preimage of $B_0$. 
\end{proof}

\section{Examples}\label{sec:examples}

In this section, we give explicit examples of skew products with uniformly bounded multiplicity on $\cK$, as well as examples where there are both rigid points with unbounded multiplicity, and points with infinite multiplicity. 
\subsection{Conjugacy to a  map of the form~\eqref{eq:cd}}

The class of maps that we have studied so far arises quite naturally, as our next result shows.

\begin{thm}\label{thm:formal}
Let $(k,|\cdot|)$ be any algebraically closed complete metrized field of characteristic $0$, 
and pick two integers $d>c\geq 2$. 
Suppose that  $f\colon (\nA^2_k,0)\to(\nA^2_k,0)$ is a finite analytic germ such that:
\begin{enumerate}[(a)]
\item the curve $\{z=0\}$ is (locally) totally invariant and $f^*[z=0] = d[z=0]$ as divisors;
\item the restriction of $f$ to $\{z=0\}$ has a super-attracting point at $0$ of order $c$.
\end{enumerate}
Then $f$ is formally conjugated to a (formal) map of the form
\begin{equation}\label{eq:normal}
(z,w) \mapsto \Biggl(z^d, w^c + \sum_{j=0}^{c-1} h_j(z) w^j\Biggl) 
\end{equation}
where the  $h_j$ are power series vanishing at $0$. 
\end{thm}

\begin{rmk}
The previous result applies to any polynomial map $f$ of the complex affine plane $\A^2_k$
that extends to a regular endomorphism to $\P^2_k$ of degree $d\ge2$, 
and has a super-attracting fixed point on the line at infinity that  is not totally invariant. 
\end{rmk}

\begin{rmk}
When the formal normal form is analytic, then one can argue that the conjugacy is indeed analytic. 
Denote by $\varphi$ the conjugacy between $f$ and its normal form. 
Up to a polynomial change of coordinates, we may assume that $f$ and its normal form given by \eqref{eq:normal} coincide up to high order
so that $\varphi$ is tangent to the identity up to any given order, say $\gg c$. 
Up to a blow-up of the form $\pi(z,w)=(zw^n,w)$ with $n \geq c+1$, the map $f$ lifts to a rigid germ of class $6$ in the sense of \cite{favre:rigidgerms}.
Then $\varphi$ lifts through $\pi$ and conjugates $(zw^n,w)$ to the lift of the normal form (which is analytic). It implies the lift of $\varphi$ to be analytic
which forces $\varphi$ to be analytic. 
\end{rmk}
The latter  remark raises the following problem. 
\begin{quest}
Is it possible to find a holomorphic map 
\begin{equation*}
f(z,w) =
\left(z^d, w^c + z h(z,w) \right)
\end{equation*}
with  $d>c \ge2$ and $h(0)=0$
which is not analytically conjugated to a map of the form~\eqref{eq:normal}?
\end{quest}

\begin{rmk}
Lemma~\ref{lem:zd} below shows that, under the assumptions of Theorem~\ref{thm:formal},
 we can at least make sure that the first component of $f$ is $z^d$ after a holomorphic change of coordinates.
\end{rmk}

The proof of the theorem proceeds  in several steps. 
First observe that by the B\"ottcher theorem, we can assume that the restriction of $f$ to $\{z=0\}$
is $w \mapsto w^c$,  so that in suitable coordinates
any map satisfying the assumptions of the theorem  can be written in the   form 
\begin{equation}\label{eqn:thegerm}
f(z,w)=\Big(z^d(1+\eps(z,w)),w^c +zh(z,w)\Big),
\end{equation}
where $\eps$ and $h$  are two germs of analytic functions at $0$ vanishing at that point.

\begin{lem}\label{lem:zd}
There exists a unique analytic germ $\Phi (z,w) = (z (1+ \varphi(z,w)), w)$,
with $\varphi(0)=0$,
such that $\Phi\circ f \circ \Phi^{-1} (z,w) = (z^d, \cdots)$.
\end{lem}
\begin{rmk}\label{rmk:assumptioncd}
Note that the assumption $c<d$ is not used in this lemma. 
\end{rmk}

\begin{proof}
This is a consequence of \cite{favre:rigidgerms}, see also \cite{ruggiero:rigidgerms}.
We include a proof here for convenience.
The conjugacy equation states:
\[
z \circ \Phi \circ f = z^d(1+\eps)(1+\varphi \circ f) =  z^d(1+\varphi)^d.
\]
Since $k$ has characteristic $0$, $(1+z)^{1/\ell}$ is well-defined and analytic in a neighborhood of $0$ for any $\ell\ge1$.
A solution of the above equation is then given by 
$$1+\varphi=\prod_{n\geq 1} (1+\eps\circ f^{n-1})^{d^{-n}}.$$
This expression converges as long as $\sum_{n\geq 1} d^{-n} \abs{\eps \circ f^{n-1}(z,w)} < +\infty$,
a property that holds in a neighborhood of the origin, thanks to the superattracting behavior of $f^n$.
More explicitly, for $\norm{(z,w)} \ll 1$, we have $\norm{f^n(z,w)}\leq\lambda^n$ for some $\lambda < 1$, while $\abs{\eps(z,w)} \leq M\abs{(z,w)}$ for some $M \gg 0$, 
so $\norm{\eps\circ f^n} = O(\lambda^n)$ and we are done. 

The uniqueness follows from the fact that since $f$ is super-attracting, 
$(1+\varphi \circ f) =  (1+\varphi)^d$ implies  $\varphi \equiv 0$. 
\end{proof}

\begin{proof}[Proof of Theorem~\ref{thm:formal}]
By the previous lemma we may assume that $$f(z,w)=\Big(z^d, w^c+\sum_{n \geq 1}z^n g_n(w)\Big),$$
where $g_n(0) = 0$. Write $g_n = g_{n}^- + g_n^+$
where $g_{n}^-$ is a polynomial of degree at most $c-1$ and
$\ord_w(g_n^+) \ge c$.

We will prove by induction on $m$ that   $f$ can be  analytically conjugated 
so that $g_1^+ = \cdots = g_{m-1} ^+ =0$. Indeed, 
suppose that this condition is satisfied, 
and for any analytic power series $\varphi$ vanishing at $0$, let us 
consider the analytic map:
\[
\Phi (z,w)=(z,w+z^{m}\varphi(w))~.
\]
We  define $\tilde{f}$ by $\tilde{f} = \Phi \circ {f}\circ \Phi^{-1}$, or 
 $\Phi \circ {f} = \tilde{f} \circ \Phi$, and observe that  $\tilde{f}$ is expressed as 
$\tilde{f}(z,w)=\Big(z^d, w^c+\sum_{n \geq 1}z^n \tilde{g}_n(w)\Big)$,
where the  $\tilde{g}_n$ are analytic and vanish at $0$. 

By the inductive hypothesis, and using $d\ge2$ we get
\begin{align*}
w \circ \Phi \circ f(z,w) 
&= w^c+\sum_{n\ge 1} z^n g_n(w) + z^{md} \varphi\left( w^c+ \sum_{n\ge 1} z^n g_n(w)\right)
\\
&=
w^c+\sum_{n\le m} z^n g_n^-(w) + z^{m} g_m ^+(w)~~ \mod (z^{m+1})
\end{align*}
On the other hand, we have
\begin{align*}
w \circ \tilde{f} \circ \Phi (z,w) & = (w+z^m\varphi (w))^c + \sum_{n\geq 1}z^n \tilde{g}_n(w+z^m\varphi (w)) \\
&
= 
w^c+cw^{c-1} z^m\varphi(w) + \sum_{n \leq m}z^n \tilde{g}_n(w) ~~ \mod (z^{m+1})
\end{align*}
Identifying term  by term, we infer that  $\tilde{g}_n(w)=g_n^-(w) $ for all $n\le m-1$, and setting
\[
\varphi(w) := \frac{g^+_m(w)}{c w^{c-1}}
\]
we get $\tilde{g}_m = g_m^-$, thereby completing  the induction step.

To conclude we observe that for any sequence of analytic maps $\varphi_m(w)$ vanishing at $0$ 
the sequence of analytic diffeomorphisms
\[\hat{\Phi}_m(z,w) := (z, w + z^m \varphi_m(w)) \circ \cdots \circ (z, w + z^2 \varphi_2(w)) \circ (z, w + z \varphi_1(w))
\]
satisfies $\hat{\Phi}_m(z,w)= \hat{\Phi}_{m-1} (z,w) \mod (z^m)$, hence
formally converges. 
\end{proof}

\subsection{Some explicit examples}

We include here some explicit examples of skew-products $f\colon (\nC^2,0) \to (\nC^2,0)$
in order to illustrate our results. 
To simplify   notation, we   use the following convention:
if $\phi \in \widehat{\nL}$ is a generalized Puiseux series, 
we denote the point $\zeta(\phi,e^{-t})$ by $\phi+\star z^t$.  

\subsubsection{An example with uniformly bounded multiplicity in $\cK$}

Consider the map $f(z,w)=(z^4,w^2-z^4)$. In this case, $\jac_w(f) = 2w$, and we have only one critical branch $c_0=(w=0)$.
In this case, $\cK$ is contained in $\overline{B}(0;e^{-2}))$, and 
$f_\diamond(\phi(z)) = \phi(z^{1/4})^2 - z$ so 
$f_\diamond(c_0) =  f_\diamond (0) = -z \not\in \overline{B}(0;2)$, hence the orbit of $c_0$ converges to $\xg$.
In this case the multiplicity is uniformly bounded in $\cK$. Moreover,
 $\overline \cB_1$ consists of two balls, with boundary points $y_1= z^2 + \star z^3$ and $y_2 = -z^2 + \star z^4$. In particular $\bigcup \mc{B}_1$ is disjoint from the critical tree $[\xg,0]$.
The generic multiplicities of $y_1, y_2$ and of their images under $f_\diamond$ are all equal to $1$, and we deduce that all elements in $\cK$ have multiplicity $1$.

This example was studied extensively by Gignac in \cite{gignac:conjarnold}, where the author studies  the growth of the intersection of curves in $\cK$ under iteration.

\subsubsection{An example with  unbounded multiplicity in $\cK$}

Consider now the map $f(z,w)=\big(z^5, w^3 -3zw^2\big)$.
Its jacobian is $\jac_w(f) = 3 (w^2- 2zw)$, and we get two critical branches $c_0=0$ and $c_1=2z$.
Observe that $c_0$ is fixed, and 
$f_\diamond(c_1) = -4z^{3/5}$, $f^2_\diamond(c_1) = -64z^{9/25} (1+ o(1))$
so that $f^n_\diamond(c_1) \to \xg$. 
Also $c_0$ is not totally invariant, hence $\cK$ is a Cantor set by Theorem~\ref{thm:invariantcantor}.
It contains non rigid points and rigid points of unbounded multiplicity by Theorem~\ref{thm:infinity-mult}.

\medskip

\section{Dynamics in the complex basin of attraction}\label{sec:complex}
In this section, we suppose that  
\begin{equation*}\tag{\ref{eq:f_intro}}\label{eq:fhol}
f(z,w) =  
\left(z^d, w^c + z h(z,w) \right)
\end{equation*}
where $h$ is a holomorphic map defined in $D(0, r)^2$ such that $h(0,0)=0$.
Under this assumption, the point $(0,0)$ is fixed and superattracting and we denote by  $\Omega$   its immediate basin of attraction. 
We write $\abs{p} = |(z,w)|=\max\{|z|,|w|\}$. 
Recall that the contraction rate $c_\infty(p)$ of a point in $\Omega$ is defined by 
  \begin{equation}\label{eq:attraction} 
\log c_\infty(p) = \lim_{n\to \infty} \unsur{n} \log \abs{\log \abs{f^n (p)}},
\end{equation} 
when the limit exists. 
The following is a    more precise  version of 
Theorem~\ref{thm:complex_intro}.

\begin{thm} \label{thm:complex}
Let $f$ be of the form~\eqref{eq:fhol} with $2\leq c<d$. 
Then the sequence of psh functions $(z,w)\mapsto \unsur{c^n} \log \abs{f^n(z,w)}$
converges 
in $L^1_{\mathrm{loc}}$ to a psh function 
$g\colon \Omega \to \R\cup\{-\infty\}$
such that $e^{g}$ is continuous, and
satisfies
\[g \circ f = c g \, . \]

The positive closed $(1,1)$ current ${\mathrm{d}}{\mathrm{d}}^c g$ 
is supported on 
the locus $\cW:=\{g = -\infty\}$ which is a  closed pluripolar set.
It gives no mass to curves  unless $f$ is conjugated to a product map, in which case $T$ is concentrated on a smooth 
curve intersecting  transversally $\set{z=0}$.

Finally, the contraction rate exists  at every $p\in\Omega$, and 
$c_\infty(p)=c$ if $p\notin\cW$, and $c_\infty(p)=d$ if $p\in\cW$. 
 \end{thm}

\begin{rmk} 
Note that the existence of the limit $g$ follows from the general 
result~\cite[Thm~B]{favre-jonsson:eigenval}. 
However, for a general superattracting germ, $e^{g}$  may  not be continuous, and the set 
of possible attraction rates  $\{c_\infty(p), p\in \Omega\}$
may  not necessarily be finite. A global analog of this   phenomenon was exhibited in~\cite{dinh-dujardin-sibony}.
\end{rmk}

\begin{cor}\label{cor:Pk}
If $P$ is of the form $$P(z,w) = P_z(w) = w^k + \sum_{j=0}^{k-1} \alpha_j(z) w^j,$$ with $\alpha_j(0) = 0$ for every 
$0\leq j\leq k-1$, then the sequence $\unsur{kc^n} \log P\circ f^n$ converges in $L^1_{\loc}(\Omega)$ 
to $g$. 
\end{cor}

\begin{proof}[Proof of Theorem~\ref{thm:complex}]
Writing $f^n(z,w) = (z_n, w_n)$, we have 
$$\unsur{c^n}\log\abs{f^n(z,w)} = \unsur{c^n}\log \max(\abs{z_n}, \abs{w_n}) = \max\lrpar{\frac{d^n}{c ^n}\log\abs{z}, \unsur{c^n}\log \abs{w_n}}.$$ Since $\frac{d^n}{c^n}\log\abs{z}$ converges uniformly to $-\infty$, it is enough to focus on the convergence of 
$g_n\colon (z,w)\mapsto \unsur{c^n}\log \abs{w_n}$. 

After appropriate scaling of the coordinates, we may assume that $r<1/2$ and  $\norm{h}<1/2$. 
Let   $\Omega_0 := \set{(z,w), \ \abs{z} <  \abs{w}^c}$. We claim that  
$f(\Omega_0)\subset \Omega_0$. Indeed, for $p = (z,w)\in \Omega_0$ we have 
\[\abs{w_1} = \abs{w^c + zh(z,w)} \geq \abs{w}^c -  \norm{h}\cdot \abs{z} \geq \unsur{2} \abs{w}^c
\]
 hence 
\[\abs{z_1}  = \abs{z}^d <   \abs{w}^{cd} \leq \abs{w}^{c(d- c)}2^c \abs{w_1}^c 
<  \abs{w_1}^c.
\]
Now, for $n\geq 1$, since $f^n(p)\in \Omega_0$, we have  that  
$\abs{f^n(p)} \sim \abs{w_n}$ and 
\[w_{n+1} = w_n^c + z_n h(z_n, w_n)   = w_n^c\lrpar{1+ \frac{z_n}{w_n^c} h(z_n, w_n)}.\]
By definition of $\Omega_0$  we have that $  \abs{ {z_n}{w_n^{-c}} h(z_n, w_n)}\leq 1/2$,
hence  
\begin{equation}\label{eq:wn}
\unsur{c^{n+1}}\log \abs{w_{n+1}} - \unsur{c^n} \log\abs{w_n}  = \unsur{c^n} \log \abs{1+  \frac{z_n}{w_n^c} h(z_n, w_n)}\leq \frac{\log 2}{c^n}.
\end{equation}
Recall that $g_n(z,w)  =  \unsur{c^n}\log\abs{w_n}$. 
It follows from the previous inequality that the sequence 
$(g_n)$ converges uniformly in $\Omega_0$ to a pluriharmonic function $g$  
 such that $g(p) = \log\abs{p} +  O(1)$ at the origin.
Furthermore, since in $\Omega_0$  $f^n(p)\sim w_n$, $\unsur{c^n}\log\abs{f^n}$ converges to 
$g$ as well, so  $g \circ f = c g$. 

Let now  $\Omega' = \bigcup_{n\geq 0} f^{-n}(\Omega_0)\subset \Omega$ and $\cW = \Omega\setminus \Omega'$.  
Then $g$ extends by iteration  to a pluriharmonic function in $\Omega'$.
If $p  =(z,w)\in \cW$ then for every $n\geq 0$ we have $\abs{z_n}\geq  \abs{w_n}^c$, and since $z_n = z^{d^n}$ we infer that $c_\infty(p) = d$, and  $\lim_{n\to\infty}  \unsur{c^n}\log\abs{f^n(p)}= -\infty$. 
Put $g(p)= -\infty$ for $p\in \cW$. 
  
Let us show that $g(p) \to -\infty$  when $p\in \Omega$ tends to $ \cW$.
To establish this, for $N\geq 1$, let $\Omega_N  = \bigcup_{n=0}^N f^{-n}(\Omega_0)$, which is an increasing sequence of open sets, and fix $p\in \Omega_{N}\setminus  \Omega_{N-1}$.  For $0\leq j\leq N-1$ we have $f^j(p)\notin \Omega_0$ hence  
$$
\abs{f^{N-1}(p)}\leq \max({\abs{z_{N-1}}, \abs{w_{N-1}}})\leq   \abs{z_{N-1}}^{1/c} \leq \abs{p}^{d^{N-1}/c}.
$$ 
Being $f$ superattracting, it follows that $\abs{f^{N}(p)}\leq \abs{p}^{d^{N-1}/c}$ as well. Now, since $f^N(p)\in \Omega_0$, 
when $n\to \infty$ we have $\abs{f^n(p)}\sim |w_n|$, and  
\begin{align}\label{eq:ineqcn}
\begin{split}
\unsur{c^n}\log \abs{w_n} &= \unsur{c^N}\log \abs{w_N}  + \sum_{q=N}^{n -1}\lrpar{
\unsur{c^{q+1}}\log \abs{w_{q+1}} -  \unsur{c^{q }}\log \abs{w_{q }}} \\
& \leq \unsur{c^N} \log \abs{f^N(p)} + \sum_{q=N}^{n-1} \frac{\log 2}{c^q}
 \leq \unsur{dc} \frac{d^N}{c^N} \log\abs{p} + \frac{c}{c-1} \log 2\text{,}
\end{split}
\end{align}
where in the inequalities we use the estimate on $\abs{f^n(p)}$ and \eqref{eq:wn}. Thus we conclude that 
$$
g(p) \leq C_1 \frac{d^N}{c^N} \log\abs{p} + C_2
$$
for positive constants $C_1$ and $C_2$, and the result follows by letting $N\to\infty$.

We claim  that $(g_n)$ converges  to $g$ pointwise, and even locally 
uniformly in $\Omega$, in the sense that $\abs{w_n}^{1/c^n}$ 
converges locally uniformly to $e^{g}$. Therefore,  the convergence 
holds in $L^1_{\loc}(\Omega)$, which implies 
 that $g$ is psh, $g\circ f = c g$,  and the remaining  properties. 
 We already know that the sequence converges 
locally uniformly outside $\cW$ to a continuous pluriharmonic function $g$ and pointwise to $-\infty$ on $\cW$. 
Fix a small $s\neq 0$ and a vertical  line $\set{z=s}\subset D(0, r)^2$. In this line,  for every $n$, $w_n$ depends only on $w$, 
and $  \Omega_N ^\complement$ is a    subset of the form 
$w_N\inv\lrpar{\set{\abs{w}<s_N }} $, where $s_N=\abs{s}^{d^N/c}  $, so it is polynomially convex. 
Therefore, by the Maximum Principle, 
since the  inequality~\eqref{eq:ineqcn} holds holds at the boundary of $\overline{\Omega_N}$, it holds 
everywhere in $  \Omega_N ^\complement$. Thus,  the announced local uniform convergence follows. 

Suppose that $\cW$ is an analytic curve $C$. Then 
$T= {\mathrm{d}}{\mathrm{d}}^c g$  is a current of integration, and  
since  $f^*T = cT$, it follows that $C$ is totally invariant. 
Let $\set{C_i}$ be the set of branches of $C$, and denote 
   the associated points in $B(0,1)^{\an}$ by  $x_{C_i}$. 
Then the finite set $\{ x_{C_i}\}$ is totally invariant by $f_\diamond$, which implies
$\cK$ to be reduced to a singleton by the convergence~\eqref{eq:plbck} in 
Theorem~\ref{thm:ergo-meas}. This shows that 
$C$ is irreducible and totally invariant by $f$. 
By Theorem~\ref{thm:invariantcantor}, $C$ is smooth and transversal to $(z=0)$. 
Changing coordinates, we may suppose that $C=(w=0)$
so that 
$f(z,w) = (z^d, w^c\times \mathrm{unit})$. It then follows from Lemma~\ref{lem:zd} (with coordinates reversed, see Remark~\ref{rmk:assumptioncd})
that $f$ is analytically conjugated to a product map.

Finally, suppose that $T$ carries positive mass on an analytic curve. In this case we claim that $\cW$ is a subvariety, in which case the previous paragraph applies. Indeed write the Siu decomposition $T = \sum_{W\in \mathcal{A}}\alpha_W[W]+ T_{\mathrm{diffuse}}$, where  $\mathcal{A}$ is an at most countable collection of analytic subvarieties and $T_{\mathrm{diffuse}}$ gives no mass to curves. 
Since $  f^*$ preserves both the analytic and diffuse part, the current $\sum_{W\in \mathcal{A}}\alpha_W[W]$
is invariant under $c\inv f^*$.

Suppose by contradiction that $\mathcal{A}$ is infinite. One can then find $W_0\in \mathcal{A}$ 
such that $f^k(W_0)$
is not a critical component for any $k\ge0$, thus  $f^*[f^{k+1}(W_0)] - [f^{k}(W_0)]$ 
gives no mass to  $[f^{k}(W_0)]$. 
The invariance relation then yields 
\[
\alpha_{f^k(W_0)}= \ord_{f^k(W_0)}\left(c\inv f^*\sum_{W\in \mathcal{A}}\alpha_W[W]\right)=c\inv \alpha_{f^{k+1}(W_0)}
\]
which implies $\alpha_{f^{k+1}(W_0)}\to\infty$, a contradiction.
%
%
So we are back to the setting of the previous paragraph and the argument is complete. 
\end{proof}
%
%
%

\begin{proof}[Proof of Corollary~\ref{cor:Pk}]
Put 
\begin{align*}
h_n(z,w)   &:= \unsur{kc^n} \log P\circ f^n(z,w) = \unsur{kc^n}\log P(z_n, w_n)  
= \unsur{kc^n}\log P_{z_n}(w_n)\\
&=  \unsur{kc^n} \log\biggl({ w_n^k+ \sum_{j=0}^{k-1} \alpha_j(z_n) w_n^j}\biggl). 
\end{align*}
With $g_n = c^{-n}\log\abs{w_n}$ as above, we get 
$$ h_n(z,w) - g_n(z,w)= \unsur{kc^n} \log\biggl({ 1+ \sum_{j=0}^{k-1} \alpha_j(z_n) w_n^{j-k}}\biggl).$$
For $(z,w) \notin \cW$ we have $z_n = z^{d^n}$ and $\abs{w_n}\gtrsim \beta^{c^n}$ where $\beta<1$ is locally uniform, so 
$ h_n - g_n$ converges locally uniformly to 0 outside $\cW$. Thus, 
 the sequence of psh functions 
$(h_n)$ is bounded from above and converges to $g$ 
outside the pluripolar set $\cW$, hence it 
converges to $g$  in $L^1_{\loc}(\Omega)$, and we are done. 
\end{proof}

\section{Structure of the invariant current}\label{sec:structure}

This section is devoted to the proof of Theorem~\ref{thm:structure_intro} (see also Theorem~\ref{thm:structure} below) 
which describes the structure of the invariant current as an average
of integration currents over curves in $\cK$. 
We first  need to 
explain  how a map of the form~\eqref{eq:fhol} acts on $B^{\an}(0,1)\subset \A^{1,\an}_{\C\laurent{z}}$. 

\subsection{Action of a formal conjugacy on the Berkovich open unit ball}\label{subs:update}
If $K$ is any complete non-Archime\-dean field,
any  power series $\phi= \sum_{n\ge 0} a_n T^n$ such that 
$|a_0|<1$ and $|a_i|\le 1$ for all $i\ge1$ induces a self map of the (classical) 
unit ball $B(0,1)\subset K$. 
If $K$ is algebraically closed, it follows from Hensel's lemma that $\phi$ is invertible 
 iff $|a_1|= 1$ and $|a_i| <1$ for all $i\neq 1$ (see e.g. \cite[\S 3.4]{benedetto:dyn1NA}). 
This  action extends to the Berkovich unit ball $B^{\an}(0,1)$ (see \cite{baker-rumely} for a discussion  of  $\overline B^{\an}(0,1)$ in these terms, which is readily adapted to this case 
by using 
the fact that $B^{\an}(0,1) = \bigcup_{r<1}  \overline B^{\an}(0,r)$). In particular if $|a_1|= 1$ and $|a_i| <1$ for   $i> 1$, then $\phi$ induces an analytic isomorphism of 
  $B^{\an}(0,1)$ (whenever $K$ is algebraically closed or not).

Back to $K = k\laurent{z}$, recall   that a germ of irreducible formal curve $C$ containing $0$
and not included in $\{z=0\}$ defines a 
a type 1 point $x_C\in B^{\an}(0,1)$ by setting 
$\displaystyle-\log|P(x_C)|= \frac{C \cdot (P=0)}{m(C)}$
for all $P\in \C\laurent{z}[w]$, where we recall that $m(C)=C \cdot (z=0)$. 
The map $C\mapsto x_C$ is a bijection between germs of curves as above and
type 1 points in $B^{\an}(0,1)$.

From now on we assume  that $f$ is a holomorphic map in $(\C^2, 0)$ 
of the form~\eqref{eq:fhol} with $2\leq c<d$.
It follows from Theorem~\ref{thm:formal} that $f$ is conjugated by a \emph{formal} map 
of the form $\Phi(z,w) = \displaystyle \Big(z, w + \sum_{n\ge1} z^n \varphi_n(w)\Big)$ with $\varphi_n\in\C\fps{w}$
to 
\[
\tilde{f}(z,w) := \biggl(z^d, w^c + \sum_{j=0}^{c-1} h_j(z) w^j\biggl)
\]
so that $ f= \Phi\circ \tilde{f} \circ \Phi^{-1} $. 
From the above discussion, 
the map $\Phi$ induces  an analytic isomorphism  of the Berkovich open unit ball 
$B^{\an}(0,1)\subset \A^{1, \an}_{\C\laurent{z}}$ (resp. $\A^{1, \an}_{\nL}$), 
and for the corresponding actions on  $B^{\an}(0,1)$, we have 
  $f_\diamond = \Phi \circ \tilde{f}_\diamond \circ \Phi^{-1}$. Thanks to this conjugacy, all the definitions 
  and results from 
  Sections~\ref{sec:dynamics_psp} to~\ref{sec:infinite} apply to $f_\diamond$.   
 In particular we can define 
 $\cK:=\cK(f)$ and $\cK_\nL:=\cK_\nL(f)$, and the ergodic  measure $\mu_{\na}$ supported on $\cK$
 (resp. $\mu_{\na, \nL}$ on $\cK_{\nL}$), 
 which are the images of the corresponding objects associated to $\tilde f$.

\subsection{The main statement}

Let us first present  a basic correspondence between the dynamics of $f$ in $(\C^2, 0)$ and that 
of $f_\diamond$ in $B^{\an}(0,1)$, which may be viewed as an introduction to Theorem~\ref{thm:structure_intro}. 

\begin{prop}\label{prop:anal-na}
Let $C$ be any germ of analytic irreducible curve containing $0$
and not included in $\{z=0\}$. 

Then $x_C\in \cK$ if and only if 
$C\subset \cW$.
\end{prop}

\begin{proof}
Fix any parametrization $t \mapsto (t^q,\phi(t))$ where  $q = m(C)$ and $\phi$ is analytic. 
Note that $\phi(z^{1/q})$ is then a Puiseux series determining $C$.

Suppose first that  $x_C\notin \cK$. Then  for any given $\rho<1$, there exists an integer $n$ 
 such that 
$x_{f^n(C)}$ belongs to the annulus $\{\rho< |x| <1\}\subset B^{\an}(0,1)$.
Choose any $c'< 1/c$, and write $\rho= e^{-c'}$. Recall from the proof of Theorem~\ref{thm:complex} that since $h(0,0)=0$, we may assume $\abs{h}<1/2$, and 
for any point $p=(z,w)$ in $\Omega_0= \{|z| < |w|^c\}$, the inequality
\[
\abs{w_1} = \abs{w^c + zh(z,w)} \geq \abs{w}^c -  \norm{h}\cdot \abs{z} > \frac{1}{2} \abs{w}^c
\]
holds, so that in particular $\Omega_0\subset \{g>-\infty\}$. 
Replacing $C$ by its image under $f^n$, we may suppose that $\rho< |x_C| <1$
which means that $  \ord_t(\phi) < qc'$. 
Now one can find a constant $A>0$ such that for any   small enough $t$,  
\[
 |\phi(t)|^c
\ge A 
|t|^{c \times \ord_t(\phi)}
\ge A (|t|^q)^{c \times \ord_t(\phi)/q} >  (|t|^q)^{c c'} >|t|^q~ .
 \]
Possibly reducing $C$, this implies that  $C \subset \Omega_0\subset \{g>-\infty\}$, hence 
$C$ is not included in $\cW$. 

Conversely suppose that
$x_C\in \cK$. 
Define  a sequence of negative subharmonic functions on a neighborhood of the unit disk by 
 $g_{n,C} \colon t\mapsto  \frac1{c^n} \log|w_n(t^q, \phi(t))|$.
Suppose by contradiction that $(g_{n,C})$ does not converge to $-\infty$. 
Then by  the Hartogs lemma there exists a subsequence $(g_{n_j,C})$ converging to a subharmonic
function $\tilde{g}$. 
Consider the Lelong number $\nu\lrpar{g_{n,C}, 0}$ of the subharmonic function $g_{n,C}$ at 0: 
it is equal to
\[\frac1{c^n}  \ord_t (w_n(t^q, \phi(t))) = - \frac1{c^n} \log|w_n(x_C)| = - \frac{d^n}{c^n} \log \abs{f_\diamond^n(x_C)}\] which tends to infinity by Theorem~\ref{thm:green}. On the other hand, the 
upper semicontinuity of the 
  Lelong number implies that  
\[\limsup_j- \frac1{c^{n_j}} \log|w_{n_j}(x_C)|  \leq \nu \lrpar{\tilde g, 0} <\infty . \]
This contradiction shows that  $g|_C \equiv -\infty$, hence $C\subset \cW$.
\end{proof}

When no critical point  of $f_\diamond$ is contained in $\cK$,   Corollary~\ref{cor:mult} implies that 
points in $\cK$ are associated with formal curves of uniformly bounded multiplicity. 
By the previous proposition, 
 this condition 
is equivalent to the statement that no irreducible component of the critical locus
 of $f$ (besides $\set{z=0}$) 
 is contained in $\cW$. 
In this situation, we have a description of the current of Theorem~\ref{thm:complex} in terms of non-Archimedean data. The following is a precise  version of  Theorem~\ref{thm:structure_intro}:

\begin{thm}\label{thm:structure} 
 Assume that $f$ is a holomorphic map of the form~\eqref{eq:fhol} with $2\leq c<d$, and that
no critical branch is contained in $\cW$.  
Then, $\cW$ is a union of analytic curves through the origin. Furthermore, 
outside the origin,  $\cW$ 
 is the support of  a lamination by Riemann surfaces 
  whose transversals are Cantor sets.  

More precisely, there exists $r_0>0$ such that all Puiseux series 
$\phi\in\cK_\nL$ converge in the disk $\D(0,r_0)$. 
For any $x\in\cK$,  if we denote by $C(x)$ the analytic curve in $\D(0,r_0)\times \C$
parameterized by $t\mapsto (t^{m(x)},\phi(t^{m(x)}))$,  where $\phi\in\C\puiseux{z}$ is such that $x=\pi(\phi)$, we have the integral representation  
\[
T =  {\mathrm{d}}{\mathrm{d}}^c g = \int_\cK \frac{[C(x)]}{m(x)} \, {\mathrm{d}}\mu_{\na}(x)~. 
\]
\end{thm}

%
%

In particular $T$ is a uniformly laminar current outside the origin. 
See~\cite{bls} or~\cite{dujardin:ICM} for this notion and some basic facts     that will be needed in \S~\ref{subs:laminarity}.

\begin{quest}
Consider  any holomorphic map   of the form~\eqref{eq:fhol} with $2\leq c<d$. 
Does there exist $r>0$ such that for any series 
$\sum_n a_n z^{\beta_n} \in \cK_\nL$, then $\sum_n |a_n|r^{\beta_n} <\infty$?
\end{quest}
A positive answer to this question could possibly lead to an improvement of the previous result,  
without any assumption on the critical of $f$.  
When $x\in \cK_{\nL}$,  the curve  $C(x)$ would be replaced 
by the positive closed $(1,1)$ current defined in~\cite[Proposition~6.9]{favre-jonsson:valanalplanarplurisubharfunct}.

The remainder of this section will be  devoted to the proof of the theorem.

\subsection{The open ball partition and the coding map}\label{subs:coding}
The  first step is to construct a sequence of coverings $\cB_n$   of $\cK$ by open balls. 
This sequence differs from the one introduced in the proof of Theorem~\ref{thm:invariantcantor}
which involves closed balls. 

We start with any open ball $B_0$ containing $\cK$ whose boundary $x_0$ is a type 2 point.  
Its image under  $f_\diamond$ is an open ball which also contains $\cK$ hence either 
$f_\diamond(B_0) \subseteq B_0$ or $f_\diamond(B_0) \supsetneq B_0$. Since 
$(f^n)_\diamond(x_0) \to \xg$, the former case
does not happen, and we infer that $f_\diamond^{-1}(B_0)$ is a finite union of open balls
whose closures are contained in  $B_0$. 

Set $\cB_0=\{B_0\}$, and let $\cB_n$ be the collection of connected components of $f_\diamond^{-n}(B_0)$. 
Each $\cB_n$ is a finite collection of disjoint open balls, such that $\cK\subset \bigcup\cB_n$, and $\bigcup \cB_n\subset \bigcup \cB_{n-1}$. 
More precisely, for any $x\in \cK$, denote by $B_n(x)$ the unique open ball belonging to $\cB_n$ and containing $x$. Then  for any $x\in \cK$, $\bigcap_{n\in\N} B_n(x) = \{x\}$. 
This follows for instance 
from the fact that  $(d/c)^n (f^{n})^*_\diamond\delta_{x_0}$ 
converges weakly to a probability measure whose support is $\cK$: indeed, if  
$\bigcap_{n\in\N} B_n(x)$ were strictly bigger than $\set{x}$, it would contain an open set 
intersecting $\cK$, but disjoint from $(d/c)^n (f^{n})^*_\diamond\delta_{x_0}$  for every $n$. 

By our assumption together with compactness, we can find  $N\in \N$ such that 
$\bigcup \cB_{N-1}$ does not contain any critical branch. This integer $N$ will be 
 fixed from now on, and we will use the two covering  of balls  $\cB_{N-1}$ and $\cB_N$  
 to encode the dynamics. 
From now on we denote $\cB = \cB_N$ and $\cB' = \cB_{N-1}$.

Let us  introduce  an oriented graph $\Gamma$ whose vertex set 
$V(\Gamma)$ is the collection of open balls $\cB$ by declaring that 
we draw an oriented edge
from $B_1$ to $B_2$ when $B_2\subset f_\diamond(B_1)$. 
We  then consider the set $\Sigma$ of all infinite sequences
$(B_i)_{i\geq 0} = B_0B_1B_2\cdots$, with $B_i\in \cB$ such that $(B_i, B_{i+1})$ is an 
edge of $\Gamma$ for all $i$.  Then  $(\Sigma, \sigma)$ is a \emph{subshift of finite type}, where 
 $\sigma\colon  (B_i)_{i\geq 0}\mapsto (B_{i+1})_{i\geq 0}$ is the left shift.
 
Denote by $A = (A_{ v v'})_{v,v'\in V(\Gamma)}$ the transition matrix of $\Gamma$ (that is, $A_{vv'} = 1$ if $vv'$ is an edge, and 0 otherwise).
The construction of the sequence of ball coverings $(\cB_n)_n$ guarantees that $A$ is primitive, that is $A^k$ has positive entries for some $k$. Let $M$ be the unique positive (right) eigenvector of $A$, normalized so that $\sum_{v\in V(\Gamma)}M_v=1$.
The Parry measure $\mu_\Sigma$ on  $(\Sigma, \sigma)$ is its unique measure of maximal entropy. 
When the topological degree is constant and equal to some $c$, the matrix $A$ has exactly $c$ entries equal to $1$ in each column, and the left Perron Frobenius eigenvector is $(1, \ldots, 1)$, associated to the dominating eigenvalue $c$. Then the structure of $\mu_\Sigma$ is easy to describe:
  for any cylinder  $[B_{m}\cdots B_{n}]\subset \Sigma$ we have 
\begin{equation}\label{eq:subshift}
\mu_\Sigma ([B_{m}\cdots B_{n}]) = c^{-(n-m) } A_{B_{m}B_{m+1}} \cdots A_{B_{n-1}B_n} \mu_\Sigma ([B_n])
\end{equation} and furthermore $\mu_\Sigma([B]) = M_B$ is the entry of $M$ corresponding to $B$. In particular 
$\mu$ is balanced, that is, $\sigma^*\mu_\Sigma = c\mu_\Sigma$.

There is  a canonical coding map ${\mathsf c} \colon \cK \to \Sigma$
sending a point $x$ to the sequence $\big(B(f^i_\diamond(x))\big)_{i\geq 0} \in \Sigma$, 
which semi-conjugates $f_\diamond$ to the left shift $\sigma$.

Recall that we have a continuous surjective map 
$\pi\colon B(0,1)^{\an}_\nL\to B(0,1)^{\an}$. 
We also denote by $\cB_{n,\nL}$ the collection of open balls $\pi^{-1}(B)$ with $B\in \cB_n$. 
For the same integer $N$ as above, we get an oriented graph
$\Gamma_\nL$, a set of itineraries $\Sigma_\nL$,   and 
 a coding map  ${\mathsf c}_\nL \colon \cK_\nL \to \Sigma_\nL$
which semi-conjugates $f_\diamond$ to the left shift. 

The absolute Galois group $G$ of $\C[[z]]$ also acts on $\Gamma_\nL$, hence on $\Sigma_\nL$, and its action  on $\Sigma_\nL$ is compatible with the one on $\cK_\nL$.  
To summarize the discussion, we have a commutative diagram which is  $G$-equivariant and respects the dynamics: 
\begin{equation}\label{eq:commutative_diagram}
\begin{tikzcd}[ampersand replacement=\&]
\cK_\nL \& \cK \\
\Sigma_\nL \& \Sigma
	\arrow["{\pi}", from=1-1, to=1-2]
	\arrow["{\mathsf{c}_\nL}", from=1-1, to=2-1]
	\arrow["{\mathsf{c}}"', from=1-2, to=2-2]
	\arrow["{}", from=2-1, to=2-2]
\end{tikzcd}
\end{equation}

Under our assumption that $\cB$ does not contain any critical branch, the map 
$\mathsf{c}_\nL$ becomes a homeomorphism. More precisely, we have
\begin{prop}~\label{prop:semicon}  
\begin{enumerate}
\item
The coding map ${\mathsf c} \colon \cK \to \Sigma$ is a surjective continuous map, semi-conjugating $f_\diamond\colon \cK\to \cK$ to the left shift 
$\sigma\colon \Sigma\to \Sigma$. 
\item
The coding map ${\mathsf c}_\nL \colon \cK_\nL \to \Sigma_\nL$ is a homeomorphism conjugating $f_\diamond\colon \cK_\nL\to \cK_\nL$ to the left shift 
$\sigma\colon \Sigma_\nL\to \Sigma_\nL$.  
\item
The space $\cK_\nL/G$ (resp. $\Sigma_\nL/G$) is homeomorphic to 
$\cK$ (resp. to $\Sigma$).  
\end{enumerate}
\end{prop}

Recall from \S~\ref{ssec:generic-mult} 
that if the boundary point of an open ball $B\subset B(0,1)^{\an}$ has generic multiplicity $1$, 
then $m(B) = 1$ as well, hence  
  $\pi^{-1}(B)$ is a single  open ball in $B(0,1)^{\an}_\nL$. 
  
\begin{cor}\label{cor:semicon}
Suppose that the boundary point of any open ball from the open covers $\cB$ and $\cB'$ has generic multiplicity $1$. Then $\pi\colon \cK_\nL\to \cK$ and 
${\mathsf c} \colon \cK \to \Sigma$ are homeomorphisms conjugating the dynamics. 
Furthermore,  ${\mathsf c}_*\mu_{\na} = \mu_\Sigma$.  
\end{cor}

\begin{proof}
Our assumption implies that the oriented graphs $\Gamma$ and $\Gamma_\nL$ are
 isomorphic under $\pi$, so that 
the left shifts on $\Sigma$ and $\Sigma_\nL$ are conjugated. Therefore 
the bottom arrow $\Sigma_\nL\to \Sigma$
in  the commutative diagram~\eqref{eq:commutative_diagram} is a homeomorphism and 
so are $\pi\colon \cK_\nL\to \cK$ and 
${\mathsf c} \colon \cK \to \Sigma$.

 The last assertion follows from the fact that if $\cK$ is homeomorphic to $\cK_\nL$, then 
 $f_\diamond:\cK \to\cK $ is $c$-to-1 and 
by  Theorem~\ref{thm:ergo-meas}
 the measure $\mu_{\na}$ is balanced, thus  ${\mathsf c}_*\mu_{\na}$ must coincide with  $\mu_\Sigma$.  
\end{proof}

\begin{proof}[Proof of Proposition~\ref{prop:semicon}]
The continuity of $\mathsf{c}$ follows directly from the description of $\cK$ as 
$\cK  = \bigcap_{n\geq 0} \bigcup \cB_n$ and the continuity of $f_\diamond$. Its surjectivity follows from the fact that given any allowed sequence $(B_i)\in \Sigma$,  we can define a nested sequence of balls 
$\bigcap_{i\geq 0} f_{-i}(B_i)$ whose intersection is non-empty, connected  and  contained in $\cK$,  hence   reduced to a point, whose image under $\mathsf{c}$ is the given itinerary.

The second assertion follows from the claim that in $\A^{1, \an}_{\nL}$,
for any open ball $B$
which does not contain any critical branch, then $f_\diamond \colon B\to f_\diamond(B)$ is injective.  
Indeed, this   implies that $f_\diamond^{-1}(B) \cap B'$ is an open ball,  for any pair of 
balls $B$ and $B'$ belonging to $\cB_n$
and any  $n\ge N$, which in turn forces ${\mathsf c}_{\nL}$ to be injective, hence a homeomorphism. 

To prove our claim, with notation as in Lemma~\ref{lem:f_rond_tau}, since 
  $\tau$   is a homeomorphism, 
we can work with $\tilde{f}$ instead of $f$,  and 
we are reduced to the statement that a polynomial 
$P(w)= w^c + \sum_{j=0}^{c-1} h_j(z) w^j$ with $h_j\in z\C[[z]]$ 
is injective on any open ball $B$ containing no critical point. This is established 
   e.g. in~\cite[Prop.~2.6]{trucco:algJuliaBerko}. 

Finally, the third assertion follows from Theorem~\ref{thm:invariantcantor},  the commutative diagram~\eqref{eq:commutative_diagram} and the discussion preceding it. 
\end{proof}

\subsection{Base change and reduction to the multiplicity $1$ case}\label{subs:base_change}
The base change of order $k\in \N^*$ is the map 
$\beta_k(z,w) = (z^k,w)$. 
Thanks to the special form of $f$, we 
 have the following commutative diagram
\begin{equation}\label{eq:base_change}
\begin{tikzcd}[ampersand replacement=\&]
	(\C^2,0) \&  \& \& (\C^2,0)\\
	(\C^2,0) \& \& \& (\C^2,0)
	\arrow["{(z^d, w^c+z^kh(z^k,w))}", from=1-1, to=1-4]
	\arrow["{\beta_k}", from=1-4, to=2-4]
	\arrow["{\beta_k}"', from=1-1, to=2-1]
	\arrow["{(z^d, w^c+zh(z,w))}", from=2-1, to=2-4]
\end{tikzcd}
\end{equation}
Observe that the top map $f_k\colon  (z,w) \mapsto  (z^d, w^c+z^kh(z^k,w))$ is also of the form~\eqref{eq:fhol}.   
By Lemma~\ref{lem:tau}, 
$\beta_{k,\diamond}\colon \phi \mapsto \phi(z^{1/k})$ is a homeomorphism on $B(0,1)^{\an}_\nL$,
such that $\beta_{k,\diamond} \circ f_{k,\diamond} = f_\diamond \circ \beta_{k,\diamond}$.
It is also well defined over $\C\laurent{z}$ but there it is not necessarily a homeomorphism 
(see Remark~\ref{rmk:tau}).  Note also that  $\CritB(f_k) = (\beta_{k, \diamond})\inv(\CritB(f))$. 

\begin{lem}
Given any type 2 point $x\in B(0,1)^{\an}\subset \A^{1, \an}_{\C\laurent{z}}$,
  there exists an integer $k\geq1$
such that for any multiple $\ell$ of $k$, every point in $\beta_{\ell,\diamond}^{-1}(x)$ has  generic multiplicity $1$. 
\end{lem}

\begin{proof}
We can suppose that $x = \pi(\zeta(\phi,r))$ with $\phi \in\C\puiseux{z}$, and $r\in e^{\Q_+}$ (see \S~\ref{sec:berko-rep}). 
Then any point $y\in \beta_{\ell,\diamond}^{-1}(x)$ is equal to $\pi (\zeta (u\cdot\phi(z^\ell),r^\ell))$ for some $u$ in the absolute Galois group
of $\C[[z]]$. To  conclude the proof, by definition of the generic multiplicity it is enough to 
choose $k$ such that $\phi(z^k)$ is a power series and $r^k\in e^\Z$. 
\end{proof}

We now return to the proof of the theorem. Recall that we  fixed an integer $N$, and 
two finite  covers $\cB$ and $\cB'$ of $\cK(f)$ by open balls. 
For any given $k$, then we let $\cB^{(k)}$ (resp. $\cB'^{(k)}$ be the set of all 
connected components of $\beta_{k,\diamond}^{-1}(B)$  
with $B\in \cB$ (resp. $B\in \cB'$). This  forms a finite open cover of $\cK(f_k)$, and since 
$\CritB(f) = \beta_k(\CritB(f_k))$, we get   $\cB'^{(k)} \cap \CritB(f_k) = \emptyset$.

By the previous lemma, we can find a sufficiently divisible  integer $k$ such that 
for all points $x= \partial B$ with $B\in \cB\cup\cB'$, 
and for all $y\in \beta_{k,\diamond}^{-1}(x)$, then
we have $b(y) =1$. In other words, all boundary points of the open balls in
 $\cB^{(k)}\cup\cB'^{(k)}$ have generic multiplicity $1$.

%

\subsection{Construction of the geometric model}\label{subs:model}
From this point and up to \S~\ref{subs:conclusion_mult1}, 
 we do the base change of the previous paragraph, and for notational ease, we
 replace $f$ by $f_k$. In particular   the boundary points of
all balls in $\cB\cup \cB'$ have generic multiplicity $1$ and  the conclusions of 
 Corollary~\ref{cor:semicon} hold for $f$.

A \emph{model} $\kappa\colon X \to (\C^2,0)$ 
is a smooth complex surface $X$ endowed with a proper bimeromorphic 
map $\kappa$ to a neighborhood of $0$ in $\C^2$ which is an isomorphism
away from $0$. Any such map can be decomposed into a finite sequence of point blow-ups.  

A point $p\in\kappa^{-1}(0)$ is said to be  \emph{free}  if it is a smooth point of the curve
$\pi^{-1}(z=0)$.  Observe that $0\in (\C^2,0)$ is free. 
When $X$ is obtained by successively    blowing-up  free points only, we say that $X$
is a \emph{free model}.

Any irreducible component $E$ of $\kappa^{-1}(0)$ determines a type 2 point
$x_E$ such that for any $P\in \C\laurent{z}[w]$, 
\begin{equation}\label{eq:mubE}
|P(x_E)| = \exp(-b_E^{-1}\ord_E(P))
\text{ where } b_E = \ord_E(z\circ \kappa)
\end{equation} 
(note that here we abuse notation and write $\ord_E(P)$ for $\ord_E(P\circ\kappa)$).
One proves by induction on the number of blow-ups that $b_E = b(x_E)$, where $b$ is the generic multiplicity from \S~\ref{ssec:generic-mult} (see~\cite[\S 6]{favre-jonsson:valtree}). 
If $C$ is any germ of curve transverse to $E$ at a free point (a \emph{curvette}, in the terminology of \cite{favre-jonsson:valtree}), then $\kappa(C)$ has multiplicity $b_E$ at the origin, hence the terminology (see~\cite[Rmk 3.41]{favre-jonsson:valtree}). 

Conversely, pick a type 2 point $x\in B(0,1)^{\an}$. Then there exists a 
unique sequence of blow-ups $\mathrm{bl}_{p_i} \colon X_i \to X_{i-1}$, $i=0, \cdots, n$
such that $p_i$ belongs to the exceptional divisor $E_i$ of  $\mathrm{bl}_{p_{i-1}}$
for all $i$, $X_0=(\C^2,0)$ and $x = x_{E_n}$.
When $b(x)=1$,   $X_n$ is a free model, i.e, all $p_i$ are free. Furthermore, $\diam(x_{E_n}) = e^{-n}$.

More generally,  a model $\kappa\colon X\to(\C^2,0)$ is free iff $b(x_E)=1$ for every
component $E$ of $\kappa^{-1}(0)$.

Pick any free point $p\in \kappa^{-1}(0)$ lying on some  irreducible component $E$. 
Choose any smooth curve $C'$ passing through $p$ and intersecting  $E$ transversally. 
Denote by $\C\{z\}$ the ring of convergent power series, and 
let $P_C\in \C\{z\}[w]$ be any equation for $C= \kappa(C')$.
Then the set 
\[U(p) := \{x\in B(0,1)^{\an}, |P_C(x)| < \exp(- \ord_E(P_C))/b_E)= |P_C(x_E)|\}\]
is the  open ball  whose boundary point is equal to $x_E$ and containing  $x_C$. 
This ball can be characterized geometrically as follows: a germ of irreducible curve $D$ belongs to $U(p)$
if and only if its strict transform in $X$ contains $p$.

Pick any type 2 point $x_0$ of generic multiplicity $1$. Then
$m(x_0)=1$, and the set of type 2 points $x\in B(0,1)^{\an}$ such that $x\ge x_0$ and $b(x)=1$
is finite. Indeed on $[x_0,\xg]$ the function $A$ is strictly decreasing, and the multiplicity is
equal to $1$ so that
$b(x)= q(A(x))$; since $A$ is bounded from above, the claim follows. 
Geometrically, these type 2 points  correspond to the 
generic points of divisors located under $x_0$ in the corresponding sequence of blow-ups. 

The next proposition is a variation on~\cite[Proposition~4.6]{zbMATH06618544}. 

\begin{prop}
Let $\{B_i\}$ be any finite collection of disjoint open balls in $B(0,1)^{\an}$ such that 
$b(\partial B_i) = 1$ for all $i$. Write $x_i= \partial B_i$.  

Then there exists a unique (up to isomorphism)  free model $\kappa\colon X\to (\C^2,0)$
such that the following holds:
\begin{itemize}
\item 
for each $i$, there exists a free point $p_i$ such that $B_i = B(p_i)$;
\item 
the set of type 2 points $\mathcal{D}(X) := \{x_E\}$ where $E$ ranges over all irreducible components of $\kappa^{-1}(0)$
coincides with the set of all type 2 points $x$ such that $b(x)=1$ and $x\ge x_i$ for some $i$.  
\end{itemize}
\end{prop}

\begin{proof}
Let us fix an arbitrary model $\kappa\colon X\to (\C^2,0)$, and first review the contents of~\cite[Corollary~6.34]{favre-jonsson:valtree}. 
The open set  $B(0,1)^{\an} \setminus \mathcal{D}(X)$ is a union of finitely many annuli (i.e., a connected open set having two boundary points), 
and infinitely many disjoint open balls. More precisely, if $U$ is a connected component of  $B(0,1)^{\an} \setminus \mathcal{D}(X)$
whose boundary point is $x_E$, then there exists a (unique) free point $p\in E$ such that $U= U(p)$. 
When $U$ is a connected component of  $B(0,1)^{\an} \setminus \mathcal{D}(X)$ having two boundary points $x_E \leq x_{E'}$, 
then $E$ intersects $E'$ at a point $p$, and a curve $C$ belongs to $U$ iff its strict transform in $X$ contains $p$.

\smallskip

We now construct $X$ by induction on the number of balls. 
When the collection of balls is empty, then we set $\kappa=\id$. 
Assume we have built the model $X$ for a collection $\{B_i\}$ 
and pick  any open ball $B$ whose boundary point $x$
is a type 2 point of generic multiplicity $1$, and  such that $B\cap B_i=\emptyset$ for all $i$. 

Let $U$ be the connected component of $B(0,1)^{\an} \setminus \mathcal{D}(X)$ containing $B$. 
Since the generic multiplicity of $x$ equals $1$, then $U$ cannot be an annulus. Indeed
if $\partial U=\{x_E,x_{E'}\}$, then $b(x) \ge b(x_E)+ b(x_{E'})$ for every type 2 point $x\in U$, as 
every type 2 point $x\in U$ is of the form 
$x_F$, where $F$ is a divisor above the non-free point $E\cap E'$ so $b(x)\geq 2$.

Therefore, $U$ is an  open ball $U= U(p)$ where $p$ is a free point lying on some 
component $E$ of the exceptional divisor of $\kappa: X \to (\C^2, 0)$. 
Let $X_1\to X$ be the blow-up at $p$ with exceptional divisor $E_1$. 
Then $X_1$ is a free model, and $x_{E_1}\le x_E$. We then repeat
the same argument. We obtain a sequence of blow-ups at free points $p_\ell \in E_{\ell-1}$
and type 2 points $x_{E_\ell} \le x_{E_{\ell-1}}$ such that $x$ belongs to the closed ball 
$B(\ell)= \{y\le x_{E_\ell}\}$. 
The process stops because $\diam(x)\le \diam(B(\ell)) =  e^{-1}\diam(B(\ell-1))$, and we conclude that 
 $x=x_{E_\ell}$ for some $\ell$.

The model $X_\ell$ then satisfies all our requirements. The details are left to the reader, see~\cite[Chapter~6]{favre-jonsson:valtree}. 
\end{proof}

We say that a model $\kappa \colon X \to (\C^2,0)$ dominates 
$\kappa' \colon X' \to (\C^2,0)$ when $(\kappa')^{-1} \circ \kappa$ is regular, i.e., 
$X$ is obtained by blowing-up points above $X'$. 
Suppose that $X$ dominates $X'$. Then the set 
of type 2 points $\mathcal{D}(X')$ is included in $\mathcal{D}(X)$. 
Conversely, when $\mathcal{D}(X')\subset \mathcal{D}(X)$ then 
$X$ dominates $X'$, as follows from~\cite[Proposition~3.2]{favre-jonsson:eigenval} applied to the identity map.

We denote by $\kappa\colon X\to (\C^2,0)$ (resp. $\kappa'\colon X'\to (\C^2,0)$) the  model
obtained by applying the previous proposition to $\cB$ (resp. to $\cB'$). 
Denote by $\{p_j\}$ (resp. $\{p'_i\}$)   the collection of free points in $X$ (resp. in $X'$)
such that $\{U(p_j)\} = \cB$ (resp. $\{U(p'_i)\} = \cB'$). 
By construction, $\mathcal{D}(X)$ contains all type 2 points $x$ of generic multiplicity 1  such that  $x\geq \partial B$
for some  $B\in \cB$, hence in particular the boundary points of all open balls in $\cB'$.  
Thus from the   above discussion we see   that $X$ dominates $X'$.

We  thus have two commutative diagrams: 
\[\begin{tikzcd}[ampersand replacement=\&]
	X \&  X'\&\&X \& X'\\
	(\C^2,0) \&\&\&(\C^2,0)\& (\C^2,0)
	\arrow["{\theta}", from=1-1, to=1-2]
	\arrow["{\kappa}"', from=1-1, to=2-1]
	\arrow["{\kappa'}", from=1-2, to=2-1]
	\arrow["{\kappa}", from=1-4, to=2-4]
	\arrow["{\kappa'}", from=1-5, to=2-5]
	\arrow["{f}", from=2-4, to=2-5]
	\arrow["{F}", dashed, from=1-4, to=1-5]
\end{tikzcd}\]
where $F$ is a meromorphic map.  

\begin{prop}\label{prop:gen1} With notation as above:
\begin{itemize}
\item
The map $\theta$ is a finite sequence of free blow-ups. Moreover, we have 
$\theta(p_j)=p'_i$ where $U(p'_i)$ is the unique ball in $\cB'$ containing
$U(p_j)$.
\item
The map $F$ is holomorphic at all points $p_j$. 
Moreover, we have $F(p_j) = p'_\ell$, where $p'_\ell$ is such that $f_\diamond(U(p_j)) = U(p'_\ell)$. 
\end{itemize}
\end{prop}
\begin{proof}
The first assertion follows from the definitions. The second one 
 is a consequence of~\cite[Proposition~3.2]{favre-jonsson:eigenval} applied to $f$ (where  divisorial valuations correspond to type~2 points and $U(p)$ has the same meaning). 
\end{proof}

\subsection{The graph transform}\label{subs:graph_transform} 
We work with the free models $X$ and $X'$ constructed in the previous subsection. 
For each $p'_i$ we fix coordinates  $(z'_i,w'_i)$ centered at $p'_i$ such that 
$(z'_i=0)$ is the exceptional divisor. 

Pick any $p_j$, and let $p'_i = \theta(p_j)$.
Since $\theta$ is a finite sequence of free blow-ups, then there exists a unique choice
of coordinates $(z_j,w_j)$ centered at $p_j$ and such that  $(z_j=0)$ is the exceptional divisor at $p_j$, and
\begin{equation}\label{eq:free_blowup}
(z'_i,w'_i)= \theta(z_j,w_j) = (z_j, z_j^{n_j} w_j + P_j(z_j))
\end{equation}
where $n_j\in \N^*$ and $P_j$ is a polynomial of degree at most $n_j$ vanishing at $0$. 
Indeed, each free blow-up can be expressed in coordinates as 
$(z,w)\mapsto (z, zw+w_0)$, by  choosing the exceptional divisor to be $\set{z=0}$, from which~\eqref{eq:free_blowup} readily follows by composition.

On the other hand, if $p'_\ell = F(p_j)$, then there exists
other coordinates $(\hz_j,\hw_j)$ centered at $p_j$ such that $(\hz_j=0)$
is the exceptional divisor and
\begin{equation}\label{eq:free_blowup2}
(z'_\ell,w'_\ell)= F(\hz_j,\hw_j) = (\hz_j^{m_j}, \hw_j)
\end{equation}
with $m_j\in \N^*$.
Indeed, since $\CritB(f)\cap \cB = \emptyset$, close to $p_j$, the critical set of $F$ is reduced to $(z_j=0)$ and $F^*(z'_i=0)$
is a multiple of $(z_j=0)$.  Using these assumptions we get that 
 in the coordinates $(z_j, w_j)$, 
$F$ is of the form $(z_j^{m_j}(c_1+g(z_j, w_j)), c_2 w_j + h(z_j, w_j))$, with $c_1c_2\neq0$, 
 $g(0,0) = h(0,0) = 0$ and $\frac{\partial h}{\partial w} (0,0)  =0$ 
and we choose  
$\hat z_j = z_j(c_1+ g(z_j, w_j))^{1/m_j}$ and $ \hat w_j = c_2 w_j + h(z_j, w_j)$ to arrive at the form~\eqref{eq:free_blowup2}.

\smallskip

The graph $\Gamma$ can be re-interpreted as follows: the vertex set is in bijection with $\cB$
hence with the set of free points $p_j$ in $X$,  and an oriented vertex joins $p_j$ to $p_\ell$
iff $F(p_j) = \theta(p_\ell)$.

Fix some sufficiently small $r>0$, and 
consider the space $\mathcal H_r$ of holomorphic functions $z\mapsto h(z)$
defined in the open disk of radius $r$, vanishing at 0, and 
whose sup norm is bounded by $r$. In other words, ${\mathcal H}_r$ is the closed ball 
of radius $r$ in the Hardy space $H^\infty(\D(0,r))$.

For each vertex $p_j$ and every $h\in \mathcal H_r$, we define the 
 smooth curve $C_j(h)$ parameterized by $(z_j(t),w_j(t))= (t, h(t))$ in the 
bidisk of radius $r$ centered at $p_j$, in the coordinates $(z_j,w_j)$, 
which thus intersects  $(z_j=0)$ transversally at $p_j$. 

For any edge $e= (p_j \to p_\ell)$,  
we  can  define    
 a   graph transform map  $\cL_e$  on $\mathcal{H}_r$ 
by the property 
 that  $C_j(\cL_e(h))$ is the pull-back by $F$ of the image by $\theta$ of $C_\ell(h)$. 
For this, let $p'_i:= \theta(p_j)$. 
Then $\theta(C_j(h))$ is parametrized by $(z'_i(t),w'_i(t))= (t, t^{n_j} h(t) + P_j(t))$, and
we obtain an equation for $C_\ell(\cL_e(h))$ of the form
\[
\hw_j = (\hz_j^{m_j})^{n_j} h(\hz_j^{m_j})+ P_j(\hz_j^{m_j})
\]
To obtain the explicit form of $\cL_e(h)$ in terms of $h$,  it is then necessary 
to invert $(z_j, w_j)$ in terms of $(\hat{z}_j,\hat{w}_j)$. 
The formulas are not particularly convenient, but
we still can prove the following. 
\begin{prop}\label{prop:12}
For any $r>0$ small enough,   $\cL_e$ maps 
  $\mathcal{H}_r $ into itself, and it is  $1/2$-Lipshitz for the sup norm. 
Moreover, for any $h_1, h_2\in \mathcal{H}_r $ we have
\begin{equation}\label{eq:contNA}
\ord_t(\cL_e(h_1)- \cL_e(h_2)) \ge 1+ \ord_t(h_1- h_2)
\end{equation}
\end{prop}

\begin{proof}
The map $\cL_e$ is the composition of two maps  $\cL_e = T_2 \circ T_1$
where 
\[T_1(h) = t^{m_j n_j} h(t^{m_j})+ P_j(t^{m_j}),\] 
and $T_2 $ is defined by  the following condition: 
the curve parameterized by $(t, T_2(\gamma)(t))$ in the coordinates
$(z_j,w_j)$ is identical to the one defined by $(\tau, \gamma(\tau))$ in the
 coordinates
$(\hz_j,\hw_j)$.
Write $\hz_j = z_j A_j(z_j,w_j)$ with $A_j(0) \neq0$, 
and $w_j = B_j(\hz_j,\hw_j)= b_j \hw_j + O(\hz_j, \hw_j^2)$ with $b_j\neq0$. Then we have
$\tau = t A_j(t, h(t))$ and  $T_2(h)(t)= B_j(\tau, h(\tau))$,
so that 
\[T_2(h)(t) = B_j( t A_j(t, h(t)), h( t A_j(t, h(t))))~.\]

The map $T_1$ is easy to analyze, and it is strictly contracting both for the complex and the non-Archimedean
norm because $m_j n_j \ge 1$.  More precisely, on the disk of radius $r>0$,   we have that 
$|T_1(h_1) - T_1(h_2)|_\infty \le r |h_1 - h_2|_\infty$, and
$$\ord_t(T_1(h_1)- T_1(h_2)) \ge m_j n_j + m_j \ord_t(h_1- h_2).$$

On the other hand,  $T_2$ is $C$-Lipschitz with $C$ depending on the sup norm of the partial derivative 
of $A_j$ and $B_j$ (details are left to the reader) so that the first statement follows
when $Cr <1/2$. 
Observe also that  $\ord_t(A_j(t, h_1)- A_j(t, h_2))\ge  \ord_t(h_1-h_2)$
which implies  $\ord_t(T_2(h_1)- T_2(h_2))\ge  \ord_t(h_1-h_2)$. 
Combining this with $\ord_t(T_1(h_1)- T_1(h_2)) \ge 1 +  \ord_t(h_1- h_2)$, we get the estimate~\eqref{eq:contNA}, and 
the proof is complete. 
\end{proof}

\subsection{Conclusion in the multiplicity 1 case}\label{subs:conclusion_mult1}
As in the previous step, we consider the two models $X$ and $X'$ and the directed graph $\Gamma$
whose vertices are in bijection with the set of balls $\cB $. In order to clarify notation, 
we shall write $B_v$ for the ball in $\cB$ associated with a vertex $v$. Similarly, we write $p_v$
for the point in $X$ corresponding to $B_v$. 
Recall that $f_\diamond(B_v)$ is a ball in $\cB'$, and an edge $e$ joins $v$ to $w$
iff $B_w\subset f_\diamond(B_v)$.

Write $\mu_v := \mu_{\na}(B_v)$. Then $\sum_v \mu_v = 1$ since $\mu_{\na}$ is a probability measure,
and, as explained in \S~\ref{subs:coding},
$\mu_v$ is the $v$-entry   of the  Perron-Frobenius eigenvector of the transition matrix, associated to the eigenvalue $c$.  From this we get that 
$$\unsur{c}\,  \sum_{v'\colon  v\to v' \in E(\Gamma)} \mu_{v'} = \mu_v.$$

Pick any ${\underline{v}} =(v_0v_1\cdots)  \in \Sigma$, and set $e_i= (v_iv_{i+1})$ for every $i$.
 We then define the map $\cL_{\underline{v}}^{(n)}\colon \mathcal{H}_r  \to \mathcal{H}_r$ by  
\[
\cL_{\underline{v}}^{(n)} (h) =  
\cL_{e_0} \circ  \cL_{e_1}\circ \cdots  \circ \cL_{e_{n-1}}(h)~.
\]
Given any vertex $v\in \Gamma$, and any analytic map $h\in \mathcal{H}_r $,  recall that 
$C_v(h)$  is the  analytic curve defined  as the graph of $h$ in the coordinates 
$(z_v, w_v)$ at $p_v$, as  introduced in the previous subsection.  

\begin{lem}\label{lem:convergence}
For any ${\underline{v}} = (v_n)_{n\geq 0} \in \Sigma$, the 
 sequence $(\cL_{\underline{v}}^{(n)} (h))_{n\geq 0}$ converges in 
$\mathcal H_r$
to a limit function $h_{\underline{v}}$ independent of $h$, hence  the corresponding graphs 
$C_{v_0}(\cL_{\underline{v}}^{(n)} (h))$ converge in the sense of currents. 
Likewise, the sequence of germs $C_{v_0}(\cL_{\underline{v}}^{(n)} (h))$ converges in 
$B^{\an}(0,1)$ to $C_{v_0} (h_{\underline{v}})$
\end{lem}

\begin{proof}
This is a simple consequence of the contraction statement in Proposition~\ref{prop:12}. Indeed, since for any $m \geq n$
$$
\norm{\cL_{\underline{v}}^{(n)} (h) - \cL_{\underline{v}}^{(m)} (h)} \leq 2^{-n} \norm{h-\cL_{e_{n}} \circ \cdots  \circ \cL_{e_{m-1}}(h)} \leq 2^{1-n} r\text{,}
$$ we see that  
$(\cL_{\underline{v}}^{(n)} (h))_{n\geq 0}$ is a Cauchy sequence in $\mathcal H_r$. The independence 
of the limit with respect to  $h$ follows from the contraction as well. 

For the second statement, we argue similarly, this time by using the contraction property of the graph transform operator in the $t$-adic norm. 
\end{proof}

Consider now the positive closed $(1,1)$-current  
\[T_0 =  \sum_v \mu_v[\kappa(C_v(0))]
\] 
and for $n\geq0$,  set $T_n = \frac1{c^n} (f^{n})^*(T_0)$. 
By Corollary~\ref{cor:Pk}, the sequence of currents $(T_n)$ converges to ${\mathrm{d}}{\mathrm{d}}^c g$.  
For any vertex $v$, the structure of $f$ as explained in \S~\ref{subs:model} and~\ref{subs:graph_transform} shows that 
\[
(f^{n})^*[\kappa\lrpar{C_v(0)} ]
=
\sum_{{\underline{v}}\in \Sigma,  \, {\underline{v}}_n  =v }  \left[\kappa \left({C_{v_0}\left(\cL_{\underline{v}}^{(n)}(0) \right)} \right) \right]~,
\]
so that 
\[
T_n
=
\frac1{c^n} 
\sum_v \mu_v \sum_{{\underline{v}}\in \Sigma, \,  {\underline{v}}_n =v } 
 \left[\kappa \left({C_{v_0}\left(\cL_{\underline{v}}^{(n)}(0) \right)} \right) \right]~.
\]
If we denote by $\Sigma(n,{\underline{v}})$ the cylinder of sequences in $\Sigma$ coinciding with ${\underline{v}}$ on their first $n$ entries, by the formula~\eqref{eq:subshift}, this can be rewritten as 
\begin{align*}
T_n&=\frac1{c^n} 
\sum_v \mu_v \sum_{{\underline{v}}\in \Sigma,  \, {\underline{v}}_n  =v }  \mu\lrpar{\Sigma(n,{\underline{v}})} 
 \left[\kappa \left({C_{v_0}\left(\cL_{\underline{v}}^{(n)}(0) \right)} \right) \right]
 \\&= \int_\Sigma 
  \left[\kappa \left({C_{v_0}\left(\cL_{\underline{v}}^{(n)}(0) \right)} \right) \right]
 \,{\mathrm{d}}\mu_{\Sigma}\lrpar{\underline{v}}.
 \end{align*}
By Lemma~\ref{lem:convergence} above, $[C_{v_0}(\cL_{\underline{v}}^{(n)}(0))]$ converges to 
$\left[C_{v_0}(h_{\underline{v}})\right]$, so by the dominated convergence theorem we 
conclude that 
\begin{equation}\label{eq:convergence1}
{\mathrm{d}}{\mathrm{d}}^cg  = \lim_{n\to\infty }T_n = 
\int_\Sigma \left[\kappa\lrpar{C_{v_0}\left(h_{\underline{v}}\right)}\right] \,{\mathrm{d}}\mu_{\Sigma}\lrpar{\underline{v}}.
\end{equation}

On the other hand, again by Lemma~\ref{lem:convergence}, 
$C_{v_0}(\cL_{\underline{v}}^{(n)}(0))$ converges to $C_{v_0}\lrpar{h_{\underline{v}}}$ in 
$B^{\an}(0, 1)$.
Denote by $\widehat F: X\dasharrow X$   the lift of $f$ to $X$. 
 By construction, for every $n\geq 0$,  the curve 
$\widehat F^n\lrpar{C_{v_0}(h_{\underline{v}})}$  
intersects the exceptional divisor
$\kappa^{-1}(0)$ in $X$ at the point $p_{v_n}$,
therefore  $\kappa(C_{v_0}(h_{\underline{v}}))$ belongs to $\cK$
and its itinerary is   given by $\underline{v}$. Therefore, $\kappa(C_{v_0}(h_{\underline{v}}) )= C\lrpar{\mathsf{c}\inv\lrpar{\underline{v}}}$, where $C(\cdot)$ is as in Theorem~\ref{thm:structure}
 and 
$\mathsf{c} \colon \cK \to \Sigma$ is the conjugacy from \S~\ref{subs:coding}. 
In particular,  $C(x)$ is a closed analytic curve in a fixed neighborhood of the origin, and a graph over the $z$ coordinate. Thus,  from~\eqref{eq:convergence1}, we get 
\begin{equation}\label{eq:integral_rep}
{\mathrm{d}}{\mathrm{d}}^cg = 
\int_\Sigma \left[C\lrpar{\mathsf{c}\inv\lrpar{\underline{v}}}\right] \,{\mathrm{d}}\mu_{\Sigma}\lrpar{\underline{v}}
 = \int_{\cK} [C(x)] \,{\mathrm{d}}\mu_{\na}(x). 
 \end{equation}
 To complete the proof, it remains to verify the assertions on the laminar structure of $\cW$ and $T$. 
 For this, we observe that 
  if $\underline{v}$ and $\underline{v'}$ are distinct itineraries, the corresponding 
  curves $C_{v_0}\left(h_{\underline{v}}\right)$  and $C_{v'_0}\left(h_{\underline{v'}}\right)$, 
  constructed above,  get separated after finitely many iterations, so they are disjoint in $X$. Pushing forward under $\kappa$, we see that their images intersect only at the origin so the curves $C(x)$ for $x\in \cK$ form a lamination outside the origin.    
 From the integral representation~\eqref{eq:integral_rep}, the support of this lamination is 
   $\supp(T) = \cW$. Since $\cW$ is a closed pluripolar set, this lamination 
  is transversally polar, hence transversally 
 totally disconnected, and it has no isolated leaf for otherwise $T$ would give positive mass to 
 a curve.

\subsection{Uniform laminarity of the current $T(f)$}\label{subs:laminarity}
To complete the proof of the theorem, we push forward the previous picture by 
the base change $\beta_k$ from \S~\ref{subs:base_change}, and need to 
verify the assertions of the theorem. We resume the notation from \S~\ref{subs:base_change}.

Since $g(z,w) = \lim_{n\to \infty} \unsur{c^n}\log\abs{w_n}$ we see that 
 $g(f)\circ \beta_k = g(f_k)$, so  
 $\beta_k^* {\mathrm{d}}{\mathrm{d}}^c g(f) = {\mathrm{d}}{\mathrm{d}}^c g(f_k)$, hence 
 $(\beta_k)_*T(f_k) = (\beta_k)_*\beta_k^* T(f) = k T(f)$ because $\beta_k$ has topological degree $k$. 

From this we see that  the current $T(f)$ is the image 
of a uniformly laminar current outside the origin under $\beta_k$. While this still implies that  $\mathrm{d}\mathrm{d}^cg$ is laminar (because $\cW$ is pluripolar), 
to get uniform laminarity we need to argue that the images of the   curves $C(x)$ do not intersect each other outside the origin also in this case.
Fix an open ball $U$ such that $\set{z=0}\cap   \overline U = \emptyset$. Since $\beta_k$ is a covering outside $\set{z=0}$, the preimage of $U$ is a union of $k$ disjoint open sets, so that $T(f)$ is actually the sum of $k$ uniformly laminar currents $S_i$ in $U$, and $\cW(f)$ is the union of the corresponding laminations. 
 
Assume that two of these currents (say $S_1$ and $S_2$)   have non-trivially intersecting plaques, then by \cite[Lemma 6.4]{bls} some of these intersections are transverse, and by reducing $U$ further we may assume that in $U$, the plaques are graphs of definite size over two different directions, and every plaque of $S_1$ intersects every plaque of $S_2$.

We will argue that the fact (from Corollary~\ref{cor:Pk})
 that $T(f)$ is the limit of pull-backs of smooth curves makes such a local picture impossible. 
For this, we first work with $f_k$. By construction of the model, there is a neighborhood of the points $p_j$ in $X$ that is disjoint from the (proper transform of) the critical set. By pushing under $\kappa$,  
there is a 
$f_k\inv$-invariant ``neighborhood'' (not open at the origin) $\mathcal N(f_k)$ 
of $\cW(f_k)$ disjoint from $\CritB(f_k)$, and, since $\CritB(f_k) = (\beta_{k, \diamond})\inv(\CritB(f))$
we get a corresponding neighborhood $\mathcal N(f) = \beta_k(\mathcal N(f_k)) $ for $f$. 

From our choice of $r$ in \S~\ref{subs:graph_transform}, 
 a curve of the form $C_j = C_j(h)$ passing by some $p_j$ is disjoint from the critical set. Let 
 $\tilde \Gamma = \kappa (C_j(h))$ and $\Gamma   = \beta_k(\tilde \Gamma)$, so that $\Gamma$ is contained in $\mathcal N(f)$. If necessary, we work 
 in a smaller neighborhood of 0 so that $\Gamma$ is smooth outside 0, and we conclude that all 
 curves $f^{-n}(\Gamma)$ are smooth outside the origin. Furthermore $f^{-n}(\Gamma)\cap U$ is made up of plaques of uniformly bounded geometry (since they descend from $X$ via $\beta_k$), so any transverse intersection between $S_1$ and $S_2$ 
 at the limit must originate  from a self-intersection of $f^{-n}(\Gamma)$. This contradiction 
 completes the proof that $\cW(f)$ is a lamination, which must have Cantor transverse structure as before,  and that $T(f)$ is a uniformly laminar current.  
 \qed
 
 \subsection{Integral representation of the current $T(f)$}
 It remains to justify the integral representation of $T(f)$. 
 To simplify notation we write $\beta=\beta_k$.  For a probability measure $\mu$ on 
 $\A^{1, \an}_{\C\laurent{z}}$, we put $\beta^* \mu = k \beta_\diamond^* \mu$, which is a probability measure (see \S~\ref{subs:invariant_measure}). 
 
 \begin{lem}\label{lem:resp}
 Let $\mu$ be a positive measure with finite support on $B^{\an}(0,1)$ on a set of points corresponding to 
 germs of irreducible analytic curves. 
 Then we have
 \begin{equation}\label{eq:resp}
 \beta_*\left(
 \int_{u\in B^{\an}(0,1)} \frac{[C(u)]}{m(u)} d  \beta^*\mu(u)
\right)= k\,
 \int_{v\in B^{\an}(0,1)} \frac{[C(v)]}{m(v)} d\mu(v)
\end{equation}
\end{lem}

Admitting this result for the moment, we complete the proof of the theorem. 
 Since  $\beta_{\diamond}$ 
 semi-conjugates   $f_{k,\diamond}$ to $f_\diamond$,  
from  the equidistribution Theorem~\ref{thm:ergo-meas}, 
and the definition of $\beta^*$
 we deduce that 
$\beta^*\mu_{\na}(f) =   \mu_{\na}(f_k)$. 
In addition, 
$\mu_{\na}$ is a limit of atomic measures as in the Lemma, and by the continuity of  the operators
$\beta^*$ and $\beta_*$ for the respective weak topologies on $\A^{1, \an}_{\C\laurent{z}}$ and the space of 
(1,1) currents, we infer that~\eqref{eq:resp}
  also holds for the measure $\mu=\mu_{\na}$. 
Recall that by construction we have that $m(C(v))=1$ for all curves $C(v)$ in the support of $\cK(f_k)$, and using the relation  
 $\beta_* T(f_k) = kT(f)$ we obtain 
\begin{align*}
T(f)
=
\frac1k \, \beta_* T(f_k)
&=
\frac1k \, \beta_*\left(
 \int_{u\in B^{\an}(0,1)} [C(u)] d \beta^*(\mu_{\na}(f))(u)
\right)
\\
&=
 \int_{v\in B^{\an}(0,1)} \frac{[C(v)]}{m(v)} d\mu_{\na}(f)(v)
\end{align*}
as was to be shown. \qed

\begin{proof}[Proof of Lemma~\eqref{lem:resp}]
We can assume that $\mu$ is a Dirac mass at the point $x_{C}$ where $C$ is an irreducible analytic germ. 
In terms of currents, we may write
$\beta^*[C] = \sum_i k_i [C_i]$ where $C_i$ are the irreducible components of $\beta^{-1}(C)$,
and $k_i\in \N^*$.
Observe that if $e_i\in \N^*$ denotes the local topological degree of the restriction $\beta\colon C_i \to C$, then we have
\[
\beta_* \beta^* [C] = \sum_i  k_i \beta_*[C_i]  =  \sum_i  k_i e_i [C]  = k[C]
\]
since the topological degree of $\beta$ is equal to $k$. 

Pick any equation of $C$  of the form $C = \{P(w) =0\}$, where $P$ is a monic polynomial of degree $m(C)$ in $w$
with coefficients in the ring of holomorphic functions in $z$. 
Then, with notation as in \S~\ref{subs:invariant_measure}, by~\eqref{eq:fundid2}, we get 
 $\Delta \log_P  = m(C)\delta_{x_C}$. 
We can write $P\circ \beta = \prod_i P_i^{k_i}$ where $P_i$ are equations of $C_i$, so that 
\begin{align*}
\beta^*(m(C)\delta_{x_C}) &= k \beta_\diamond^* (m(C)\delta_{x_C})  
 = k \Delta \lrpar{\log_P\circ \beta_\diamond} \\&= \Delta\lrpar{x\mapsto \log\abs{P(\beta(x))}} = 
\sum_i k_i\,\Delta \log_{P_i}
=
\sum_i k_i m(C_i) \delta_{x_{C_i}} , 
\end{align*}
therefore 
$$
\int_{u\in B^{\an}(0,1)}\frac{[C(u)]}{m(u)}  d \beta^*\delta_{x_C} (u) 
 = \unsur{m(C)} \lrpar{\sum_i k_i[C_i]},
$$
and finally 
\begin{align*}
\beta_*\left(
 \int_{u\in B^{\an}(0,1)}\frac{[C(u)]}{m(u)}  d \beta^*\delta_{x_C} (u) 
\right)=
\beta_*\left(
  \frac1{m(C)} \sum_i k_i [C_i]
\right)
= \frac{k}{m(C)}[C]
\end{align*}
which concludes the proof. 
\end{proof}

\medskip


\bibliographystyle{alpha}

\bibliography{biblio}

\end{document}